\documentclass[10pt]{article}
\RequirePackage[OT1]{fontenc}
\RequirePackage{amsthm,amsmath}
\RequirePackage[numbers]{natbib}
\RequirePackage[colorlinks,citecolor=blue,urlcolor=blue]{hyperref}

\usepackage[utf8]{inputenc}
\usepackage{times}
\usepackage{fullpage}

\usepackage{amsfonts,amssymb}	
\usepackage{mathrsfs}	
\usepackage{dsfont}
\usepackage{authblk}

\usepackage{bbm}
\usepackage{color}
\usepackage[table]{xcolor}
\usepackage{euscript,graphicx}
\usepackage{float}
\usepackage{subfigure}
\usepackage{xspace}
\usepackage{booktabs}
\usepackage[font=small,labelfont=bf]{caption}
\usepackage{appendix}

\newcommand{\blue}[1]{{\textcolor{blue}{#1}\color{black}\xspace}}

\newtheorem{theorem}{Theorem}[section]
\newtheorem{cor}[theorem]{Corollary}
\newtheorem{lem}[theorem]{Lemma}
\newtheorem{prop}[theorem]{Proposition}
\newtheorem{rem}[theorem]{Remark}
\newtheorem{defn}[theorem]{Definition}

\newtheorem{hyp}{\bf Hypothesis}
\makeatletter
\newcounter{subhyp}
\let\savedc@hyp\c@hyp
\newenvironment{subhyp}
 {\setcounter{subhyp}{0}%
  \stepcounter{hyp}%
  \edef\saved@hyp{\thehyp}% 
  \let\c@hyp\c@subhyp     %
  \renewcommand{\thehyp}{H\saved@hyp}%
 }

\newtheorem{hypbis}{\bf Hypothesis}
\makeatletter
\newcounter{subhypbis}
\let\savedc@hypbis\c@hypbis
\newenvironment{subhypbis}
 {\setcounter{subhypbis}{0}%
  \stepcounter{hypbis}%
  \edef\saved@hypbis{\thehypbis}% 
  \let\c@hypbis\c@subhypbis     % 
  \renewcommand{\thehypbis}{H\saved@hypbis}%
 }

\newtheorem{primehyp}{\bf Hypothesis}
\makeatletter
\newcounter{subprimehyp}
\let\savedc@primehyp\c@primehyp
\newenvironment{subprimehyp}
 {\setcounter{subprimehyp}{1}%
  \stepcounter{primehyp}%
  \edef\saved@primehyp{\theprimehyp}% 
  \let\c@primehyp\c@subprimehyp     % 
  \renewcommand{\theprimehyp}{H\saved@primehyp{'}}%
 }

\numberwithin{equation}{section}
\newcommand{\brac}[1]{\big[#1\big]}

\def\RR{\mathbb{R}}
\def\QQ{\mathbb{Q}}
\def\Ff{\mathcal{F}}

\def\qs{X_{s}^x}

\def\EE{\mathbb{E}}
\def\PP{\mathbb{P}}

\def\qn1{{X}_{t_{n+1}}}

\newcommand{\Dt}{\Delta t}
\newcommand{\dt}{\delta(t)}
\newcommand{\ds}{\delta(s)}
\newcommand{\Ld}{\overline{\mathcal{L}}}

\newcommand{\tabhead}[1]{\textbf{#1}}

\begin{document}

%--------------------------------------------------------------------------
\title{On the weak convergence rate of an exponential Euler scheme for SDEs governed  by coefficients with superlinear growth
}

\author[1]{Mireille Bossy\footnote{mireille.bossy@inria.fr}}

\author[2]{Jean-Fran\c{c}ois Jabir\footnote{ jjabir@hse.ru; The second author is supported by the Russian Academic Excellence Project '5-100'.}}

\author[3]{Kerlyns Mart\'inez\footnote{kerlynsmartinez@gmail.com; The third author acknowledges the support of the CONICYT National Doctoral scholarship N\textordmasculine 2116094.}}

\affil[1]{Universit{\'e} C{\^o}te d'Azur, Inria, France}

\affil[2]{School of Mathematics, University of Edinburgh, Scotland and 
National Research University Higher School of Economics, Moscow, Russia}

\affil[3]{Universidad de Valpara\'iso, Doctorate in Mathematics of Valpara\'iso, Chile}

\maketitle

\begin{abstract}
We consider the problem of the approximation of the solution of a one-dimensional SDE with non-globally Lipschitz drift and diffusion coefficients behaving as $x^\alpha$, with $\alpha>1$.
We propose an (semi-explicit) exponential-Euler scheme and study its convergence through its weak approximation error. To this aim,
we analyze the $C^{1,4}$ regularity of the solution of the associated backward Kolmogorov PDE using its Feynman-Kac representation and the flow derivative of the involved processes. From this, under some suitable hypotheses on the parameters of the model ensuring the control of its positive moments,  we recover a  rate of weak convergence of order one for the proposed exponential Euler scheme. Finally, numerical experiments are shown in order  to support and complement our theoretical result.
\medskip

\it{\paragraph{\it Keywords.}{Stochastic differential equation}, {Numerical scheme}, {Polynomial coefficients}, {Weak convergence}, {Rate of convergence}; 
}

\end{abstract}

\section{Introduction}\label{sec:intro}
%--------------------------------------------------------------------------

Within the extensive literature on the numerical analysis of  approximation schemes for Brownian-driven stochastic differential equations  with non-Lipschitz coefficients, existing convergence results mainly deal separately with the singularity hypothesis on the drift coefficient or on the diffusion coefficient. More rarely, the Lipschitz property is dropped for both coefficients.  In this paper, we  propose a  numerical scheme for  one-dimensional stochastic differential equations (SDEs for short) having non-globally Lipschitz, polynomial  drift and diffusion coefficients, and we analyze its  convergence  for the weak error. In this context, we present the first direct proof of  the weak convergence with rate one, accompanied by an  extendable methodology to analyze the $C^{1,4}$ regularity of the Feynman-Kac representation involving the exact process.

More precisely, we are interested in the numerical approximation of the solution to the following class of SDEs
\begin{equation}\label{eq:IntroSDE}
dX_t = b(X_t)dt + \sigma X_t^\alpha dW_t,~X_0=x>0,
\end{equation}
where $(W_t;0\leq t\leq T)$ is a standard Brownian motion on the probability space $(\Omega,\Ff,\PP)$ equipped with its natural filtration $(\Ff_t;\,0\leq t\leq T)$, and the power $\alpha$, characterizing the diffusion,  is assumed strictly  greater than one. The drift  $b:[0,+\infty)\rightarrow\RR$  is a locally Lipschitz continuous function allowing polynomial dependence in the Lipschitz constant (see Definition \ref{def:Lipschitz} below for a precise statement), and with a polynomial growth bound of the form:
\begin{equation}\label{eq:IntroHypo}
b(x)\leq B_1 x - B_2 x^{2\alpha-1}+b(0),~\forall x\in\RR^+,
\end{equation}
for some constants $B_1,B_2,b(0)\geq 0$.
Since exponent $\alpha$ can be non-integer, some particular hypotheses under which the SDE \eqref{eq:IntroSDE} has a unique positive solution have to be made (see Proposition~\ref{prop:XMoments}).

Convergence results in this particular setting of non-Lipschitz coefficients rarely  deal directly on the weak error analysis.
In this particular setting, Gy{\"o}ngy~\cite{Gyongy} obtained pathwise almost surely convergence, with a  convergence rate's order of at most $\tfrac{1}{4}$, for the classical  Euler-Maruyama scheme applied to SDEs with locally Lipschitz continuous drift and diffusion coefficients satisfying some Lyapunov condition. Such result immediately implies weak convergence for continuous bounded test functions but not $L^p$-strong convergence. Similarly, Higham, Mao and Stuart~\cite{Higham} established $L^2$-strong convergence of the classical Euler-Maruyama scheme for locally Lipschitz coefficients but  assuming a priori the control of some $p$\,th-moments ($p>2$) of the continuous solution of the SDE and of its approximation. A  rate of strong convergence of order $\tfrac{1}{2}$ was also established for the  time-implicit split-step backward Euler-Maruyama scheme, when the diffusion  is globally Lipschitz, and the drift  satisfies a one-sided Lipschitz condition and locally Lipschitz condition. Staying in the framework of the Euler-Maruyama scheme, Yan~\cite{Yan} obtained the weak convergence  with  diffusion and drift coefficients continuous only almost everywhere and having at most linear growth.

With superlinear growth coefficients, classical Euler-Maruyama scheme may present some degenerated behavior. Hutzenthaler, Jentzen and Kloeden~\cite{Kloeden} established the $L^p$-strong divergence, for $p\in[1,+\infty)$, related to the Euler-Maruyama scheme for SDEs with both drift and diffusion satisfying some superlinear growth condition.
In particular, the authors obtained the divergence of the  moments of the Euler approximation. Later in \cite{Kloeden2}, the authors proposed a time-explicit tamed-Euler scheme to overcome this \emph{divergence} problem of the Euler approximation, based on renormalized-increments to the scheme.
Recently Hutzenthaler and  Jentzen \cite{Jentzen-b} proved  the $\tfrac{1}{2}$ rate of the $L^p$-strong convergence for the tamed-Euler scheme for a family of SDE  that includes some locally Lipschitz cases for both  continuous drift  and diffusion coefficients.

In the same vein of explicit in time alternative scheme to the $L^p$-strongly divergent Euler-Maruyama scheme, Sabanis \cite{Sabanis} obtained the $L^p$-strong convergence for a scheme with renormalized coefficients under some superlinear growth condition, and recovered  the $\tfrac{1}{2}$-$L^p$-strong convergence rate under Lispchitz diffusion and one-sided global Lispchitz drift.

Other time-explicit numerical schemes have been proposed over the years to solve the approximation problem of SDEs with locally Lipschitz continuous coefficients. For instance,  Lamba, Mattingly and Stuart~\cite{Lamba} proposed an adaptive Euler algorithm based on the control of the drift coefficient, and proved the $L^2$-strong convergence assuming the control of some moments of the solution and of its approximation. Chassagneux, Jacquier and Mihaylov~\cite{Cha16}  considered the case of globally Lipschitz diffusion and locally-Lipschitz drift function satisfying a one-sided Lipschitz condition and  proposed a modified explicit Euler numerical scheme, for which, under suitable assumptions on the control of some moments of the solution, an $L^2$-strong convergence with explicit rate is proven.

Fewer results dealing with weak convergence are available. Milstein and Tretyakov~\cite{MilsteinTr} established a  weak convergence result  for a class of SDEs with non-globally Lipschitz coefficients, based on existing schemes with known rate of weak convergence for Lipschitz and smooth coefficients, and  on the rejection of the approximated trajectories that go out a sphere of given radius. But the  relation between the level of error, the radius of the rejection sphere and the time step threshold to be used in order  to observe the convergence is not explicit, making the algorithm difficult to use in practice.

In this paper, we propose a new scheme for SDEs with smooth coefficients under some superlinear growth condition.  The scheme is designed to ease  the upper bound control of some moments of the approximated process  and we prove the optimal convergence rate of order one for the weak error.
The convergence analysis extends the methodology introduced in Bossy and Diop~\cite{BD15} to establish the regularity of the associated backward Kolmogorov PDE.

\subsubsection*{Our motivating problem}

Our interest for the numerical approximation  of \eqref{eq:IntroSDE} was  initially motivated by the simulation/calibration  problem for the {\it instantaneous turbulent kinetic energy model} issued from the Lagrangian description of a non inertial particle dynamics within a turbulent fluid flow (see \cite[Chapter 1]{kerlyns}). Such model can be described by a SDE  having the prototype form:
\begin{equation}\label{eq:Generalized CIR}
dX_t = -B\;X_t^{2\alpha-1}dt+\sigma X_t^\alpha dW_t, \;X_0=x>0,
\end{equation}
where $\alpha>1$. To your knowledge, no weak convergence rate for this model are available. Only strong convergence results are proposed.
Equation \eqref{eq:Generalized CIR} is a particular case of \eqref{eq:IntroSDE} and can be seen as a generalized Constant Elasticity of Variance (CEV) model  (see e.g. Delbaen and Sirakawa~\cite{Delbaen}).   In particular, the transformation $r_t=\tfrac{X_t^{2(1-\alpha)}}{4\sigma^2(\alpha-1)^2}$ applied to the solution of \eqref{eq:Generalized CIR} produces the so-called CIR process (Cox, Ingersoll and Ross~\cite{CIR}) classically used for modeling short interest rate dynamics, and for which various schemes have been considered over the years. For the $L^p$-strong convergence of some  proposed explicit schemes for CEV models,  we refer to Bossy and Olivero \cite{B17} and the reference therein; for implicit proposed schemes, we refer to Dereich, Neuenkirch and Szpruch~\cite{DNS11}, Alfonsi~\cite{Alfonsi} and the references therein.

Alternatively, the transformation $Y_t=\tfrac{X_t^{(1-\alpha)}}{\sigma(\alpha-1)}$ produces a Bessel process for which we can use an Explicit Projected Euler scheme proposed in~\cite{Cha16}, obtaining a strong rate of convergence of order $\frac{1}{6}$ provided that we control up to the $4(\alpha-1)$-th moments of the process $(X_t;0\leq t\leq T)$ (or higher rate of convergence by controlling higher moments).

\subsubsection*{Exponential scheme}

The keyword {\it exponential scheme} refers to generic semi-linear integration methods and is of main importance in numerical analysis.
Methods for ODEs proposing integration schemes based on the semi-linear integration of equations are classics (see e.g. Pope \cite{Pope-da-63}, Hochbruck and Ostermann \cite{HocOst-10} and the references therein).  Extend this methodology for SDEs is straightforward (particularly in dimension one where affine diffusions allow exact scheme), but establishing  the weak rate of convergence results  in the context of non-globally Lipschitz coefficients is much more demanding.
With the same appellation,  semi-linear integration methods are proposed for PDEs or SPDEs and concern schemes based on a mild formulation of the underlying equations  (see e.g. Beccari {\it et al} \cite{Jentzen-a} for SPDE problems with superlinear coefficients).

We would like to stress out that we were  looking for a scheme that, potentially applied to prototype model \eqref{eq:Generalized CIR}, allows weak convergence rate of order one to set up an efficient calibration method for the model.  Motivated by this problem  for the model \eqref{eq:Generalized CIR}, the requirement of  stability condition on the moments   brings us to the variant scheme \eqref{eq:IntroNumericScheme} below,  as a remedy for the  divergence problem of the Euler-Maruyama scheme, and an alternative to the tamed-Euler scheme (for which the weak  convergence rate of order one is not established).

The proposed numerical approximation, which will be referred, from now on, to as the  ex\-po\-nen\-tial-Euler  scheme (exp-ES, for short), originates  from rewriting the SDE \eqref{eq:IntroSDE} into
\[
dX_t = X_t \big( \tfrac{b(X_t)}{X_t}dt + \sigma {X_t^{\alpha-1}} dW_t\big),\quad X_0=x>0,
\]
and semi-linear integration produces, for an  homogeneous $N$-partition of the time interval $[0,T]$   with  time-step $\Delta t=t_{n+1}-t_n$, the approximation algorithm:
\begin{equation}\label{eq:IntroNumericScheme}
\overline{X}_{t_{n+1}}=\overline{X}_{t_n}\exp\{\sigma\overline{X}_{t_n}^{\alpha-1}\;(W_{t_{n+1}}-W_{t_n})+
 \big(\tfrac{b(\overline{X}_{t_n})-b^+(0)}{\overline{X}_{t_n}}-\frac{\sigma^2}2 \overline{X}_{t_n}^{2(\alpha-1)}\big)\Delta t\}+ b^+(0)\Delta t,
\end{equation}
(where $b^+(0)$ stands for $b(0)\vee 0$) that preserves the positiveness of the solution.
We refer the reader to Section \ref{weakerror}
for a detailed construction of  \eqref{eq:IntroNumericScheme}.

The exponential Euler scheme \eqref{eq:IntroNumericScheme} can  be applied to large family of SDEs with non-globally Lipschitz coefficients, having  strictly positive solution. The range of possible applications of our results includes some meaningful financial models such as the generalized CEV model, the non-linear mean reversion model (see e.g., Ait-Sahalia~\cite{Ait96}, Higham et al.~\cite{Spr}) and  the Chan-Karolyi-Longstaff-Sanders model~\cite{Chan} among others.

As it will be established later on, a main advantage of the exp-ES scheme is that it  preserves the control of the moments of the continuous model, assuming a superlinear growth condition on the drift coefficient (see  Proposition \ref{prop:Weak rate}  and Lemma \ref{leq:Schememoments}).

\subsubsection*{Weak convergence and $C^{1,4}$-regularity of the  Kolmogorov PDE associated to \eqref{eq:IntroSDE} } 

Our main result, stated in  Proposition \ref{prop:Weak rate}, exhibits an optimal theoretical rate of convergence of order one under hypotheses that are introduced in Section \ref{weakerror}, and for bounded $C^4$ test functions.

Although some space of improvement are identified, the hypotheses  in Section \ref{weakerror} are stated in order to balance the control moments of the exact and approximated processes  with the moments and exponential moments required  for the flow-derivative process used to establish the regularity of the Feynman-Kac formula. Indeed, the key point of the convergence rate analysis  is to estimate the regularity of the solution to the backward Kolmogorov PDE associated with the representation $\EE[f(X_T^x)]$, where $(X_t^x;0\leq t\leq T)$ denotes the flow of diffeomorphisms with initial condition $x>0$.

The technique presented in this paper for the analysis of the Kolmogorov PDE can be derived for a larger family of SDEs. Adapted from \cite{BD15} which was dealing with the particular situation where $\tfrac{1}{2} < \alpha <1$, this methodology allows to control the  successive derivatives of the Feynman-Kac  representation up to the order four,  by  bypassing the difficulty of deriving the flow process more than one time, through a change of measure technique  (see Sections \ref{sec:sketch} and \ref{sec:changemeasure} for dedicated results and details on this main point).

The paper is organized as follows. Conditions for the well-posedness of the generic SDE \eqref{eq:IntroSDE} as well as  on the finiteness of the positive, negative and exponential moments of its solution are stated in  Section \ref{sec:Analytical}.
In Section \ref{weakerror} we construct the exponential-Euler scheme
\eqref{eq:IntroNumericScheme} and we present our main  Proposition \ref{prop:Weak rate} on the weak rate of convergence when applied to model \eqref{eq:IntroSDE}  and when applied to model  \eqref{eq:Generalized CIR} (Corollary \ref{prop:Weak rateCor}).
Section \ref{sec:Numeric} presents some numerical experiments in order to show the effectiveness of the theoretical rate of convergence of the proposed exponential Euler scheme. We also compare exp-ES with the classical Symmetrized Euler scheme, the Symmetrized Milstein scheme, and the (stopped)-tamed scheme. Section \ref{sec:CauchyP} is devoted to the analysis  of the regularity of the backward Kolmogorov PDE (Proposition  \ref{prop:KolPDE}) and Section \ref{sec:Proofs}  presents the proof of the weak error estimate.

For additional comments on the presented results and proofs,   and complements on  numerical experiments related to this work, we refer to \cite[Chapters 1 and 2]{kerlyns}.

\subsection{Notation}\label{subsec:NotDef}
Throughout this paper, $T>0$ will refer to an arbitrary finite time horizon,  $C$ will denote a positive constant, possibly depending on the parameters of the dynamic, which may change from line to line. Any process $(Z_t, t\in [0,T])$ will be simply denoted $Z$. For any measurable function $f$, $x\mapsto f^+(x)$ denotes its positive part. 

For any $a,b\in\mathbb{R}$, $a\vee b$ and $a\wedge b$ denote respectively the maximum and minimum between $a$ and $b$. Given the non-negative discrete time-step parameter $\Delta t$, we set $\eta(t) =      \Delta t  \lfloor \frac{t}{\Delta t} \rfloor$, and $\dt = t - \eta(t)$. In order to shorten  the writing of some expressions, we will use $\EE^\beta[Z]$ for $\left(\EE[Z]\right)^\beta$ and $f^{(k)}$ for the $k$-th derivative (whenever $k>1$) of a function $f:\mathbb{R}\rightarrow\mathbb{R}$.

We introduce the notion of locally Lipschitz continuity property used in this paper from the
formalism previously used in \cite{Higham} and \cite{Cha16}. The following definition specifies power-dependencies involved in the local Lipschitz factor.
\begin{defn}\label{def:Lipschitz}
Let $f$ be a real-valued function and $\textsl{dom}(f)$ denotes  its definition domain. 
We say that $f$ is $(\overline{\gamma}, \underline{\gamma})$-locally Lipschitz if there exists a non-negative constant $C$, $\overline{\gamma}$ and $\underline{\gamma}$ such that
\begin{equation}\label{eq:defLocallyLips}
|f(x)-f(y)|\leq C\left(1+|x|^{\overline{\gamma}}+|y|^{\overline{\gamma}}+|x|^{-\underline{\gamma}}+|y|^{-\underline{\gamma}}\right)|x-y|,~\forall x,y\in \textsl{dom$(f)$-\{$0$\}}.
\end{equation}
When $\underline{\gamma}=0$, $f$ is said to be $\overline{\gamma}$-locally Lipschitz continuous and
\begin{equation}\label{eq:defLocallyLips1}
|f(x)-f(y)|\leq C\left(1+|x|^{\gamma}+|y|^{{\gamma}}\right)|x-y|,~\forall x,y\in \textsl{dom}(f).
\end{equation}
\end{defn}

With this definition, a Lipschitz function is $0$-locally Lipschitz, and it is included in the set of all $\gamma$-locally Lipschitz functions, $\gamma\geq 0$.

The following lemma formalize the straightforward link between the locally Lipschitz property of a function and its derivative.
\begin{lem}\label{lem:LocallyLipschitz}
Let $f$ be a real-valued function, continuously differentiable, with $f'$ being $(\overline{\beta},\underline{\beta})$-locally Lipschitz continuous in the sense of Definition \ref{def:Lipschitz}. Then, $f$ is $(\overline{\alpha},\underline{\alpha})$-locally Lipschitz continuous with $\overline{\alpha}\leq\overline{\beta} +1$, and  $\underline{\alpha}\leq \underline{\beta}$.
\end{lem}

\section{Strong wellposedness for the solution to SDE \eqref{eq:IntroSDE}}\label{sec:Analytical}

Some sufficient conditions ensuring the strong well-posedness and control of moments of the solution to \eqref{eq:IntroSDE} are now exhibited.
Hereafter we assume the following hypotheses on the SDE~\eqref{eq:IntroSDE}:

\begin{subhyp}
\begin{hyp}\label{C2:H1}
$\alpha>1$, $\sigma>0$, and the (deterministic) initial condition $x>0$.
\end{hyp}
\end{subhyp}
\begin{subhyp}
\begin{hyp}\label{C2:H_loclip}
The drift $b$ is $2(\alpha-1)$-locally Lipschitz continuous (in the sense of Definition \ref{def:Lipschitz}) and $b(0)\geq0$.
\end{hyp}
\end{subhyp}
\begin{subhyp}
\begin{hyp}\label{C2:H_poly}
There exist some finite constants $B_i\geq0$, with $i=1,2$,  such that, for all $x\geq0$,
\begin{equation*}
b(x)\leq B_1 x -B_2 x^{2\alpha-1}+b(0).
\end{equation*}
\end{hyp}
\end{subhyp}

The proofs of the following proposition and lemmas are technical and by itself not directly in relation with the convergence analysis of our scheme. They are given in Appendix \ref{sec:appendix:proof_wellposedness}.
\begin{prop}\label{prop:XMoments}
Assume \ref{C2:H1}, \ref{C2:H_loclip} and \ref{C2:H_poly}. Then there exists a unique (strictly) positive strong  solution $X$  to the SDE \eqref{eq:IntroSDE}. In  addition, for all exponent $p$ such that $0\leq 2p\leq 1+\tfrac{2B_2}{\sigma^2}$, we have
\begin{align}\label{eq:moments_cont}
\sup_{t\in[0,T]}\EE\big[X_t^{2p}\big]\leq~C_p(1+ x^{2p}),
\end{align}
where the  non-negative constant $C_p$ may depend on $p$, but does not depend on $x$.
\end{prop}

\begin{rem}\label{rem:exponentialPrototype}

In all the results of this section, Assumption \ref{C2:H_loclip} can be weaken to a $\gamma$-locally Lipschitz property on $b$, provided that $\gamma\geq 2(\alpha-1).$ The limit case $\gamma=2(\alpha-1)$ is considered in order to simplify the computation.
\smallskip
The effect of the non-negative constants $B_1$ and $b(0)$ on the SDE \eqref{eq:IntroSDE} is to push the solution away from zero, while the constant $B_2$ enables to counter the growth of the solution due to the diffusion term. The assumption on $B_2 \geq 0$ can so be weaken into $B_2 \geq - \tfrac{\sigma^2}{2}$ in order to preserve the well-posedness of SDE \eqref{eq:IntroSDE}. However, in that case, only the moments of order strictly less than one are controlled.
\end{rem}

We complement the preceding proposition by the control of the negative and exponential moments.
\begin{lem}\label{lem:Negative}
Assume \ref{C2:H1}, \ref{C2:H_loclip}  and \ref{C2:H_poly}. Then, for all $q>0$, the solution $X$ to  \eqref{eq:IntroSDE} satisfies
\[\sup_{t\in[0,T]}\EE[X_t^{-q}]\leq C_{q}(1+x^{-q}),\]
where  the non-negative constant $C_{q}$ may depend on $q$, but does not depend on $x$.
\end{lem}

\begin{lem}\label{lem:Exponentials}
Assume \ref{C2:H1}, \ref{C2:H_loclip},  and \ref{C2:H_poly}.
Then, when $b(0)=0$, there exists a constant $C$, independent on $x$, such that for all $\mu\leq\tfrac{(\sigma^2+2B_2)^2}{8\sigma^2}$,
\begin{equation}\label{ExpMoments0}
\sup_{t\in[0,T]}\EE\big[\exp\{\mu\int_0^tX_s^{2\alpha-2}ds\}\big]\leq\;C\;\left(1+x^{\frac12+\frac{B_2}{\sigma^2}}\right).
\end{equation}
Otherwise, when $b(0)>0$, if we assume in addition that $\alpha> \frac32$, then for all $\mu< B_2\sigma^2$,
\begin{equation}\label{ExpMoments}
\sup_{t\in[0,T]}\EE\big[\exp\{\mu\int_0^tX_s^{2\alpha-2}ds\}\big]\leq\;C(1+x^{\frac{2\mu}{\sigma^2}})~\left(1+\exp\{C\mu x^{-1}\}\right).
\end{equation}

In the above upper-bounds, the non-negative constant $C$ ca be bounded uniformly in $\mu$.
\end{lem}

\section{The exponential Euler scheme and its rate of  convergence}\label{weakerror}

Given the possible non-integer power value for $\alpha$ in the diffusion term,  we seek for an appropriate numerical approximation preserving the positiveness of the process and exponential form is a good candidate for this purpose.
By rewriting the SDE \eqref{eq:IntroSDE} as
\[
dX_t = X_t\Big(\tfrac{b(X_t)}{X_t}dt+\sigma X^{\alpha-1}dW_t,\Big),~~ X_0=x>0,
\]
and given $\{t_0=0,t_1,\ldots,t_{N-1},t_N=T\}$, a $N$-partition of the time interval $[0,T]$   with  time-step $\Delta t=t_{n+1}-t_n$,  we consider  first  the approximation $(\widehat{X}_{t_{n}},n\geq1)$ given by
 \begin{equation}\label{eq:Explicit exp-Euler}
\widehat{X}_{t_{n+1}}=\widehat{X}_{t_n}\exp\Big\{\Big(\tfrac{b(\widehat{X}_{t_n})}{\widehat{X}_{t_n}}-\frac{\sigma^2}2 \widehat{X}_{t_n}^{2(\alpha-1)}\Big)\Dt + \sigma\overline{X}_{t_n}^{\alpha-1}\;(W_{t_{n+1}}-W_{t_n}) \Big\}, ~~\widehat{X}_0=x,
\end{equation}
and its continuous version given by the interpolation in time:
\[d\widehat{X}_t = \widehat{X}_t\Big(\tfrac{b(\widehat{X}_{\eta(t)})}{\widehat{X}_{\eta(t)}}dt+\sigma \widehat{X}^{\alpha-1}_{\eta(t)}dW_t\Big), ~~X_0=x>0,\]
where $\eta(t):=\sup\{t_i:t_i <  t\}$. Ensuring the strict positivity of the approximation at all times, the scheme \eqref{eq:Explicit exp-Euler}  enables also  to counterbalance the rapid growth of the diffusion $\widehat{X}^{\alpha-1}_{\eta(t)}$ by the drift contribution $\tfrac{b(\widehat{X}_{\eta(t)})}{\widehat{X}_{\eta(t)}}$ subject to  \ref{C2:H_poly} and \ref{C2:H_loclip}. The scheme \eqref{eq:Explicit exp-Euler}  is also sensitive to the value of $b$ near zero: when $b(0)=0$,  \ref{C2:H_poly} yields to
\[ \frac{b(x)}{x}\leq B_1-B_2 x^{2(\alpha-1)}, \forall x\geq 0,\]
and, combined with \ref{C2:H_loclip},  enables to prove the existence of some positive moments for $(\widehat{X}_t;0\leq t\leq T)$  (replicating for instance the last proof steps of Proposition \ref{prop:XMoments}). But when $b(0)> 0$, numerical instabilities can be observed when $\widehat{X}$ comes close to zero. More specifically, we  haven't been able to find  a threshold $\xi$ such that $\PP(\widehat{X}_t\leq \xi)$ decays in $\Delta t$, nor to control some positive moments in that case. 

To overcome such instabilities, the continuous version of the scheme \eqref{eq:Explicit exp-Euler} can be modified by adding and subtracting $b(0)$ as follows:
 \[d\widehat{X}_t = b(0)dt+\widehat{X}_t\Big(\tfrac{b(\widehat{X}_{\eta(t)})-b(0)}{\widehat{X}_{\eta(t)}}dt+\sigma \widehat{X}^{\alpha-1}_{\eta(t)}dW_t\Big),\]
or equivalently, defining ${\dt}:= t-{\eta(t)}$,
\begin{align*}
\widehat{X}_t=\widehat{X}_{\eta(t)}\exp\Big\{\sigma\widehat{X}_{\eta(t)}^{\alpha-1}\;(W_t-W_{\eta(t)})+\Big(\tfrac{b(\widehat{X}_{\eta(t)})-b(0)}{\overline{X}_{\eta(t)}}-\tfrac{\sigma^2}2 \widehat{X}_{\eta(t)}^{2(\alpha-1)}\Big){\dt}+\int_{\eta(t)}^t\tfrac{b(0)}{\widehat{X}_s}ds\Big\},
\end{align*}
for which we need to discretize the integral appearing in the right hand-side to turn it in  a numerical algorithm. The approximation
$$
\int_{\eta(t)}^t\tfrac{b(0)}{\widehat{X}_s}ds \approx \tfrac{b(0)}{\widehat{X}_{\eta(t)}}{\dt},$$
 makes the corresponding scheme comes back to  \eqref{eq:Explicit exp-Euler} for which we do not control --a priori-- positive moments. In contrast,  the approximation
$$\int_{\eta(t)}^t\tfrac{b(0)}{\widehat{X}_s}ds \approx \tfrac{b(0)}{\widehat{X}_{t}}{\dt},$$
produces the following implicit numerical scheme:
\begin{equation}\label{eq:Explicit exp-Euler III}
f\left(t,\check{X}_t\right)=\check{X}_{\eta(t)}\exp\big\{\sigma\check{X}_{\eta(t)}^{\alpha-1}\;(W_t-W_{\eta(t)})+\big(\tfrac{b\left(\check{X}_{\eta(t)}\right)-b(0)}{\check{X}_{\eta(t)}}-\tfrac{\sigma^2}2 \check{X}_{\eta(t)}^{2(\alpha-1)}\big){\dt}\big\},
\end{equation}
where $f(t,x)= x\exp\{-\frac{b(0){\dt}}{x}\}$, for which control of positive moments for $(\check{X}_t;0\leq t\leq T)$ are obtained under~\ref{C2:H_poly} (see \cite[Chap. 2]{kerlyns}).

To turn \eqref{eq:Explicit exp-Euler III} in a numerical scheme, we combine it with an approximation  method for $x\mapsto f^{-1}(x)$,  by considering a Taylor expansion of first order:
\[f(t,x)\approx x-b(0){\dt}.\]
With this last approximation, we define the scheme  $(\overline{X}_{t_{n}},n\geq1)$, that we now refer to as exp-ES, for \emph{ Exponential-Euler Scheme}, by $\overline{X}_0=x$, and
\begin{equation}\label{eq:DecomposednumericScheme2}
\overline{X}_{t_{n+1}}=b(0)\Delta t+\overline{X}_{t_n}\exp\Big\{\sigma\overline{X}_{t_n}^{\alpha-1}(W_{t_{n+1}}-W_{t_n})+\big(\tfrac{b\left(\overline{X}_{t_n}\right)-b(0)}{\overline{X}_{t_n}}-\tfrac{\sigma^2}2 \overline{X}_{t_n}^{2(\alpha-1)}\big)\Dt\Big\},
\end{equation}
admitting the continuous version
\begin{equation}\label{eq:DecomposednumericScheme}
\overline{X}_t=b(0){\dt}+\overline{X}_{\eta(t)}\exp\Big\{\sigma\overline{X}_{\eta(t)}^{\alpha-1}(W_t-W_{\eta(t)})+\big(\tfrac{b\left(\overline{X}_{\eta(t)}\right)-b(0)}{\overline{X}_{\eta(t)}}-\tfrac{\sigma^2}2 \overline{X}_{\eta(t)}^{2(\alpha-1)}\big){\dt}\Big\},
\end{equation}
driven by the SDE:
\begin{equation}\label{eq:ContExpscheme}
d\overline{X}_t=\left(\overline{X}_t-b(0){\dt}\right)\Big(\tfrac{b\left(\overline{X}_{\eta(t)}\right)-b(0)}{\overline{X}_{\eta(t)}}dt+\sigma\overline{X}_{\eta(t)}^{\alpha-1}dW_t\Big)+b(0) dt.
\end{equation}

\begin{rem}\label{rem:SchemePositive}
By construction,  due to the exponential form in \eqref{eq:DecomposednumericScheme}, $\overline{X}_t-b(0){\dt}>0$, provided that $x>0$. In particular, for all $0\leq t\leq T$, $\overline{X}_t>0$ .
\end{rem}

For the exp-ES  $\overline{X}$, we bound the same order of  $2p$\,th-moments than for $X$ in Proposition \ref{prop:XMoments}.
\begin{lem}\label{leq:Schememoments}
Assume \ref{C2:H1}, \ref{C2:H_loclip} and \ref{C2:H_poly}. For all exponent $0 < 2p \leq 1+\frac{2B_2}{\sigma^2}$, there exists a non-negative constant $C_p$, depending on $p$ but not on $x$, such that
\[
\sup_{t\in[0,T]}\EE\big[{\overline{X}_t^{2p}}\big]\leq C_p(1+x^{2p}),\,x>0.
\]
\end{lem}
\begin{proof}
Applying It\^{o}'s formula to the  process $\overline{X}$ (for simplification, we omit the localization argument previously used in the proof of Proposition \ref{prop:XMoments}), we get
\begin{align*}
& \EE[\overline{X}_{t}^{2p}]
= x^{2p}+2pb(0)\EE\big[\int_0^{t\wedge\tau_M}\overline{X}_s^{2p-1}ds\big]\\
& +p\EE\Big[\int_0^{t}\overline{X}_s^{2p-2}(\overline{X}_s-b(0){\ds})\Big\{2\overline{X}_s\tfrac{b(\overline{X}_{\eta(s)})-b(0)}{\overline{X}_{\eta(s)}}+(2p-1)\sigma^2\left(\overline{X}_s-b(0){\ds}\right)
\overline{X}_{\eta(s)}^{2(\alpha-1)}\Big\}ds
\Big]\\
&\leq x^{2p}+2pb(0)\int_0^{t}\EE[\overline{X}_s^{2p-1}] ds
+2pB_1
\EE\big[\int_0^{t}\overline{X}_s^{2p-1}(\overline{X}_s-b(0){\ds})ds\big]\\
&\leq C(1+x^{2p})+C\int_0^{t}\EE[\overline{X}_{s}^{2p}] ds .
\end{align*}
The proof ends by applying Gronwall's inequality.
\end{proof}

\subsection{Main results}
Under the following hypotheses we state below the weak rate of convergence of order one for \eqref{eq:IntroSDE} associated with the exponential-Euler scheme \eqref{eq:DecomposednumericScheme2}:

\setcounter{primehyp}{1}
\begin{subprimehyp}
\begin{primehyp}\label{C2:H_loclip'}
{\bf (For the regularity of the  Kolmogorov PDE \eqref{eq:KolPDE}.)} The function $b$ is $2(\alpha-1)$-locally Lipschitz, and $b(0)\geq0$. In addition, $b$ is of class $\mathcal{C}^4(\mathbb{R}^+)$, with derivatives $b^{(i)}$ being  $({\overline{\gamma}}_{(i)},{\underline{\gamma}}_{(i)})$-locally Lipschitz continuous, for $i=1,\ldots,4$.
\end{primehyp}
\end{subprimehyp}
\ref{C2:H_loclip'} implies in particular that 
\begin{align}\label{eq:poly_bi}
|b^{(i)}(x)| \leq C \big(1 +x^{\overline{\gamma}_{(i)}+1} + x^{-\underline{\gamma}_{(i)}}\big), ~ \mbox{ for }i=1,2,3,4,
\end{align}
where, according to Lemma \ref{lem:LocallyLipschitz}, $2\alpha-3 \leq \overline{\gamma}_{(1)}$, $\overline{\gamma}_{(i)}\leq \overline{\gamma}_{(i+1)} + 1$, and $\underline{\gamma}_{(i)}\leq \underline{\gamma}_{(i+1)}$ for $i=1,2,3$.

\begin{subprimehyp}
\begin{primehyp}\label{C2:H_poly'}
{\bf (For the exponential moments of $X$.)} There exist a set of constants $B_i, B_i' \geq0$, with $i=1,2$, such that
\begin{equation*}
b(x)\leq B_1 x - B_2 \,x^{2\alpha-1}+b(0),\;\mbox{ and }~ b'(x)\leq B_1' - B_2^\prime\,x^{2(\alpha-1)}.
\end{equation*}
\end{primehyp}
\end{subprimehyp}

\setcounter{hypbis}{3}
\begin{subhypbis}
\begin{hypbis}\label{C2:H4_poly}%{C2:H_polyiv}
{\bf (For the weak convergence rate derivation.)} The powers $\underline{\gamma}_{(i)}$ in \ref{C2:H_loclip'} satisfy:
\[\underline{\gamma}_{(i)}\leq i-1, ~ \mbox{ for }i=1,2,3 ~~\mbox{and }~\underline{\gamma}_{(4)}\leq 4.\]
\end{hypbis}
\end{subhypbis}

\begin{subhypbis}
\begin{hypbis}\label{C2:H5}
{\bf (For the regularity of the Kolmogorov PDE in Proposition  \ref{prop:KolPDE}).} The constants $B_2$, $B_2'$, $\alpha$ and the $(\overline{\gamma}_{(i)}, i=1,\ldots,4)$ in \ref{C2:H_loclip'}  and  \ref{C2:H_poly'} satisfy
\begin{align*}
B_2 \geq &  3 \sigma^2 \alpha  +  \tfrac{\sigma^2}{2}
\big[
\big\{
(2\overline{\overline{\beta}})\vee(\overline{\overline{\beta}} + 2\alpha)
\big\} -1
\big],\\
B_2^\prime \geq & \sigma^2\alpha\left(\tfrac{17}{2}\alpha -3\right)
\end{align*}
where $\overline{\overline{\beta}}=3(\overline{\gamma}_{(2)}+1)\vee ( \overline{\gamma}_{(2)} + \overline{\gamma}_{(3)}+2) \vee (\overline{\gamma}_{(4)}+1)$.

In addition whenever $b(0)>0$, we assume that $\alpha >\tfrac32$ and  we modify the constraint on $B_2$ as
$$B_2 \geq 3  \sigma^2 \alpha  + \frac{\alpha^2}{2} \vee \frac{\sigma^2}{2}
\Big[
\big\{
(2\overline{\overline{\beta}})\vee(\overline{\overline{\beta}} + 2\alpha)
\big\} -1
\Big].$$
\end{hypbis}
\end{subhypbis}

\begin{prop}\label{prop:Weak rate}
Let $f$ be a bounded $\mathcal{C}^4(\RR^+)$ function with bounded derivatives up to order 4. Consider the process $X$ solution to the SDE \eqref{eq:IntroSDE} with deterministic initial condition $x>0$, together with its approximation $\overline{X}$ in \eqref{eq:DecomposednumericScheme}. Assume \ref{C2:H1}, \ref{C2:H_loclip'},\ref{C2:H_poly'}, \ref{C2:H4_poly},  and \ref{C2:H5}. Then, there exists a constant $C>0$ depending on the parameters $B_i,B_i',\alpha,\sigma$ and possibly on $T$ and $x$, but independent on $\Delta t$, such that
\begin{equation}\label{eq:RateWeakConv1}
\left|\EE\big[ f(X_T)\big]- \EE\big[ f(\overline{X}_T)\big] \right|\leq C\;\Delta t.
\end{equation}
\end{prop}

The hypotheses in  Proposition \ref{prop:Weak rate}  are all sufficient conditions, considered in order to simplify the analysis of the regularity associated with the solution of the backward Kolmogorov PDE.

Precisely, Hypotheses \ref{C2:H_poly'} and \ref{C2:H_loclip'} are considered in order to obtain polynomial bounds  for the derivatives of the solution to the backward Kolmogorov PDE (Proposition \ref{prop:KolPDE}).
Later, in the computation of the weak error,  we use \ref{C2:H5} specifically to control the resulting positive moments of the exp-ES process (see the proof of  Proposition \ref{prop:Weak rate} in Section 6), and  by considering \ref{C2:H4_poly} we seek to avoid the need to control the negative moments of the  approximation scheme  arising also from the estimation of these derivatives  (see for instance the inequality \eqref{eq:exemple_control} below).

We also emphasis that the analysis exposed in this paper can be easily  adapted to the case of locally bounded $\mathcal{C}^4(\RR^+)$-function $f$  with locally bounded derivatives.

\section{Numerical Experiments}\label{sec:Numeric}

This section illustrates with some experiments the theoretical rate of convergence in  Proposition \ref{prop:Weak rate}. More experiments are shown in Appendix \ref{appendix:numerical} and  explore the hypotheses on the coefficient parameters by testing the convergence  through different cases,  fulfilling or not some of the hypotheses.

In particular we illustrate the fact that hypothesis  \ref{C2:H5} do not correspond to a necessary condition on the parameters $B_2,B_2',\alpha,\sigma$ involved in the model. First we restrict the set of parameters by  considering   the explicit model 
\begin{equation}\label{eq:proto_CEV_with_const}
dX_t = (B_0 + B_1 X_t -B_2\;X_t^{2\alpha-1}) dt+\sigma X_t^\alpha dW_t, \;X_0=x>0,
\end{equation}
for which $B'_1=1$ and $B'_2= (2\alpha-1) B_2$.  Proposition \ref{prop:Weak rate} can be shapely adapted in this particular situation as follows:
\begin{cor}\label{prop:Weak rateCor}
We consider the solution $X$  to \eqref{eq:proto_CEV_with_const}. When $B_0 =0$, assume $\alpha >1$ and
\begin{align}\label{eq:constraint_a1}
B_2 - 3 \sigma^2 \alpha  -   \tfrac{\sigma^2}{2}
\big[
(12\alpha-19) \vee (8\alpha-10) \vee \tfrac{5\alpha^2}{(2\alpha -1)}
\big]\geq 0,
\end{align}
When $B_0 >0$, assume  in addition that $\alpha >\tfrac32$ and
\begin{align}\label{eq:constraint_a32}
B_2 - 3  \sigma^2 \alpha  -  \tfrac{\alpha^2}{2} \vee \tfrac{\sigma^2}{2}
\big[
(12\alpha-19) \vee (8\alpha-10) \vee \tfrac{5\alpha^2}{(2\alpha -1)}
\big]\geq 0.
\end{align}
Then, for $f \in \mathcal{C}_b^4(\RR^+)$, there exist a constant $C>0$ depending on the parameters $B_i,\alpha,\sigma$ and possibly on $T$ and $x$, but independent on $\Delta t$, such that
\begin{equation}\label{eq:RateWeakConv}
\left|\EE[f(X_T)]-\EE[f(\overline{X}_T)] \right|\leq C\;\Delta t.
\end{equation}
\end{cor}

\paragraph{Numerical parameters.  }For all the presented numerical experiments,  we consider a unit terminal time $T=1$, the initial condition $x=1$ and the time step $\Delta t=1/2^p$, for $p=1,\ldots, 9$. In addition, the empirical mean of the scheme approximating $\EE f(\overline{X}_T)$ is estimated by a Monte Carlo approximation, involving $n=10^5$ independent trajectories.

\paragraph{Test functions. }Along this section, we consider four different test functions, not all bounded,
\[f(x)=x,\;x^2,\;\exp(-x^2).\]

\paragraph{Model cases. }Denoting $\kappa$  as the left-handside of \eqref{eq:constraint_a1}  or \eqref{eq:constraint_a32}, we consider the following cases, determined by the data $(B_0, B_1,B_2,\sigma,\alpha)$
\begin{center}
\begin{tabular}{lll}
{\bf Case 1}\qquad
$(0,0,2,\tfrac{1}{10},\tfrac32)$
&  $dX_t = -2X_t^{2} dt+ \frac{X_t^{3/2}}{10} dW_t$
&\qquad $\kappa > 1.95$ \\
{\bf Case 2}\qquad %
$(0,0,3,1,\tfrac54)$
&   $dX_t = -3X_t^{{3}/{2}} dt+ X_t^{{5}/{4}} dW_t $
&\qquad $\kappa < -3$ \\
{\bf Case 3}\qquad %
$(0,0,1,1,\tfrac32)$
&   $dX_t = -X_t^{2} dt+\sigma X_t^{3/2} dW_t$
&\qquad $\kappa < -3$ \\
{\bf Case 4}\qquad %
$(1,1,\tfrac{2}{5},\tfrac{1}{10},3)$
&  $dX_t =(1+X_t -\tfrac{2}{5} X_t^{5} )\ dt+\frac{X_t^{3}}{10} dW_t$
&\qquad $\kappa < -4$ \\
\end{tabular}
\end{center}
with two of them, {\bf Cases 2} and {\bf 3},  that are not satisfying \ref{C2:H5}.

Moreover, {\bf Case 3} satisfies assumptions of Theorem 2.1 in \cite{Kloeden} that states that  the approximated moments by the  Euler-Maruyama scheme and the strong $L^p$-error associated to moment-approximations diverges. Also in \cite{Kloeden}, the authors prove the divergence in the weak sense for  the $p$-th moments of the Euler-Maruyama scheme in that case.

From Proposition \ref{prop:XMoments} and Lemma \ref{lem:Negative}, the triplets $(B_2,\sigma,\alpha)$ in
{\bf Case 1} to {\bf Case 4},  guarantee the finiteness of the expectation $\EE f(X_T)$ for each test function in the set $\{x,x^2,\exp(-x^2)\}$.

\paragraph{Computation of the reference values. }
For both test functions $f(x)=x$ and $f(x)=x^2$,  reference values of $\EE[X_T]$ and $\EE[X_T^2]$ are computed analytically  for {\bf Case 1} to {\bf Case 3},  (see details in  \cite[Chapter 1]{kerlyns} and Appendix \ref{appendix:numerical}). For the others cases, the reference values are computed based on a $n_0$-Monte Carlo method combined with  the scheme \eqref{eq:Explicit exp-Euler}, $n_0 = 10^7$  and  $\Delta t_{\text{ref}} =2^{-14}$:
\[\EE[f(X_T)]\approx\frac{1}{n_0}\sum_{i=1}^{n_0}f(\overline{X}_T(\omega_i,\Delta t_{\text{ref}})).\]

Numerical results are shown in Table \ref{tab:C2experiment1}, where we can observe the rate of convergence of order one, except for Case 3 with test function $f(x)=x^2$, in all the rows corresponding to the selection of bounded/unbounded test functions: the error is divided by 2 when going from left to right,  even if some saturation can be observed for the smallest error values $(p=8,9)$ when Monte Carlo error starts to be dominant.
\begin{figure}[H]%[ht!]
\centering
{\includegraphics[width=\textwidth]{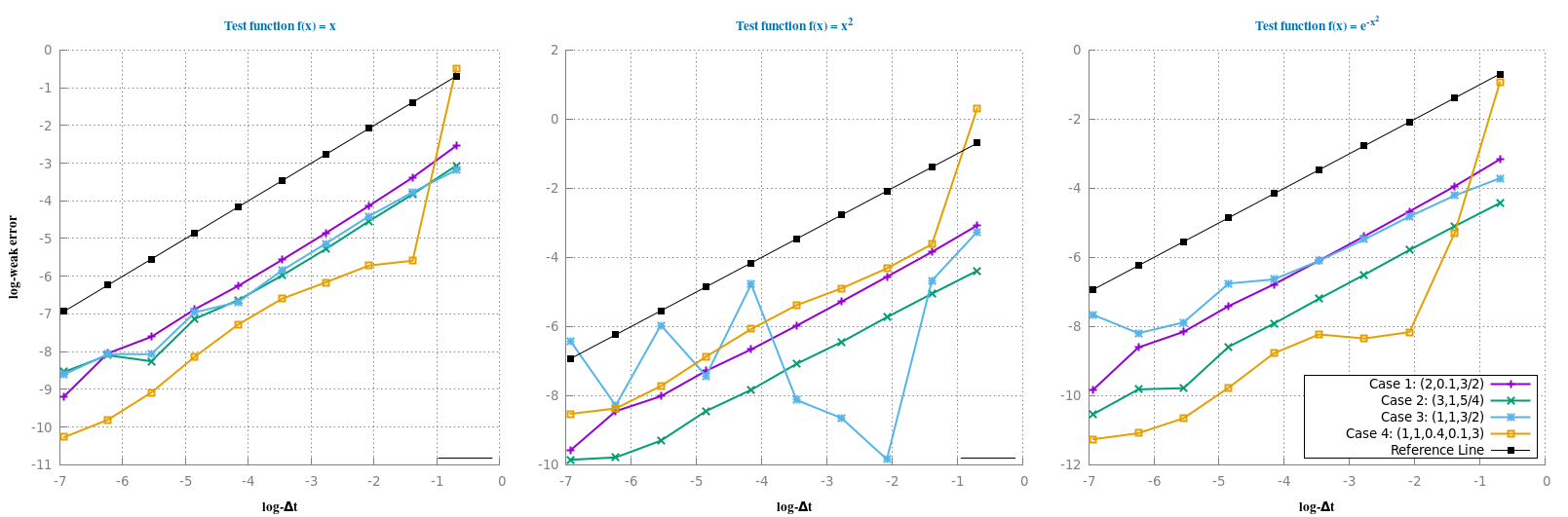}}
\caption{ Weak approximation error for the exponential-Euler scheme applied to  \eqref{eq:proto_CEV_with_const}, with  {\bf Case 1  to 4} (in log-log scale), the weak error is compared with the reference slope of order 1 (black).\label{im:C2experiment1}}
\end{figure}

\begin{table}[H]%[ht]
\centering
{\footnotesize{
\begin{tabular}{l| l l l l l l l l}
\toprule
\multicolumn{9}{c}{Weak Error with $\Delta t={2^{-p}}$, for $p=2,\ldots,9$}\\
\hline
\multicolumn{9}{c}{\tabhead{Case 1: $(B_2,\sigma,\alpha) = (2,\frac{1}{10},\frac32)$ and \ref{C2:H5} is valid.}}\\
\hline
\tabhead{{\small Test function}}
& \tabhead{$p=2$} &\tabhead{$p=3$}& \tabhead{$p=4$} & \tabhead{$p=5$} &\tabhead{$p=6$} & \tabhead{$p=7$} & \tabhead{$p=8$}&\tabhead{$p=9$}\\
\midrule
$f(x)=x$ & 3.397e-2& 1.606e-2& 7.756e-3& 3.823e-3& 1.923e-3& 1.033e-3& 4.965e-4&3.199e-4 \\
$f(x)=x^2$ & 2.147e-2& 1.043e-2&5.102e-3 &2.529e-3 &1.277e-3 &6.864e-4 &3.297e-4&2.131e-4\\
$f(x)=e^{-x^2}$ & 1.94e-2& 9.378e-3& 4.568e-3&2.258e-3 &1.135e-3 &6.06e-4 &2.874e-4 &1.829e-4\\
\hline
\multicolumn{9}{c}{\cellcolor{blue!25}\tabhead{Case 2: $(B_2,\sigma,\alpha) =(3,1, \frac54)$ and \ref{C2:H5} is  not valid.}}\\
\hline
\tabhead{{}}
& \tabhead{$p=2$} &\tabhead{$p=3$}& \tabhead{$p=4$} & \tabhead{$p=5$} &\tabhead{$p=6$} & \tabhead{$p=7$} & \tabhead{$p=8$}&\tabhead{$p=9$}\\
\midrule
$f(x)=x$& 2.179e-2& 1.069e-2& 5.07e-3& 2.529e-3 & 1.32e-3& 8.021e-4& 2.598e-4&3.043e-4\\
$f(x)=x^2$ &6.412e-3 &3.243e-3 &1.582e-3 &8.397e-4 &3.965e-4 &2.148e-4 &9.101e-5 &5.065e-5 \\
$f(x)=e^{-x^2}$ & 6.113e-3&3.07e-3 &1.5e-3 &7.5e-4 &3.65e-4 &1.868e-4 &5.625e-5 &5.439e-5  \\
\hline
\multicolumn{9}{c}{\cellcolor{violet!25} \tabhead{Case 3: $(B_2,\sigma,\alpha) =(1,1, \frac32)$ and \ref{C2:H5} is  not valid.}}\\
\hline
\tabhead{{}} & \tabhead{$p=2$} &\tabhead{$p=3$}& \tabhead{$p=4$} & \tabhead{$p=5$} &\tabhead{$p=6$} & \tabhead{$p=7$} & \tabhead{$p=8$}&\tabhead{$p=9$}\\
\midrule
$f(x)=x$& 2.3e-2& 1.219e-2&5.864e-3 &2.893e-3 & 1.255e-3& 9.507e-4& 3.13e-4& 3.14e-4\\
\cellcolor{red!25} $f(x)=x^2$ \cellcolor{red!25}  &\cellcolor{red!25}  9.408e-3 &\cellcolor{red!25}   5.302e-5 & \cellcolor{red!25}  1.749e-4 & \cellcolor{red!25}  2.956e-4 &\cellcolor{red!25}  8.41e-3 & \cellcolor{red!25}  5.95e-4 &\cellcolor{red!25}  2.574e-3 & \cellcolor{red!25}  2.52e-4  \\
$f(x)=e^{-x^2}$ & 1.485e-2&8.108e-3 &4.162e-3 &2.248e-3 &1.31e-3 &1.164e-3 &3.78e-4 &2.762e-4   \\
\hline
\multicolumn{9}{c}{\cellcolor{blue!25} \tabhead{Case 4: $(B_0,B_1,B_2,\sigma,\alpha) = (1,1,\frac25,\frac{1}{10},3)$ and \ref{C2:H5} is not valid.}}\\
\hline
\tabhead{{}}
& \tabhead{$p=2$} &\tabhead{$p=3$}& \tabhead{$p=4$} & \tabhead{$p=5$} &\tabhead{$p=6$} & \tabhead{$p=7$} & \tabhead{$p=8$}&\tabhead{$p=9$}\\
\midrule
$f(x)=x$ & 3.741e-3& 3.292e-3& 2.103e-3& 1.364e-3& 6.915e-4 & 2.936e-4& 1.131e-4& 5.497e-5\\
$f(x)=x^2$ & 2.687e-2&1.332e-2 &7.476e-3 &4.584e-3 &2.302e-3 &1.034e-3 &4.423e-4 &2.312e-4 \\
$f(x)=e^{-x^2}$ &5.027e-3 &2.846e-4 &2.372e-4 &2.666e-4 & 1.545e-4& 5.72e-5 &2.363e-5 &1.529e-5  \\
\bottomrule
\end{tabular}
}}
\caption{ Observed numerical weak error $|\EE[f(X_T)] - \frac{1}{n}\sum_{i=1}^{n}f(\overline{X}_T(\omega_i, 2^{-p}))|$.\label{tab:C2experiment1}}
\end{table}
This behavior is also illustrated in Figure \ref{im:C2experiment1}, plotting  the  obtained results in a log-log scale.  This confirms that our proofs can certainly be extended for a larger class of test functions, and model parameters.
In particular, we highlight {\bf Case 3} that converges weakly with order one for $f(x) = x, \exp\{-x^2\}$, even if \ref{C2:H5} is not fulfilled, and even moreover   we know that the classical Euler-Maruyama scheme is  strongly diverging (as stated in \cite{Kloeden}) in this case.
\medskip

Additional  numerical results  are proposed in  Appendix \ref{appendix:numerical} where, in particular, some comparisons on the exp-ES with other dedicated numerical schemes are presented.

\section{Analysis of the backward Kolmogorov PDE related to \eqref{eq:IntroSDE}}\label{sec:CauchyP}

This section is devoted to the regularity analysis on the solution of the backward Kolmogorov PDE related to \eqref{eq:IntroSDE}. Stochastic analysis is used here to establish key estimates on the solution of the PDE.  We consider the flow process $(X_t^x;0\leq t\leq T)$ starting from $x>0$:
\begin{equation}\label{eq:flux}
X_t^x = x +  \int_0^tb(X^x_s)ds + \sigma\int_0^t(X_s^x)^{\alpha}dW_s, ~~\forall \;t\in(0,T].
\end{equation}
According to the Feynman-Kac representation theorem (see e.g. \cite[Chap. V]{KarShr-88}), provided that $x\mapsto X_{T-t}^x$ and a given  $f$ are smooth enough, the function
\begin{align}\label{eq:feymann-kac}
u(t,x)=\EE[f(X_{T-t}^x)]
\end{align}
is a natural candidate to be the classical solution to the backward Kolmogorov PDE:
\begin{align}\label{eq:KolPDE}
\left\{
\begin{aligned}
& \frac{\partial u}{\partial t}(t,x) + b(x)\frac{\partial u}{\partial x}(t,x) + \frac{\sigma^2}2\;x^{2\alpha}\; \frac{\partial^2 u}{\partial ^2x}(t,x)=0,~\;\mbox{ for all }~(t,x)\in[0,T)\times\RR^+, \\
& u(T,x)=f(x),\mbox{ for all }~\;x\in[0,+\infty).
\end{aligned}
\right.
\end{align}

\begin{prop}\label{prop:KolPDE}
Let $f\in\mathcal{C}^4_b(\RR^+)$. Then, assuming \ref{C2:H1},\ref{C2:H_loclip'},\ref{C2:H_poly'} and \ref{C2:H5}, the function $u(t,x)$ in \eqref{eq:feymann-kac}  is a  solution to the PDE \eqref{eq:KolPDE} of class $\mathcal{C}^{1,4}([0,T]\times\RR^+)$, and there exists a finite constant $C$ such that
\begin{equation}\label{eq:PDE bounds1}
\begin{array}{ll}
&\|u\|_{L^\infty((0,T)\times\RR^+)} + \;\left\|\frac{\partial u}{\partial x}\right\|_{L^\infty((0,T)\times\RR^+)}\leq C,\\
\mbox{and for all } x\in (0,+\infty),&\sup_{t\in[0,T]}\big|\frac{\partial u}{\partial t}\big|(t,x)\leq C\;\big(1+x^{2\alpha}\big),\\
&\sup_{t\in[0,T]}\big|\frac{\partial^2 u}{\partial x^2}\big|(t,x)\leq C\big(1+x^{\overline{\gamma}_{(2)}+1}+x^{-\underline{\gamma}_{(2)}}\big),\\
&\sup_{t\in[0,T]}\big|\frac{\partial^3 u}{\partial x^3}\big|(t,x)\leq C\;\big(1+x^{\overline{\beta}}+x^{-\underline{\beta}}\big),\\
&\sup_{t\in[0,T]}\big|\frac{\partial^4 u}{\partial x^4}\big|(t,x)\leq C\;\big(1+x^{\overline{\overline{\beta}}}
+x^{-~\underline{\underline{\beta}}}\big),
\end{array}
\end{equation}
where $\overline{\gamma}_{(i)}$ and $\underline{\gamma}_{(i)}$ are as in \ref{C2:H_loclip'}, $\overline{\overline{\beta}}$ as in \ref{C2:H5},
and $\overline{\beta}$, $\underline{\beta}$ are $\underline{\underline{\beta}}$ are given by
\begin{align*}
&\overline{\beta}=2(\overline{\gamma}_{(2)}+1)\vee(1+\overline{\gamma}_{(3)}),\,\,\,\,\underline{\beta}=2\underline{\gamma}_{(2)} \vee\underline{\gamma}_{(3)}\vee (\underline{\gamma}_{(2)}+3-2\alpha),\\
&\underline{\underline{\beta}}
= \big\{\big(\underline{\gamma}_{(2)} \vee (3 -2\alpha)\big)  + \underline{\beta} \big\} \vee \big\{\underline{\gamma}_{(2)} + \big(\underline{\gamma}_{(3)} \vee (4 -2\alpha)\big)\big\}\vee \underline{\gamma}_{(4)}.
\end{align*}
\end{prop}

\subsection{Main lines for  the proof of Proposition
\ref{prop:KolPDE}}\label{sec:sketch}

The proof of Proposition \ref{prop:KolPDE}
follows the methodology used in \cite{BD15} that combines some adequate successive  changes of measure in the Feynman-Kac formula for $u$ and in its derivatives in order to kill some unsuitable term in the obtained expression for $\frac{\partial^i u}{\partial x^i}$ before to derive it again. More precisely,  the main hypothesis $\alpha > 1$ allows to derive at most one time the diffusion coefficient $x\mapsto x^\alpha$, and by extension $x\mapsto (X^{x}_t;\,0\leq t\leq T)$, before to potentially produce some negative power term for higher order derivative.  In contrast, Lemma  \ref{lem:Exponentials} embed a local Novikov condition which allows to control the exponential martingale for the first derivative of the diffusion only (the power $2(\alpha -1)$ corresponding to the quadratic variation of resulting from this derivative).
We present here, briefly and formally, how we can combine derivatives and change of measure to overcome higher order derivative of the diffusion before detailing the proof in the rest of this section and in the Appendix \ref{sec:app:proof_FK}.

Following \cite{BD15}, we introduce the family of processes $X^x(\lambda)$, with parameter $\lambda>0$, as the solution of the SDE:
\begin{equation}\label{eq:flux2}
X_t^x(\lambda) = x + \int_0^t \big\{b(X_s^x(\lambda))+\lambda \sigma^2\;(X_s^x(\lambda))^{2\alpha-1}\big\}ds + \sigma\int_0^t(X_s^x(\lambda))^{\alpha}dW_{s}, ~~\forall \;t\in(0,T].
\end{equation}
For each $\lambda>0$, Equation \eqref{eq:flux2} can be seen as a modification of \eqref{eq:flux} with a drift component $b^\lambda$ given by
$${b}^\lambda(x)=b(x)+\lambda \sigma^2\;x^{2\alpha-1}\leq B_1 x -B_2^\lambda x^{2\alpha-1}+b(0),$$ where $B_2^\lambda:=B_2-\lambda\sigma^2$. Due to the locally Lipschitz property of the coefficients (and their derivatives) in \eqref{eq:flux2}, the process $X^x(\lambda)$ is continuously differentiable w.r.t. $x$ (see Protter~\cite[Thm V.39]{Protter-04}) for $\mathbb{P}$-almost all $\omega\in\Omega$. Therefore, we can define the derivative of the flow with respect to the initial condition $x$:
$$J_t^x(\lambda)=\frac{dX_t^x}{dx}(\lambda),~\;0\leq t\leq T,$$
as the solution to the SDE, for $t\in(0,T]$,
\begin{align}\label{eq:derivativeSDE}
\left\{
\begin{aligned}
&\frac{dJ_t^x(\lambda)}{J_t^x(\lambda)} = \big[b'(X_t^x(\lambda))+\lambda \sigma^2(2\alpha-1)\;(X_t^x(\lambda))^{2(\alpha-1)}\big]dt+\alpha \sigma(X_t^x(\lambda))^{\alpha-1}dW_t,\\
&J_0^x(\lambda)=1.
\end{aligned}
\right.
\end{align}
Whenever the process $(\int_0^t\left({X_s^x}(\lambda)\right)^{\alpha-1}dW_s;~0\leq t\leq T)$ is a square integrable martingale, $J^x(\lambda)$ admits the following exponential form (see e.g. \cite[Theorem V.52]{Protter-04})
\begin{align}\label{eq:expoJlambda}
\begin{aligned}
J_t^x(\lambda)=&\exp\big\{\int_0^t[b'(X_s^x(\lambda))+\lambda \sigma^2(2\alpha-1)X_s^x(\lambda)^{2(\alpha-1)}-\tfrac{\alpha^2 \sigma^2}2 X_s^x(\lambda)^{2(\alpha-1)}]ds\\
& \hspace{1.4cm}+\,\alpha \sigma\int_0^t X_s^x(\lambda)^{\alpha-1}dW_s\big\}.
\end{aligned}
\end{align}
Now, we may identify the first order derivative of $u(t,x)=\EE[f(X^x_{T-t}(0))]$ as:
\begin{equation}
\frac{\partial u}{\partial x}(t,x)  = \EE[f'(X_{T-t}^x(0))J_{T-t}^x(0)].\label{eq:ExchangeDerExpFormal}
\end{equation}
Before computing the second derivative, we change the measure in the expectation above in order to eliminate $J_{T-t}^x(0)$ and avoid the problem of the  {\it a priori} control of  $\tfrac{d }{dx}J_{T-t}^x(0)$: Consider the Radon-Nikodym density
\begin{align*}
 \tfrac{d \mathbb{Q}^{\alpha}}{d\mathbb{P}}\Big|_{\mathcal{F}_t}:= \tfrac{1}{\mathcal{Z}_t^{(0,\alpha)}},\quad \mbox{ with } \quad
\mbox{$ \mathcal{Z}_t^{(0,\alpha)}=\exp\{ -\alpha\sigma \int_0^t (X_{s}^x(0))^{\alpha-1}dB_s^{\alpha} - \tfrac{\alpha^2\sigma^2}2 \int_0^t(X_{s}^x(0))^{2(\alpha-1)}ds\},$}
\end{align*}
for $(B_t^\alpha;\,0\leq t\leq T)$ the standard $\mathbb{Q}^{\alpha}$-Brownian motion given by ${B_t^{\alpha}=W_t-\alpha\sigma\int_0^t (X_{s}^x(0))^{\alpha-1}ds}$.
Then
\[
\frac{\partial u}{\partial x}(t,x)= \EE_{\mathbb{Q}^{\alpha}}\big[f'(X_{T-t}^x(0))\,J_{T-t}^x(0)\,\mathcal{Z}_{T-t}^{(0,\alpha)}\big].
\]
From the explicit form of the process $J^x(0)$, we recognize
 \begin{align*}
\mbox{$ J_{T-t}^x(0)
=\exp\big\{ \int_0^{T-t}b'(X_s^x(0))ds +\alpha\sigma \int_0^{T-t}(X_s^x(0))^{\alpha-1}dB_s^\alpha +\tfrac{\alpha^2\sigma^2}{2}\int_0^{T-t}(X_s^x(0))^{2(\alpha-1)}ds\big\},
$}
\end{align*}
and hence,
$J_{T-t}^x(0)\,\mathcal{Z}_{T-t}^{(0,\alpha)} = \exp \big\{ \int_0^{T-t}b'(X_s^x(0))ds\big\}$.
Moreover, from the identification \\
$\textit{Law}^{\mathbb{Q}^{\alpha}}(X^x(0))=\textit{Law}^\mathbb{P}(X^x(\alpha))$, we can rewrite $\tfrac{\partial u}{\partial x}$ as
\begin{equation}\label{eq:du/dx2}
\mbox{$ \frac{\partial u}{\partial x}(t,x)=\EE\big[f'(X_{T-t}^x(\alpha))\, \exp\{ \int_0^{T-t}b'(X_s^x(\alpha))ds\}\big]$} ,
\end{equation}
which is continuously differentiable in $x$, with a derivative that depends on the derivative of $f'$ and  $b'$ only. The same procedure can be iterated for other higher order derivative of $x\mapsto u(t,x)$. The price to pay here resides in the assumptions needed to get, at each iteration step, the estimates stated in  Proposition~\ref{prop:XMoments} and Lemmas \ref{lem:Negative} and \ref{lem:Exponentials}, which are more constraining  to satisfy according to the  increasing value $\lambda$ introduced after each  successive changes of measure, and corresponding to successive derivatives.  Typically, while we add some unbounded term $\lambda\sigma^2 x^{2\alpha -1}$ to $b$,  the constant $B_2$ in \ref{C2:H_poly}, ensuring the wellposedness for  \eqref{eq:IntroSDE},	 must be strengthened to also ensure the wellposedness and the finiteness of the moments of the solution to \eqref{eq:flux2}, as well as some moments of $J^x(\alpha)$. 
This strengthening on $B_2$ is summarized in the following corollary, combining the  results of  Proposition~\ref{prop:XMoments} and Lemmas \ref{lem:Negative} and \ref{lem:Exponentials}:
\begin{cor}\label{cor:flux}
Assume \ref{C2:H1}, \ref{C2:H_loclip} and \ref{C2:H_poly}. Then, for any $\lambda>0$ such that $\lambda < \tfrac{B_2}{ \sigma^2}$, (and so $B_2^\lambda = B_2 - \lambda\sigma^2 >0$),  there exists a unique positive strong solution to \eqref{eq:flux2}.
For all $q>0$ and  $p>0$ such that $p \leq \frac12+\frac{B^\lambda_2}{\sigma^2}$, this solution further satisfies:
\begin{equation}\label{eq:uniformMo2}
\sup_{t\in[0,T]}\EE[(X_t^x(\lambda))^{2p}]<C_p\;(1+x^{2p}),~~\sup_{t\in[0,T]}\EE[(X_t^x(\lambda))^{-2q}]<C_q(1+x^{-2q}).
\end{equation}
The non-negative constants $C_p$, $C_q$ do not depend on $x$, but $C_p$, respectively $C_q$, may depend on $p$, respectively $q$.

If $b(0)=0$, then for all $\mu\leq\frac{(\sigma^2+ 2 B_2^\lambda)^2}{8\sigma^2}$,
\begin{equation}\label{eq:uniformExp}
\sup_{t\in[0,T]}\EE\big[\exp\{\mu\int_0^tX_s^{2(\alpha-1)}(\lambda)ds\}\big]<+\infty.
\end{equation}
The estimate \eqref{eq:uniformExp} still holds when $b(0)>0$ under the restriction that $\alpha> \frac32$,  and  $\mu<B_2^\lambda\sigma^2$.
\end{cor}

In the rest of Section \ref{sec:CauchyP}, we address rigorously the formal steps described above, validating first the representation \eqref{eq:ExchangeDerExpFormal} in the next section, detailing the change of measure leading to \eqref{eq:du/dx2}, before finally completing the proof of Proposition \ref{prop:KolPDE}

\subsection{Interchanging derivative and expectation}

The (sufficient) conditions for the possible interchange between expectation $\EE$ and derivative with respect to the initial condition $\tfrac{\partial }{\partial x}$ are stated in following:
\begin{prop}\label{prop:interchang}
Let $\Phi,g,h\in\mathcal{C}([0,+\infty))$ some continuously differentiable functions such that $\Phi$ is bounded and has bounded first derivative in $[0,+\infty)$, $g$ is bounded from above, and the functions $g'$ and $h$ satisfy the following growth conditions
\[|h'(x)|\leq C(1+|x|^{\rho_0}+x^{-\rho_1}),\quad\quad |g'(x)|\leq C(1+x^{\rho_2}+x^{-\rho_3}),\]
\[|h(x)|\leq C(1+x^{\rho_4}+x^{-\rho_5}),\]
for all $x>0$ and some non-negative constants $\rho_i$, with $i=0,1,\ldots,5$.
Assume \ref{C2:H1}, \ref{C2:H_loclip'}, \ref{C2:H_poly'} and
\begin{equation}\label{eq:Lambda1}
\max\{\tfrac12, 2(\alpha -1), \rho_0,~\rho_2+\rho_4,~ \overline{\gamma}_{(1)}\}\leq\tfrac12 + \tfrac{B^\lambda_2}{\sigma^2},\quad\mbox{ and }\quad (2\alpha-1)\lambda+\tfrac{\alpha^2}{2} 5\leq \tfrac{B_2'}{\sigma^2}.\tag{\bf $\Lambda$.1}
\end{equation}
Then for any $\lambda\geq0$, the function defined by
$$v(t,x)=\EE\big[\Phi(X_t^x(\lambda))\exp\{\int_0^t g(X_s^x(\lambda)) ds\}\big]+\int_0^t\EE\big[h(X_s^x(\lambda))\exp\{\int_0^s g(X_r^x(\lambda)) dr\}\big]ds,$$
is continuously differentiable in $x$, with	
\begin{align*}\label{eq:interchng}
	\frac{\partial v}{\partial x} (t,x)&=\EE\left[\exp\left\{\int_0^t g(X_s^x(\lambda)) ds\right\}\left(\Phi'(X_t^x(\lambda))J_t^x(\lambda)+\Phi(X_t^x(\lambda))\int_0^t g'(X_s^x(\lambda)) J_s^x(\lambda) ds\right)\right]\\
	&+\int_0^t \EE\left[\exp\{\int_0^s g(X_r^x(\lambda)) dr\}\left(h'(X_s^x(\lambda))J_s^x(\lambda)+h(X_s^x(\lambda))\int_0^s g'(X_r^x(\lambda)) J_r^x(\lambda) dr\right)\right]ds.
	\end{align*}
\end{prop}
The proof of Proposition \ref{prop:interchang} can be summarized as follows:  the goal is to show (omitting the dependence to $\lambda$ in the notation during a few lines) that  $\EE[\tfrac{\Phi(X_t^{x+\epsilon})-\Phi(X_t^x)}{\epsilon}]$
tends to $\EE[\Phi'(X_t^x)J_t^x]$ when $\epsilon$ tends to  $0$. Introducing  the process
\[J_t^{x,\epsilon}:=\tfrac{X_t^{x+\epsilon}-X_t^x}{\epsilon},\]
we start with the following decomposition
\begin{align}\label{eq:decompo}
&\EE\big[\tfrac{\Phi(X_t^{x+\epsilon})-\Phi(X_t^x)}{\epsilon} -
\Phi'(X_t^x)J_t^x \big]=\EE\big[J_t^{x,\epsilon}\int _0^1\Phi'(X_t^x+
\theta \epsilon  J_t^{x,\epsilon})d\theta  - \Phi'(X_t^x) J_t^x \big]\\
&=\EE\big[(J_t^{x,\epsilon}-J_t^x)\int _0^1\Phi'(X_t^x+ \theta \epsilon
  J_t^{x,\epsilon})d\theta\big] +\EE\big[J_t^x\int _0^1\left(\Phi'(X_t^x+
\theta\epsilon J_t^{x,\epsilon})-\Phi'(X_t^x)\right)d\theta\big]. \nonumber
\end{align}
The convergence $\EE[\left|J_t^{x,\epsilon}-J_t^x\right|]\rightarrow 0$, when $\epsilon$ tends to 0, is thus a crucial step of the proof and we first establish this result.  The detailed proof of the proposition is postponed,  after having obtained some dedicated estimates on processes  $J^{x,\epsilon}(\lambda)$ and $J^{x}(\lambda)$ in the following subsection.

\subsubsection{Preliminary estimations}
The process $J^{x,\epsilon}(\lambda)$ given by $J_t^{x,\epsilon}(\lambda)=\tfrac{X_t^{x+\epsilon}(\lambda)-X_t^x(\lambda)}{\epsilon}$ satisfies the  linear SDE
\begin{align}\label{eq:Jxe SDE}
\mbox{$ J_t^{x,\epsilon}(\lambda) =1
+\int_0^tJ_s^{x,\epsilon}(\lambda)\big( \xi_s^\epsilon\, ds
+  \sigma\,\psi_s^\epsilon \, dW_s
+ \lambda \sigma^2 \phi_s^\epsilon \,ds\big), $}
\end{align}
where we have defined 
\begin{align*}
\mbox{$\xi_t^\epsilon:=\int_0^1b'(X_t^{x}(\lambda)+\theta\epsilon\;J_t^{x,\epsilon}(\lambda))d\theta,\qquad $} &
\mbox{$\psi_t^\epsilon:=\int_0^1(X_t^{x}(\lambda)+\theta\epsilon\;J_t^{x,\epsilon}(\lambda))^{\alpha-1}d\theta,$}\\
\mbox{and}\qquad &\mbox{$\phi_t^\epsilon:=\int_0^1(X_t^{x}(\lambda)+\theta\epsilon\;J_t^{x,\epsilon}(\lambda))^{2(\alpha-1)}d\theta.$}
\end{align*}

As $J_t^{x,\epsilon}(\lambda)>0$ a.s., these auxiliary processes may also write
\begin{align*}
\xi_t^\epsilon = \tfrac{b\left(X_t^{x+\epsilon}(\lambda)\right)-b\left(X_t^{x}(\lambda)\right)}{X_t^{x+\epsilon}(\lambda)-X_t^{x}(\lambda)},
\quad
\psi_t^\epsilon = \tfrac{(X_t^{x+\epsilon}(\lambda))^{\alpha}-(X_t^{x}(\lambda))^{\alpha}}{X_t^{x+\epsilon}(\lambda)-X_t^{x}(\lambda)},
\quad
\phi_t^\epsilon =\tfrac{(X_t^{x+\epsilon}(\lambda))^{2\alpha-1}-(X_t^{x}(\lambda))^{2\alpha-1}}{X_t^{x+\epsilon}(\lambda)-X_t^{x}(\lambda)}.
\end{align*}

The two following lemmas, whose proofs are postponed to  Appendix \ref{sec:app:proof_moments}, assert the finiteness of the moments of the processes $J^{x,\epsilon}(\lambda)$ and $J^{x}(\lambda)$ as well as the $L^p$ continuity of $x\mapsto X^{x}_t(\lambda)$ (Lemma \ref{lem:Jxepsilon_Jx_bound}) and the $L^2$ continuity of $x\mapsto J^{x}_t(\lambda)$ in Lemma \ref{lem:ConvergenceJ}.

\begin{lem}\label{lem:Jxepsilon_Jx_bound}
Assume \ref{C2:H1}, \ref{C2:H_loclip}, \ref{C2:H_poly'} and
\begin{equation}\label{eq:Lambda2}
\max\{\tfrac12,2(\alpha-1)\} \leq \tfrac12+\tfrac{B^\lambda_2}{\sigma^2}.\tag{\bf $\Lambda$.2}
\end{equation}
Then, for all $q>0$ such that $(2\alpha-1)\lambda+\tfrac{\alpha^2}{2}(q-1)\leq\tfrac{B_2'}{\sigma^2}$, the processes $J^{x,\epsilon}(\lambda)$ and $J^x(\lambda)$,  respective solutions to \eqref{eq:Jxe SDE} and \eqref{eq:derivativeSDE}, satisfy
\begin{align}\label{eq:Jx alpha2 and eq:Jxepsilon alpha}
\sup_{t\in[0,T]}\EE\big[(J_t^{x}(\lambda))^q \big]+ \sup_{t\in[0,T]}\EE\big[(J_t^{x,\epsilon}(\lambda))^q \big]\leq 2\exp\{qB_1'T\},
\end{align}
and
\begin{equation}\label{eq:Lalphaconvergence q}
\lim_{\epsilon \rightarrow 0} \sup_{t\in[0,T]}\EE\big[|X_t^{x+\epsilon}(\lambda)-X_t^x(\lambda)|^q\big] = 0.
\end{equation}
\end{lem}

\begin{lem}\label{lem:ConvergenceJ}
Assume \ref{C2:H1},  \ref{C2:H_loclip'}, \ref{C2:H_poly'}  and
\begin{equation}\label{eq:Lambda3}
\max\{\tfrac12,~2(\alpha-1),\overline{\gamma}_{(1)}\}\leq\tfrac12 + \tfrac{B^\lambda_2}{\sigma^2},~~~\mbox{and }~~\;(2\alpha-1)\lambda+\tfrac{\alpha^2}{2} 5 \leq \tfrac{B_2'}{\sigma^2},\tag{\bf $\Lambda$.3}
\end{equation}
for the constants  $B_2,B_2',\sigma,\alpha,\overline{\gamma}_{(1)}$ as  in \ref{C2:H_loclip'} and \ref{C2:H_poly'}. Then, for all $t\in[0,T]$, \[\lim_{\epsilon\rightarrow0}\EE\big[\left|J_t^x(\lambda)-J_t^{x,\epsilon}(\lambda)\right|^2\big]=0.\]
\end{lem}

\subsubsection{Proof of Proposition \ref{prop:interchang}}
To simplify notation we omit the dependence on $\lambda$ in the processes and
 prove the result for $h\equiv0$. In order to prove the interchange between expectation and operator $\frac{\partial}{\partial x}$ we must show the equality
\begin{equation}\label{eq:limitInterchange}
\lim_{\epsilon\rightarrow0}\tfrac1\epsilon\EE\big[\Phi(X_t^{x+\epsilon})\exp\{Y_t^{x+\epsilon}\}-\Phi(X_t^x)\exp\{Y_t^{x}\}\big]=\EE\big[(\Phi'(X_t^x)J_t^x+\Phi(X_t^x)\tfrac{d Y_t^x}{dx})\exp\{Y_t^{x}\}\big],
\end{equation}
introducing the process $(Y_t^x:=\int_0^t g(X_s^x)ds;\; 0\leq t\leq T)$,
with derivative $\frac{d Y_t^x}{dx}=\int_0^t g'(X_s^x) J_s^x ds$.
Following the decomposition \eqref{eq:decompo}, we rewrite the difference
\begin{align}\label{eq:ABepsilon}
&\EE\big[\big|\frac{1}{\epsilon}( \Phi(X_t^{x+\epsilon})\exp\big\{Y_t^{x+\epsilon}\big\}-\Phi(X_t^x)\exp\big\{Y_t^{x}\big\}) -\left(\Phi'(X_t^x)J_t^x+\Phi(X_t^x)\tfrac{d Y_t^x}{dx}\right)\exp\big\{Y_t^{x}\big\}\big|\big] \nonumber \\
&\leq \EE\big[\big|(J_t^{x,\epsilon}\exp\big\{Y_t^{x+\epsilon}\big\}-J_t^x\exp\big\{Y_t^{x}\big\})\mbox{$\int_0^1$} \Phi'\left(X_t^x+\theta\epsilon J_t^{x,\epsilon}\right)d\theta\big|\big] \nonumber \\
&\quad\quad+\EE\big[\big|J_t^x\exp\big\{Y_t^{x}\big\}\mbox{$\int_0^1$}\left(\Phi'\left(X_t^x+\theta\epsilon J_t^{x,\epsilon}\right)-\Phi'(X_t^x)\right)d\theta\big|\big] \nonumber \\
&\quad\quad+\EE\big[\big|\Phi(X_t^x)\left(\tfrac{Y_t^{x+\epsilon}-Y_t^x}\epsilon-\tfrac{d Y_t^x}{dx}\right)\mbox{$\int_0^1$}\exp\big\{Y_t^x+\theta(Y_t^{x+\epsilon}-Y_t^{x})\big\}d\theta\big|\big] \nonumber \\
&\quad\quad+\EE\big[\big|\Phi(X_t^x)\tfrac{dY_t^{x}}{dx} \mbox{$\int_0^1$}
\left(\exp\big\{Y_t^x+\theta(Y_t^{x+\epsilon}-Y_t^{x})\big\}-\exp\{Y_t^x\}\right)d\theta\big|\big] \nonumber \\
& \quad : =  \EE\big[ |A^\epsilon_t|\big] + \EE\big[ |B^\epsilon_t|\big] + \EE \big[|C^\epsilon_t|\big] + \EE \big[|D^\epsilon_t|\big],
\end{align}
and we analyze separately the limit, when $\epsilon$ tends to 0, of each term in the right-hand side of \eqref{eq:ABepsilon}.
Notice that $g$ being bounded from above, all the exponential terms above are bounded.

Since $|\Phi'|$ and $\EE[|J_t^x\exp\{Y_t^x\}|]$ are bounded, the second term  $\EE[|B^\epsilon_t|]$ is uniformly integrable.
Moreover, $\Phi'$ is continuous,  and, under the hypotheses \eqref{eq:Lambda1}, according to Lemma \ref{lem:Jxepsilon_Jx_bound},  $X_t^{x+\epsilon}$ converges in $L^4$ and then in probability to $X_t^x$ when $\epsilon$ tends to 0. Therefore, we can apply the Lebesgue dominated convergence Theorem, obtaining (up to a subsequence still denoted by $\epsilon$) that $\lim_{\epsilon\rightarrow0} \EE[|B^\epsilon_t|]=0$.

Similarly, since $|\Phi|$ and $\EE[|\tfrac{d Y_t^x}{dx}|]\leq \EE[\int_0^t|g'(X_s^x)J_s^x|ds]$ are  bounded according to $\EE[|(J_s^x)^2|]$ and $\EE[|(X_s^x)^{2\rho_2}|]$, the sequence $D^\epsilon_t$ is uniformly integrable. Then, by the convergence in probability of $Y_t^{x+\epsilon}$ towards $Y_t^x$ obtained from the following bound
\begin{align*}
\EE\big[|Y_t^{x+\epsilon}-Y_t^x|\big] &= \mbox{$\epsilon\,\EE\big[\big|\int_0^t J_s^{x,\epsilon}\int_0^1g'(X_s^x+\epsilon\theta J_s^{x,\epsilon})d\theta ds\big|\big]$}\\
&\leq\mbox{$ C\epsilon\int_0^t\EE^{1/2}\big[|J_s^{x,\epsilon}|^2\big]~\EE^{1/2}\big[1+|X_s^{x+\epsilon}|^{2\rho_2}+|X_s^{x}|^{-2\rho_3}\big]ds$},
\end{align*}
 we obtain (again up to a subsequence) that $\lim_{\epsilon\rightarrow0} \EE[|D^\epsilon_t|]=0$.

For $\EE [|A^\epsilon_t|]$,  we use Cauchy-Schwartz inequality, from which we get
\begin{align*}
\EE[| A^\epsilon_t|]
&\leq\EE\big[\left|J_t^{x,\epsilon}-J_t^x\right|
\exp\big\{Y_t^{x+\epsilon}\big\}\mbox{$\int_0^1$}\left|\Phi'\left(X_t^x+\theta\epsilon J_t^{x,\epsilon}\right)\right|
d\theta\big]\\
&\qquad+\EE\brac{J_t^x\left|\exp\big\{Y_t^{x+\epsilon}\big\}-\exp\big\{Y_t^{x}\big\}\right|
\mbox{$\int_0^1$} \left|\Phi'\left(X_t^x+\theta\epsilon J_t^{x,\epsilon}\right)\right| d\theta}\\
&\leq C\left(\EE\big[\mbox{$\int_0^1$} \left|\Phi'\left(X_t^x+\theta\epsilon J_t^{x,\epsilon}\right)\right|^2 d\theta\big]\right)^{1/2}\left(\EE\big[\left|J_t^{x,\epsilon}-J_t^x\right|^2\big]\right)^{1/2}\\
&\qquad +\left(\EE\big[\mbox{$\int_0^1$} \left|\Phi'\left(X_t^x+\theta\epsilon J_t^{x,\epsilon}\right)\right|^2 d\theta\big]\right)^{1/2}\left(\EE\brac{\left(J_t^x\exp\big\{ Y_t^{x+\epsilon}\big\}-J_t^x\exp\big\{Y_t^{x}\big\}\right)^2}\right)^{1/2}\\
&\leq C\EE^{1/2}\big[\left|J_t^{x,\epsilon}-J_t^x\right|^2\big]
+C\EE^{1/2}\big[\left(J_t^x\exp\big\{ Y_t^{x+\epsilon}\big\}-J_t^x\exp\big\{Y_t^{x}\big\}\right)^2\big].
\end{align*}
Therefore, by Lemma \ref{lem:ConvergenceJ} and Lebesgue's Theorem,  $\EE [|A^\epsilon_t|]$ converges to 0 when $\epsilon$ tends to~0.

Finally, for $\EE[|C^\epsilon_t|]$,
\begin{align*}
&\EE[|C^\epsilon_t|]  \leq C \EE\big[\big|\tfrac{Y_t^{x+\epsilon}-Y_t^x}\epsilon-\tfrac{d Y_t^x}{dx}\big|\big]\leq C\;\mbox{$\int_0^t$}\EE\big[\left|\tfrac{g(X_s^{x+\epsilon})-g(X_s^x)}\epsilon-g'(X_s^x)J_s^x\right|\big]ds\\
&\leq C\mbox{$\int_0^t$}\left\{\EE\big[\big|(J_s^{x,\epsilon}-J_s^x)\mbox{$\int_0^1$} g'(X_s^x+\epsilon\theta J_s^{x,\epsilon})d\theta\big|\big]+\EE\big[\big|J_s^x\mbox{$\int_0^1$}\left(g'(X_s^x+\epsilon\theta J_s^{x,\epsilon})-g'(X_s^x)\right)d\theta\big|\big]\right\}ds,
\end{align*}
and we conclude that $\lim_{\epsilon\rightarrow0} \EE[|C^\epsilon_t|]=0$ with similar arguments to those used for $A^\epsilon_t$.

Coming back to \eqref{eq:ABepsilon}, we obtain the convergence to zero, up to a subsequence $\{\epsilon_k\}_{k\geq1}$, of its right-hand side. Therefore, by uniqueness of the limit, we deduce the convergence in \eqref{eq:limitInterchange}.

The condition \eqref{eq:Lambda1} makes the intersection of Conditions \eqref{eq:Lambda2} and \eqref{eq:Lambda3} with the highest $p$-moment order needed to obtain the convergence in \eqref{eq:limitInterchange}. When $h$ is not reduce to zero,  it is sufficient for the proof  to control  
$$\sup_{0\leq s,t\leq T}\EE[|g'(X_t^x)h(X_s^x)|^2 + |h'(X_t^x)|^2],$$
 by the moments  $\sup_{0\leq t\leq T}\EE[|X_t^x|^{2(\rho_2+\rho_4)} + |X_t^x|^{2\rho_0}]$ which are finite when
$\max\{\rho_0, \rho_2 + \rho_4\}\leq\frac12 + \frac{B^\lambda_2}{\sigma^2}$.

\subsection{Change of measure}\label{sec:changemeasure}
Considering a generic expression coming from the application of Proposition \ref{prop:interchang},  typically of the form
\begin{align*}
\mbox{$\frac{\partial v}{\partial x} (t,x)=\EE\big[\exp\{\int_0^t g(X_s^x(\lambda)) ds\} \ 6\Phi'(X_t^x(\lambda))J_t^x(\lambda)\big],$}
\end{align*}
we introduce the change of probability measure that allows to remove the term $J_t^x(\lambda)$ in the expression above.

Let us consider the process $(B_t^{\lambda+\alpha};\;0\leq t\leq T)$ defined as  $B_t^{\lambda+\alpha}=W_t-\alpha \sigma\int_0^t(X_s^x(\lambda))^{\alpha-1}ds$.
Then, using Lemma \ref{lemJ} below and Girsanov's Theorem,  we can construct the probability measure $\mathbb{Q}^{\lambda+\alpha}$ under which $(B_t^{\lambda+\alpha};\;0\leq t\leq T)$ is a standard Brownian motion, by introducing the Radon-Nikodyn density
\begin{equation}\label{eq:densityQ}
\tfrac{d\QQ^{\lambda+\alpha}}{d\PP}\Big|_{\Ff_t}=\tfrac1{\mathcal{Z}_t^{(\lambda,\lambda+\alpha)}},\quad
\mbox{$\mathcal{Z}_t^{(\lambda,\lambda+\alpha)}:=\exp\{-\tfrac{\alpha^2\sigma^2}{2}\int_0^t{\left(X_s^x(\lambda)\right)^{2(\alpha-1)}}ds-\alpha \sigma\int_0^t{X_s^x(\lambda)}^{\alpha-1}dB^{\lambda+\alpha}_s\}.$}
\end{equation}
Lemma \ref{lemJ} below gives a sufficient condition for the process $\mathcal{Z}^{(\lambda,\lambda+\alpha)}$ to be martingale for a given $\lambda\geq0$.
From the explicit form of the process $J_t^x(\lambda)$ in \eqref{eq:expoJlambda}, we recognize
\begin{align*}
J_t^x(\lambda)= &\mbox{$\exp\{\int_0^t\big[b'(X_s^x(\lambda))+\lambda \sigma^2(2\alpha-1)X_s^x(\lambda)^{2(\alpha-1)}\big] ds\}$}\\
&\qquad\times \mbox{$\exp\{\int_0^t\tfrac{\alpha^2 \sigma^2}2 X_s^x(\lambda)^{2(\alpha-1)}ds +  \alpha \sigma\int_0^t X_s^x(\lambda)^{\alpha-1}dB_s^{\lambda + \alpha}\}.$}
\end{align*}
Hence, by \ref{C2:H_poly'} and \eqref{eq:Lambda1},
\begin{align*}
\mbox{$ J_{t}^x(\lambda)\,\mathcal{Z}_{t}^{(\lambda,\lambda+\alpha)} = \exp \{ \int_0^{t}\big[b'(X_s^x(\lambda)) +\lambda\sigma^2(2\alpha -1)X_s^x(\lambda)^{2(\alpha-1)}\big] ds\} \leq \exp\big\{B_1'~T\big\}.$}
\end{align*}
Moreover, we can easily check the identity  $
\textit{Law}^{\QQ^{\lambda+\alpha}}\left(X^x\left(\lambda\right)\right)=\textit{Law}^{\PP}\left(X^x\left(\lambda+\alpha\right)\right)$,
so that  $\tfrac{\partial v}{\partial x}$ can be rewrite
\begin{align*}
&\tfrac{\partial v}{\partial x} (t,x)=\\
&\quad\mbox{$\EE\big[\exp\{\int_0^t \big[g(X_s^x(\lambda+\alpha))  + b'(X_s^x(\lambda +\alpha)) +\lambda\sigma^2(2\alpha -1)X_s^x(\lambda+\alpha)^{2(\alpha-1)}\big]  ds\}\Phi'(X_t^x(\lambda+\alpha))\big]$}.
\end{align*}
The following lemma is a direct consequence of the Novikov's criterion whose fulfillment is ensured by applying Corollary~\ref{cor:flux}.
\begin{lem}\label{lemJ}
Assume \ref{C2:H1}, \ref{C2:H_loclip} and \ref{C2:H_poly'}. Assume in addition that $B_2,\alpha$ and $\sigma$ in \ref{C2:H_poly'} satisfy
\begin{align}\label{hypo:constraint}
\begin{aligned}
\mbox{ if } b(0)=0,&~~ \quad  \alpha \leq \tfrac{1}{2} + \tfrac{B_2^\lambda}{\sigma^2}, \\
\mbox{ if } b(0)>0,&~~\quad  \tfrac32< \alpha, \quad\mbox{ and }\quad  {\alpha^2} \leq 2 B_2^\lambda.
\end{aligned}
\end{align}
Then  the process $(M_t^x(\lambda);\;0\leq t\leq T)$ defined by
\begin{equation*}
\mbox{$M_t^x(\lambda)=\exp\{\alpha \sigma\int_0^t(X_s^x(\lambda))^{\alpha-1}dW_s-\tfrac{\alpha^2 \sigma^2}2\int_0^t(X_s^x(\lambda))^{2(\alpha-1)}ds\}$}
\end{equation*}
is a $\mathbb{P}$- martingale.
\end{lem}

\subsection{Proof of Proposition \ref{prop:KolPDE}} 	

Let us first note that the hypotheses considered in Proposition \ref{prop:KolPDE}, in particular \ref{C2:H5}, allow to apply Proposition \ref{prop:interchang} up to  $\lambda=3\alpha$.

\paragraph{Estimates on $u$ and $\tfrac{\partial u}{\partial t}$. }
The uniform boundedness of $u(t,x)=\EE[f(X^x_{T-t}(0))]$
is an immediate consequence of the bounded of $f$.  Applying It\^o's formula and since $X^x(0)$ has finite $2\alpha$-th moment,
\begin{align*}
u(t,x)&=f(x)+ \,\int_0^{T-t}\EE[ b(\qs(0)) \, f'(\qs(0))]\, ds \\
& \quad+  \sigma \,\EE[\int_0^{T-t}  (\qs(0))^{\alpha}\, f'(\qs(0))dW_s]+\frac{\sigma^2}{2} \int_0^{T-t}\EE[(\qs(0))^{2\alpha}f''(\qs(0))]ds\\
&=f(x)+\int_0^{T-t}\EE[ b(\qs(0)) \, f'(\qs(0))]\, ds +\frac{\sigma^2}{2} \int_0^{T-t}\EE[(\qs(0))^{2\alpha}f''(\qs(0))]ds.
\end{align*}
From this expression, we can deduce that $\tfrac{\partial u}{\partial t}$ is continuous in $[0,T]\times \RR^+$ with
\begin{align*}
\left| \tfrac{\partial u}{\partial t}\right|(t,x)
&\leq  C\;\left(\EE[|b\left(X_{T-t}^x(0)\right)|]+\EE[|X_{T-t}^x(0)|^{2\alpha}]\right),
\end{align*}
where $C$ is a positive constant depending on $\alpha, \sigma, b,\|f^{(i)}\|_\infty$ for $i=0,1,2$. The $2(\alpha-1)$-locally Lipschitz continuity of the drift $b$ gives us
\[\EE[|b(X_{T-t}^x(0))|]\leq C(1+\EE[|X_{T-t}^x(0)|^{2\alpha-1}]).\]
Applying Corollary \ref{cor:flux} for the control of $\sup_{t \in[0,T]}\EE|X_t^x|^{2\alpha}]$ granted by the condition $\alpha \leq \frac{1}{2} +\frac{B_2}{\sigma^2}$ in \ref{C2:H5}, we obtain
\begin{equation*}
\mbox{$ \sup_{t\in [0,T]}\left|\tfrac{\partial u}{\partial t}\right| (t,x)  \leq C\, \left(1+x^{2\alpha}\right).$}
\end{equation*}
\paragraph{Estimates on $\tfrac{\partial u}{\partial x}$ and $\tfrac{\partial^2 u}{\partial x^2}$. } The differentiability up to order $4$ of $x\mapsto u(t,x)$ will rely on the iterative use of Proposition \ref{prop:interchang} and on rewriting (whenever its is necessary) the function $\frac{\partial^{j} u}{\partial x^{j}}(t,x)$, $j=0,1,2,3$, as
\begin{align}\label{eq:DerivativeRewritten}
 \tfrac{\partial^{j} u}{\partial x^{j}}(t,x) &= \mbox{$\EE\big[f_j(X_{T-t}^x(j\alpha))\exp\big\{\int_0^{T-t}g_j(X_s^x(j\alpha))ds\big\}\big]$}\nonumber\\
&\quad+\mbox{$\int_0^{T-t}\EE\big[h_j(s,X_{s}^x(j\alpha))\exp\big\{\int_0^{s}g_j(X_r^x(j\alpha))dr\big\}\big]ds$},
\end{align}
for some continuous differentiable functions $g_j$ and $h_j$ with locally bounded spatial derivatives in $[0,+\infty)$, with $g_j$ bounded from above, and  $f_j$ some bounded continuously differentiable functions with bounded derivative.

In order to prove the identity \eqref{eq:du/dx2} for $\tfrac{\partial u}{\partial x}$, we apply Proposition \ref{prop:interchang} for $f_0=f$ and $g_0=h_0\equiv0$ and a first change of measure.  So we need  Condition \eqref{eq:Lambda1} to be satisfied for $\rho_i=0$ and $\lambda =0$ and we need also the hypotheses of Lemma \ref{lemJ} satisfied for  $\lambda =0$.  From \eqref{eq:du/dx2}, we immediately get that
\begin{align}\label{eq:upper_first}
\left|\tfrac{\partial u}{\partial x}\right|(t,x) = \left|\EE[f'(X_{T-t}^x(0))J_{T-t}^x(0)]\right|\leq C \|f'\|_\infty.
\end{align}

Next from \eqref{eq:du/dx2},  we identify the form \eqref{eq:DerivativeRewritten} with $f_1=f'$, $g_1=b'(x)$ (bounded from above, and with $|g'_1(x)|\leq C(1 + |x|^{\overline{\gamma}_{(2)} +1} +|x|^{-\underline{\gamma}_{(2)}})$ and $h(x)\equiv0$. Applying Proposition \ref{prop:interchang} again  (with \eqref{eq:Lambda1} to be satisfied for $\rho_2= {\overline{\gamma}_{(2)} +1}$ and $\lambda =\alpha$) we obtain that $\frac{\partial u}{\partial x}$ is continuously differentiable in $x$ with derivative given by
\begin{align}\label{eq:D2Before}
\frac{\partial^2 u}{\partial x^2}(t,x)&= \mbox{$ \EE\big[\exp\big\{\int_0^{T-t}b^{\prime}(X_s^x(\alpha))ds\big\}f^{(2)}(X_{T-t}^x(\alpha))J_{T-t}^x(\alpha)\big]$} \\
&\quad+ \mbox{$ \EE\big[\exp\big\{\int_0^{T-t}b^{\prime}(X_s^x(\alpha))ds\big\}f^{\prime}(X_{T-t}^x(\alpha))\int_0^{T-t}b^{(2)}(X_s^x(\alpha))J_{s}^x(\alpha)ds\big].$} \nonumber
\end{align}
Notice that by means of the Markov property and time homogeneity of the process $(X_s^x(\alpha);0\leq s \leq T-t)$ we have
\begin{align}\label{eq:MarkovProp}
\begin{aligned}
& \mbox{$ \EE\big[f^{\prime}(X_{T-t}^x(\alpha))\exp\big\{\int_s^{T-t}b^{\prime}(X_r^x(\alpha))dr\big\}\big| {\Ff_s}\big]$} \\
&= \mbox{$ \EE\big[f^{\prime}(X_{T-t-s}^y(\alpha))\exp\big\{\int_0^{T-t-s}b^{\prime}(X_r^y(\alpha))dr\big\}\big]\big|_{y=X_s^x(\alpha)}=\frac{\partial u}{\partial x}(t+s,X_s^x(\alpha)),$}
\end{aligned}
\end{align}
where the last equality is obtained from \eqref{eq:du/dx2}. Then, we get
\begin{align*}
\begin{array}{l}
\int_0^{T-t}\EE\big[\exp\big\{\int_0^{T-t}b^{\prime}(X_r^x(\alpha))dr\big\}f^{\prime}(X_{T-t}^x(\alpha))b^{(2)}(X_s^x(\alpha))J_{s}^x(\alpha)\big]ds\\
=\int_0^{T-t}\EE\big[\exp\big\{\int_0^{s}b^{\prime}(X_r^x(\alpha))dr\big\}b^{(2)}(X_s^x(\alpha))\frac{\partial u}{\partial x}(t+s,X_s^x(\alpha))J_{s}^x(\alpha)\big]ds.
\end{array}
\end{align*}
Substituting the last equality in \eqref{eq:D2Before},
\begin{align}\label{eq:D2After}
\frac{\partial^2 u}{\partial x^2}(t,x)&=\EE\big[f^{(2)}(X_{T-t}^x(\alpha))\exp\big\{\int_0^{T-t}b^{\prime}(X_s^x(\alpha))ds\big\}J_{T-t}^x(\alpha)\big]\\
&\quad+\int_0^{T-t}\EE\big[\exp\big\{\int_0^{s}b^{\prime}(X_r^x(\alpha))dr\big\}b^{(2)}(X_s^x(\alpha))\frac{\partial u}{\partial x}(t+s,X_s^x(\alpha))J_{s}^x(\alpha)\big]ds. \nonumber
\end{align}
Introducing the change of measure $\tfrac{d \mathbb{Q}^{2\alpha}}{d\mathbb{P}}\big|_{\mathcal{F}_t}:= \frac{1}{\mathcal{Z}_t^{(\alpha,2\alpha )}}$ where  $\mathcal{Z}^{(\alpha,2\alpha)}$ is  defined in \eqref{eq:densityQ} (under the conditions \eqref{hypo:constraint}  and  \eqref{eq:Lambda1} applied for $\lambda =\alpha$), we can observe that (assuming  $2B'_2 \geq \alpha (2\alpha -1) \sigma^2$  in \eqref{eq:Lambda1})
\begin{align}\label{eq:D2After_bis}
\mbox{$
\exp\{\int_0^{t}b^{\prime}(X_s^x(\alpha))ds\}J_{t}^x(\alpha)\,\mathcal{Z}_{t}^{(\alpha,2\alpha)} = \exp\{ \int_0^{t}\big(2b'(X_s^x(\alpha)) +\alpha\sigma^2(2\alpha -1)X_s^x(\alpha)^{2(\alpha-1)}\big) ds\} \leq C.$}
\end{align}
So changing the measure in \eqref{eq:D2After}, with the observation that  $\textit{Law}^{\mathbb{Q}^{2\alpha }}(X^x(\alpha))=\textit{Law}^\mathbb{P}(X^x(2\alpha ))$, and by the boundedness of the functions $\frac{\partial u}{\partial x}$ and $f^{(2)}$  we obtain
\begin{align*}
\mbox{$
\big|\tfrac{\partial^2 u}{\partial x^2}\big|(t,x)\leq C\; \EE\big[1+\int_0^{T-t}|b^{(2)}(X_s^x(2\alpha))|ds\big]\leq C\big(1+\sup_{s\in[0,T]}\EE[|X_s^x(2\alpha)|^{\overline{\gamma}_{(2)}+1}]$}\\
\mbox{$+\sup_{s\in[0,T]}\EE[|X_s^x(2\alpha)|^{-\underline{\gamma}_{(2)}}]\big).$}
\end{align*}
We then apply Corollary \ref{cor:flux} to conclude on the estimate for the second derivative:
\begin{align}\label{eq:upper_second}
\big|\tfrac{\partial^2 u}{\partial x^2}\big|(t,x)\leq~C(1+x^{\overline{\gamma}_{(2)}+1}+x^{-\underline{\gamma}_{(2)}}),
\end{align}
under the condition that $\overline{\gamma}_{(2)}+1 \leq 1 + \frac{2B_2^{2\alpha}}{\sigma^2}$.

We end this proof by iterating the derivative estimations. This (technical) step is postponed to Appendix  \ref{sec:app:proof_FK}.

\section{Proof of  Proposition \ref{prop:Weak rate} }\label{sec:Proofs}

Introducing the notation
\begin{align*}
\overline{b}(t,x,y)=\left(x-b(0)\dt\right)\tfrac{\left(b(y)-b(0)\right)}{y} \quad \mbox{and} \quad  \overline{\sigma}(t,x,y)=\sigma\left(x-b(0)\dt\right) y^{\alpha-1},
\end{align*}
for which we have
\[\overline{b}(\eta(t),\overline{X}_{\eta(t)},\overline{X}_{\eta(t)})= b(\overline{X}_{\eta(t)}) - b(0),\qquad\overline{\sigma}(\eta(t),\overline{X}_{\eta(t)},\overline{X}_{\eta(t)})= \sigma \overline{X}_{\eta(t)}^\alpha,\]
we can rewrite the dynamics \eqref{eq:ContExpscheme} as
\begin{equation}\label{eq:ContExpschemebis}
d\overline{X}_t=  \big(b(0) +\overline{b}(t,\overline{X}_t,\overline{X}_{\eta(t)})  \big)   dt   +  \overline{\sigma}(t,\overline{X}_t,\overline{X}_{\eta(t)})dW_t,~\overline{X}_0=x.
\end{equation}
We associate to it, the differential operator
\begin{align*}
\Ld_{(t, (y,\eta(t)))} f(t,x) = \big\{\tfrac{\partial f}{\partial t}+\left(b(0)+\overline{b}\right)\tfrac{\partial f}{\partial x}+ \tfrac12\overline{\sigma}^2\tfrac{\partial^2 f}{\partial x^2}\big\} (t,x,y).
\end{align*}
Then, applying It\^{o}'s formula to the $C^{1,4}$ function $u$ along $\overline{X}$ in the time interval $[0,T]$, we obtain
\begin{align*}
\begin{aligned}
& \EE[f(X_T)-f(\overline{X}_{T})]=u(0,x)-\EE [u(T,\overline{X}_T)]
=\sum_{k=1}^N\EE[u(t_{k-1},\overline{X}_{t_{k-1}})-u(t_{k},\overline{X}_{t_{k}})] \\
& =-\sum_{k=1}^N\EE\big[\int_{t_{k-1}}^{t_k}\Ld_{(s, (\overline{X}_{\eta(s)},\eta(s)))} u (s,\overline{X}_s,\overline{X}_{\eta(s)}) ds\big] -\sum_{k=1}^N\EE\big[\int_{t_{k-1}}^{t_k}\overline{\sigma}(s, \overline{X}_{\eta(s)},\overline{X}_\eta(s)) \frac{\partial u}{\partial x}(s,\overline{X}_s)dW_s\big].
\end{aligned}
\end{align*}
Lemma \ref{leq:Schememoments} under  \ref{C2:H5} allows to control the $2p$-th moments of the exp-ES process $\overline{X}_t$ up to the order $2p :=6\alpha +2\overline{\overline{\beta}} \vee (\overline{\overline{\beta}} +2\alpha)$. By Proposition \ref{prop:KolPDE} we have for ach $k=1,\ldots,N$
\begin{equation*}
\mbox{$
\EE\big[\int_{t_{k-1}}^{t_{k}} \big(\overline{\sigma}\frac{\partial u}{\partial x}\big)^2(s,\overline{X}_s,\overline{X}_{\eta(s)}) ds\big]
\leq C \sup_{t\in[0,T]}\EE[\overline{X}_{t}^{2\alpha}]<+\infty.
$}
\end{equation*}
Moreover, since $u$ is solution to the Cauchy problem \eqref{eq:KolPDE}, we decompose the error in two contributions: 
\begin{align}\label{eq:dif}
&\mbox{ $\EE[f(X_T)-f(\overline{X}_{T})]= \sum_{k=1}^N\EE\big[\int_{t_{k-1}}^{t_k} {\tfrac{\partial u}{\partial x}(s,\overline{X}_s)\big(b(\overline{X}_s)-b(0)-\overline{b}(s,\overline{X}_s,\overline{X}_{\eta(s)})\big)} $}\nonumber\\
&\hspace{1.8in}+{\tfrac{1}{2}\tfrac{\partial^2 u}{\partial x^2}(s,\overline{X}_s)\big(\sigma^2\overline{X}_s^{2\alpha}-\overline{\sigma}^2(s,\overline{X}_s,\overline{X}_{\eta(s)})\big)}\,ds\big] \nonumber \\
&\hspace{1.2in}  =:\sum_{k=1}^N\EE\big[\int_{t_{k-1}}^{t_k} \big( I_1(s,\eta(s)) + \tfrac{1}{2} I_2(s,\eta(s))\ \big) ds\big].
\end{align}

Notice that the functions $\overline{b}$ and $\overline{\sigma}$ are continuously differentiable with respect to $x$, and piecewise continuously differentiable with respect to $t$ on each subintervals  $[t_k, t_{k+1})$. These linear functions in $x$ and $t$ produce constant values as derivatives, only parametrized by $y$: for all $t\in(0,T)$, for all $(\theta,x,y)\in(\eta(t),t)\times(0,+\infty)\times(0,+\infty)$, 
\begin{align*}
\begin{array}{ll}
\dfrac{\partial \overline{b}}{\partial \theta}(\theta,x,y) = -b(0) \dfrac{\left(b(y)-b(0)\right)}{y}, & \qquad \dfrac{\partial\overline{\sigma}}{\partial \theta } (\theta,x,y)  = - \sigma b(0) y^{\alpha-1},\\
 \dfrac{\partial \overline{b}}{\partial x}(\theta,x,y) = \dfrac{\left(b(y)-b(0)\right)}{y} , & \qquad\dfrac{\partial\overline{\sigma}}{\partial x} (\theta,x,y) =\sigma y^{\alpha-1}.
 \end{array}
\end{align*}

Then, observing that $I_i(\eta(s),\eta(s))= 0$, we apply It\^{o}'s  formula a second time in the interval $[\eta(s),s]$ and we obtain the two following decompositions for each $s\in[t_{k-1},t_k]$ with $k=1,\ldots,N$:
\begin{align*}
\begin{array}{l}
\EE[I_1]=\EE\big[\int_{\eta(s)}^s\big\{\tfrac{\partial^2u}{\partial t \partial x} \left(b-b(0)-\overline{b}\right)+  \tfrac{\partial \overline{b}}{\partial t}\tfrac{\partial u}{ \partial x}\big\}(\theta,\overline{X}_\theta,\overline{X}_{\eta(\theta)})\, d\theta\big]\\
\quad \qquad +\EE\big[\int_{\eta(s)}^s \big\{\big(b(0)+\overline{b}\big)\big(\tfrac{\partial^2u}{\partial x^2} \left(b-b(0)-\overline{b}\right) +(b' - \tfrac{\partial \overline{b}}{\partial x})\tfrac{\partial u}{\partial x} \big)\big\}(\theta,\overline{X}_\theta, \overline{X}_{\eta(\theta)})\,d\theta\big]\\
\quad \qquad +\EE \big[\int_{\eta(s)}^s\big\{ \overline{\sigma}\big(\tfrac{\partial^2u}{\partial x^2} (b-b(0)-\overline{b})+ \big(b'-
\tfrac{\partial \overline{b}}{\partial x} \big)\tfrac{\partial u}{\partial x}\big)\big\}
(\theta,\overline{X}_\theta, \overline{X}_{\eta(\theta)})\, dW_\theta\big]\\
\qquad \qquad +\EE\big[\int_{\eta(s)}^s\big\{\tfrac{1}{2}\overline{\sigma}^2
\big(\tfrac{\partial^3u}{\partial x^3} (b-b(0)-\overline{b})+\big(b'-\tfrac{\partial \overline{b}}{\partial x} \big)\tfrac{\partial^2 u}{\partial x^2}+b^{(2)}\tfrac{\partial u}{\partial x}\big)\big\}(\theta,\overline{X}_\theta, \overline{X}_{\eta(\theta)})\;d\theta\big]\\
\qquad  =:\;\EE[I_1^{1}]+ \EE[I_1^2]+\EE[I_1^3]+ \EE[I_1^4],
\end{array}
\end{align*}
\begin{align*}
\begin{array}{l}
\EE\;\big[I_2\big]
 =
\EE\big[\int_{\eta(s)}^s\big\{\tfrac{\partial^3 u}{\partial t \partial x^2}(\sigma^2\overline{X}_\theta^{2\alpha}
- \overline{\sigma}^2)
+2\sigma b(0)   \overline{X}_{\eta(\theta)}^{\alpha-1}  \overline{\sigma} \tfrac{\partial^2 u}{\partial x^2}\big\}(\theta,\overline{X}_\theta, \overline{X}_{\eta(\theta)}) \,d\theta\big]\\
\qquad \qquad +\EE\big[\int_{\eta(s)}^s\left(b(0)+\overline{b}\right)
\tfrac{\partial}{\partial x}\big\{\tfrac{\partial^2u}{\partial x^2}
\big[\sigma^2 \overline{X}_\theta^{2\alpha}- \overline{\sigma}^2\big]
\big\}
(\theta,\overline{X}_\theta,\overline{X}_{\eta(\theta)}) d\theta\big]\\
\qquad \qquad +\EE\big[\int_{\eta(s)}^s  \big\{ \overline{\sigma} \frac{\partial}{\partial x}\big\{\frac{\partial^2u}{\partial x^2}
\big[\sigma^2 \overline{X}_\theta^{2\alpha}- \overline{\sigma}^2\big]
\big\}\big\}(\theta,\overline{X}_\theta, \overline{X}_{\eta(\theta)})\,dW_\theta\big]\\
\qquad \qquad +\tfrac{\sigma^2}2\EE\big[\int_{\eta(s)}^s\left(\overline{X}_\theta- b(0)\delta(\theta) \right)^2\overline{X}_{\eta(\theta)}^{2\alpha-2}\frac{\partial^2}{\partial x^2}\big\{\frac{\partial^2u}{\partial x^2}
\big[\sigma^2 \overline{X}_\theta^{2\alpha}- \overline{\sigma}^2\big]
\big\}(\theta,\overline{X}_\theta, \overline{X}_{\eta(\theta)})\;d\theta\big]\\
\qquad  =:\EE[I_2^{1}]+ \EE[I_2^2]+  \EE[I_2^3]+\frac{\sigma^2}2 \EE[I_2^4].
\end{array}
\end{align*}
We use again the backward Kolmogorov PDE \eqref{eq:KolPDE} to compute the time derivatives
 \begin{align*}%\label{eq:dudtdx}
\tfrac{\partial^2 u}{\partial t \partial x}(t,x)
&=-b'(x)\tfrac{\partial u}{\partial x}(t,x)-b(x)\tfrac{\partial^2 u}{\partial x^2}(t,x)-\sigma^2\alpha x^{2\alpha-1}\tfrac{\partial^2 u}{\partial x^2}(t,x)-\tfrac{\sigma^2}2x^{2\alpha}\tfrac{\partial^3 u}{\partial x^3}(t,x).
\\
\tfrac{\partial^3 u}{\partial t \partial x^2}(t,x)& =\big\{-b^{(2)}\tfrac{\partial u}{\partial x}-\left(b+2\sigma^2\alpha x^{2\alpha-1}\right)\tfrac{\partial^3 u}{\partial x^3}-\left(2b'+\sigma^2\alpha(2\alpha-1) x^{2\alpha-2}\right)\tfrac{\partial^2 u}{\partial x^2} -\tfrac{\sigma^2}2x^{2\alpha}
\tfrac{\partial^4 u}{\partial x^4}\big\}(t,x).
\end{align*}
From \ref{C2:H_loclip'} we have for all $x\geq0$ and $i=1,2,3,4$
\begin{align}\label{eq:bar_b}
\tfrac{|b(x)-b(0)|}{x}&\leq C(1+x^{2\alpha-2}).
\end{align}
\indent The key of the proof is to upper bound each $I^k_j$ by combining the estimates of the derivatives of $u$, the polynomial growth of the drift and diffusion coefficients and its derivatives,  with upper-bounds of  moments of the exp-ES process obtained from Lemma \ref{leq:Schememoments}.  By considering the following Young inequality  for arbitrary $m,n\geq 0$,
\begin{align}\label{eq:M_nm}
|\overline{X}_\theta|^m \ |\overline{X}_{\eta(\theta)}|^n&\leq C \sup_{r\in[0,T]}|\overline{X}_r|^{m+n},
\end{align}
we get
\begin{equation}\label{eq:I_j^k}
\EE [|I_j^k|]\leq C \big(1+\sup_{0\leq \theta\leq T}\EE[\overline{X}_\theta^{\beta_{k,j}}]\big)(s-\eta(s)), 
\end{equation}
for all $j=1,2$, $k=1,2,4$ and some $\beta_{k,j}\in[0,2\beta]$. Then, substituting \eqref{eq:I_j^k} in \eqref{eq:dif} we recover the rate of order one for the weak approximation error:
\begin{align*}
\left|\EE[f(X_T)-f(\overline{X}_{T})]\right|\leq C\;\sum_{k=1}^N\int_{t_{k-1}}^{t_k}(s-\eta(s))ds\leq C \;\Dt.
\end{align*}

We detail  the analysis of the first term $|I_1^1|$:
\begin{align*}
\EE[| I_1^1|]&\leq  \mbox{$ \EE\big[\int_{\eta(s)}^s\big\{\left|\tfrac{\partial^2u}{\partial t \partial x}\right| (\theta,\overline{X}_\theta)\left(|b(\overline{X}_{\theta})-b(0)|+|\overline{b}(\theta,\overline{X}_\theta,\overline{X}_{\eta(\theta)})|\right)$} \\
&\hspace{2in}\mbox{$ +b(0)\left|\tfrac{\partial u}{ \partial x}\right|(\theta,\overline{X}_\theta)\tfrac{|b(\overline{X}_{\eta(\theta)})-b(0)|}{\overline{X}_{\eta(\theta)}}\big\}d\theta\big],$}
\end{align*}
where, from Proposition \ref{prop:KolPDE} and \eqref{eq:poly_bi}, \eqref{eq:bar_b}, we have
\begin{align*}
\left|\tfrac{\partial u}{ \partial x}\right|(\theta,\overline{X}_\theta)\tfrac{|b(\overline{X}_{\eta(\theta)})-b(0)|}{\overline{X}_{\eta(\theta)}}
\leq C \tfrac{|b(\overline{X}_{\eta(\theta)})-b(0)|}{\overline{X}_{\eta(\theta)}} &\leq C\big(1+\overline{X}_{\eta(\theta)}^{2\alpha-2}\big),\\
|b(\overline{X}_{\theta})-b(0)|+|\overline{b}(\theta,\overline{X}_\theta,\overline{X}_{\eta(\theta)})|
&
%+++\leq|b(\overline{X}_{\theta})-b(0)|+\overline{X}_{\theta}\left|\frac{b(\overline{X}_{\eta(\theta)})-b(0)}{\overline{X}_{\eta(\theta)}}\right|
~
\leq C \overline{X}_{\theta}\big(1+\overline{X}_{\theta}^{2\alpha-2}+\overline{X}_{\eta(\theta)}^{2\alpha-2}\big),
\end{align*}
and
\begin{align*}
& \big|\tfrac{\partial^2u}{\partial t \partial x}\big|(t,x) \leq|b'| |\tfrac{\partial u}{\partial x}|(t,x)+\left(|b(x)|+\sigma^2\alpha x^{2\alpha-1}\right)|\tfrac{\partial^2 u}{\partial x^2}|(t,x)+\frac{\sigma^2}2x^{2\alpha}|\tfrac{\partial^3 u}{\partial x^3}|(t,x)\\
&\leq C\big\{1+x^{\overline{\gamma}_{(1)}+1}+x^{-\underline{\gamma}_{(1)}}+\left(1+ x^{2\alpha-1}\right)(1+x^{\overline{\gamma}_{(2)}+1}+x^{-\underline{\gamma}_{(2)}})+(x^{\overline{\beta} +2\alpha }+x^{2\alpha -\underline{\beta}})\big\},
\end{align*}
Remaining with the biggest $\pm$ exponent (using from Lemma \ref{lem:LocallyLipschitz} that  $\underline{\gamma}_{(i)}\leq \underline{\gamma}_{(i+1)}$, $i=1,2,3$),
\begin{align}\label{eq:exemple_control}
x\big|\tfrac{\partial^2u}{\partial t \partial x}\big|(t,x) &\leq
C\big\{x^{\overline{\beta} +2\alpha +1 }+x^{(1-\underline{\gamma}_{(2)})\wedge (2\alpha+1 -\underline{\beta})}\big\},
\end{align}
where $\overline{\beta}=2(\overline{\gamma}_{(2)}+1)\vee(1+\overline{\gamma}_{(3)})$ and  $\underline{\beta}=2\underline{\gamma}_{(2)} \vee\underline{\gamma}_{(3)}\vee (\underline{\gamma}_{(2)}+3-2\alpha)$.
Therefore, we get
\begin{align*}
\EE [|I_1^1|]&\leq C \EE\big[\int_{\eta(s)}^s\big(1+\overline{X}_{\eta(\theta)}^{2\alpha-2}+\overline{X}_{\theta}^{4\alpha-1+\overline{\beta}}+\overline{X}_{\eta(\theta)}^{2\alpha-2}\overline{X}_{\theta}^{2\alpha+1+\overline{\beta}}
+\overline{X}_{\theta}^{(1-\underline{\gamma}_{(2)})\wedge (2\alpha+1 -\underline{\beta})} \ \big)d\theta\big].
\end{align*}
Since we do not have a priori control of negative moments of $\overline{X}$, we must impose $\underline{\gamma}_{(2)}\leq 1$ and $\underline{\gamma}_{(3)}\leq 2\alpha+1$. Thus, from \eqref{eq:M_nm}, we obtain as desired 
\begin{align*}
\EE [|I_1^1|]&\leq  \mbox{$ C\big(1+\sup_{r\in[0,T]}\EE\big[\overline{X}_{r}^{4\alpha-1+\overline{\beta}}\big]\big)~(s-\eta(s)).$} 
\end{align*}

The remaining terms can be bounded similarly. Explicitly, we get  the following bounds
\begin{align*}
\EE[|I_1^3|^2]&\leq  \mbox{$ C\big(1+\sup_{r\in[0,T]}\EE\big[\overline{X}_{r}^{6\alpha+2\overline{\gamma}_{(2)}} + \overline{X}_{r}^{2-2\underline{\gamma}_{(1)}}\big] \big)~(s-\eta(s))$}  ,\\
\EE[|I_2^3|^2]&\leq  \mbox{$ C\big(1+\sup_{r\in[0,T]}\EE\big[\overline{X}_{r}^{6\alpha+2{\overline{\beta}}} + \overline{X}_{r}^{4-2\underline{\gamma}_{(2)}} + \overline{X}_{r}^{6-2\underline{\beta}}\big]\big)~(s-\eta(s)),$} 
\end{align*}
that ensure that the stochastic integrals are martingales, and
\begin{align*}
\EE[|I_1^2|]&\leq  \mbox{$ C\big(1+ \sup_{r\in[0,T]}\EE[ \overline{X}_{r}^{4\alpha-2+\overline{\gamma}_{(2)}} + \overline{X}_{r}^{-\underline{\gamma}_{(1)}}+ \overline{X}_{r}^{1-\underline{\gamma}_{(2)}}]\big)~(s-\eta(s)), $} \\
\EE[|I_1^4|]&\leq  \mbox{$ C\big(1+\sup_{r\in[0,T]}\EE [ \overline{X}_{r}^{4\alpha-1+\overline{\beta}} + \overline{X}_{r}^{2-\underline{\gamma}_{(2)}-\underline{\gamma}_{(1)}}+ \overline{X}_{r}^{3-\underline{\beta}} ] \big)~(s-\eta(s)),$} \\
\EE[|I_2^1|]&\leq  \mbox{$ C\big(1+\sup_{r\in[0,T]}\EE[ \overline{X}_{r}^{4\alpha+\overline{\overline{\beta}}} + \overline{X}_{r}^{1-\underline{\gamma}_{(2)}} + \overline{X}_{r}^{2\alpha+2-\underline{\underline{\beta}}} + \overline{X}_{r}^{2-\underline{\beta}} ] \big)~(s-\eta(s)),$} \\ 
\EE[|I_2^2|]&\leq  \mbox{$ C\big( 1+ \sup_{r\in[0,T]} \EE[\overline{X}_{r}^{4\alpha-1 +{\overline{\beta}}} + \overline{X}_{r}^{1-\underline{\gamma}_{(2)}} + \overline{X}_{r}^{2-\underline{\beta}} ] \big)~(s-\eta(s)),$} \\
\EE[|I_2^4|]&\leq  \mbox{$  C\big(1+\sup_{r\in[0,T]}\EE [\overline{X}_{r}^{4\alpha+\overline{\overline{\beta}}} + \overline{X}_{r}^{2-\underline{\gamma}_ \}
{(2)}} + \overline{X}_{r}^{4-\underline{\underline{\beta}}} +\overline{X}_{r}^{3-\underline{\beta}} ]\big)~(s-\eta(s)).$} 
\end{align*}
In the previous inequalities, we observe that \ref{C2:H4_poly} eliminates all possible negative moments in the $I^i_j$:  for $|I_1^2|$,  \ref{C2:H4_poly} imposes $\underline{\gamma}_{(1)}=0$. Similarly, for $|I_2^1|$ and $|I_2^4|$ and the definition of $\underline{\beta}$, $\underline{\underline{\beta}}$ in Proposition \ref{prop:KolPDE},  \ref{C2:H4_poly} imposes $\underline{\gamma}_{(2)}\leq 1,~\underline{\gamma}_{(3)}\leq 2$ and $\underline{\gamma}_{(4)}\leq 4$, respectively. Further, the terms $|I_2^1|$ and $|I_2^3|$ contain the highest moments to be controlled, $4\alpha +\overline{\overline{\beta}}$ and $6\alpha+2\overline{\beta}$, both are less than the moment order $6\alpha +2\overline{\overline{\beta}} \vee (\overline{\overline{\beta}} +2\alpha) $ imposed by \ref{C2:H5}.

\appendix
\section{Proof of Lemmas \ref{lem:Jxepsilon_Jx_bound} and \ref{lem:ConvergenceJ}}\label{sec:app:proof_moments}
To simplify notation,  we omit the dependence on $\lambda$ in the processes $X^{x}, X^{x,\epsilon},J^{x},J^{x,\epsilon}$.
\paragraph{Proof of lemma \ref{lem:Jxepsilon_Jx_bound}. }
 We consider first the process  $J_t^{x,\epsilon}=\tfrac{X_t^{x+\epsilon}-X_t^x}{\epsilon}$ satisfying the linear SDE \eqref{eq:Jxe SDE}.
From Corollary \ref{cor:flux} with Condition \ref{eq:Lambda2},
\begin{align*}
\mbox{
$\EE\big[|\phi_t^\epsilon|^2\big] \leq\EE\big[\int_0^1(X_t^x+\epsilon\theta J_t^{x,\epsilon})^{4(\alpha-1)}d\theta\big]\leq C\;\{\EE[\left(X_t^x\right)^{4(\alpha-1)}]+\EE[\left(X_t^{x+\epsilon}\right)^{4(\alpha-1)}]\}<  C ,$}
\\
\mbox{
$\EE\big[|\psi_t^\epsilon|^2\big]\leq\EE\big[\int_0^1(X_t^x+\epsilon\theta J_t^{x,\epsilon})^{2(\alpha-1)}d\theta\big]\leq C\;\{\EE[\left(X_t^x\right)^{2(\alpha-1)}]+\EE[\left(X_t^{x+\epsilon}\right)^{2(\alpha-1)}]\}< C.$}
\end{align*}
Similarly, using also the $2(\alpha-1)$-locally Lipschitz property of $b$ in \ref{C2:H_loclip}
\begin{align*}
\EE\big[|\xi_t^\epsilon|^2\big]&\leq\EE\big[\big|\tfrac{b\left(X_t^{x+\epsilon}\right)-b\left(X_t^{x}\right)}{X_t^{x+\epsilon}-X_t^{x}}\big|^2\big]\leq C\;\left(1+\EE[\left(X_t^x\right)^{4(\alpha-1)}]+\EE[\left(X_t^{x+\epsilon}\right)^{4(\alpha-1)}]\right)\leq C.
\end{align*}
Therefore, $(\int_0^t \psi^\epsilon_s d\theta dW_s;\;0\leq t\leq T)$ is a square integrable martingale, Equation \eqref{eq:Jxe SDE} admits a unique strong solution given by the following exponential form (see e.g. \cite[Thm V.52]{Protter-04})
\begin{align}\label{eq:Jxepsilon exponential}
\mbox{$
J_t^{x,\epsilon}=
\exp\big\{\int_0^t\xi_s^\epsilon ds+\left(2\alpha-1\right)\lambda \sigma^2\int_0^t\phi_s^\epsilon ds+\alpha\sigma\int_0^t\psi_s^\epsilon dW_s -\tfrac{\alpha^2\sigma^2}{2}\int_0^t\left(\psi_s^\epsilon\right)^2ds\big\}.
$}
\end{align}
In turn, $J_t^{x,\epsilon}\geq 0$  yields the increasing property of the flow, $X_t^{x+\epsilon}\geq X_t^{x}$, and the increasing property of the map $\theta\mapsto \left(X_t^x+\epsilon\theta J_t^{x,\epsilon}\right)$ in $[0,1]$, from which we obtain the following relation
\begin{align}\label{eq:psi phi inequalit2}
\begin{aligned}
(X_t^x)^{2(\alpha-1)}\leq\phi_t^\epsilon&\leq (X_t^{x+\epsilon})^{2(\alpha-1)},\\
(X_t^x)^{\alpha-1}\leq\psi_t^\epsilon&\leq (X_t^{x+\epsilon})^{\alpha-1}.
\end{aligned}
\end{align}

From the exponential form  \eqref{eq:Jxepsilon exponential}, for all $q>0,$
\begin{align*}\label{eq:Jxepsilon alpha}
(J_t^{x,\epsilon})^q & =   \mbox{$ \exp\big\{(2\alpha-1)q\lambda \sigma^2\int_0^t\phi_s^\epsilon ds+q\int_0^t\xi_s^\epsilon ds+\alpha q \sigma\int_0^t\psi_s^\epsilon dW_s - \tfrac{\alpha^2\sigma^2}{2}q \int_0^t(\psi_s^\epsilon)^2 ds\big\} $} \nonumber\\
&\leq \exp\{qtB_1'\} \mbox{$ \exp\big\{q\sigma^2\int_0^t\big((2\alpha-1)\lambda-\tfrac{B_2'}{\sigma^2}\big)\phi_s^\epsilon ds+\left(q-1\right)\tfrac{q\sigma^2\alpha^2}2\int_0^t(\psi_s^\epsilon)^2 ds\big\} $} \\
&\hspace{1.2in} \times \mbox{$ \exp\big\{
\alpha q \sigma\int_0^t\psi_s^\epsilon dW_s -\tfrac{\alpha^2\sigma^2}2q^2 \int_0^t\left(\psi_s^\epsilon\right)^2ds\big\} $} ,
\end{align*}
where the last inequality is obtained from  \ref{C2:H_poly'}:
$\xi_t^\epsilon\leq\;B_1'-B_2'\phi_t^\epsilon$.
Since $(\psi_t^\epsilon)^2\leq \phi_t^\epsilon$,  by choosing $q>0$ such that $(2\alpha-1)\lambda+\tfrac{\alpha^2}{2}(q-1)\leq\tfrac{B_2'}{\sigma^2}$,  we get
\begin{align*}
(J_t^{x,\epsilon})^q \leq \exp\{qB_1't\} \ \exp\big\{\alpha q \sigma\int_0^t\psi_s^{\epsilon} dW_s-\tfrac{\alpha^2 q^2 \sigma^2}2\int_0^t\left(\psi_s^{\epsilon}\right)^{2} ds\big\}.
\end{align*}
Since the process in the right hand-side is a supermartingale, we have for all $t\in[0,T]$,
\[\EE\big[(J_t^{x,\epsilon})^q\big]\leq \exp\{qB_1'T\},\]
and from this, we deduce that
$\EE\big[|X_t^{x+\epsilon}-X_t^x|^q \big] \leq \epsilon^q \exp\{qB_1'T\}$ tends to $0$ with $\epsilon$.

The exponential form \eqref{eq:expoJlambda} allows to replicate the computation above for the process $J^x$, and conclude
similarly that for all $t\in[0,T]$, $\EE\big[(J_t^{x})^q\big]\leq \exp\{qB_1'T\}$.

\paragraph{Proof of Lemma \ref{lem:ConvergenceJ}. }
 Consider the difference $\mathcal{E}_t^{x,\epsilon}:=J_t^x-J_{t}^{x,\epsilon},$
which can be rewritten, using the SDEs \eqref{eq:derivativeSDE} and \eqref{eq:Jxe SDE}, as
\begin{align*}
\mathcal{E}_t^{x,\epsilon}
&= \mbox{$ \lambda\sigma^2(2\alpha-1)\int_0^t J_s^{x,\epsilon}\brac{(X_s^x)^{2(\alpha-1)}-\phi_s^\epsilon}ds +\lambda\sigma^2(2\alpha-1)\int_0^t ( \mathcal{E}_s^{x,\epsilon})(X_s^x)^{2(\alpha-1)}ds $} \\
&\quad+ \mbox{$ \int_0^t J_s^{x,\epsilon}\brac{b'(X_s^x)-\xi_s^\epsilon}ds +\int_0^t (\mathcal{E}_t^{x,\epsilon})b'(X_s^x)ds$} \\
&\quad+ \mbox{$ \alpha\sigma\int_0^t\brac{(X_s^x)^{\alpha-1}-\psi_s^\epsilon}J_s^{x,\epsilon}dW_s+\alpha\sigma\int_0^t (\mathcal{E}_t^{x,\epsilon})(X_s^x)^{\alpha-1}dW_s.$}
\end{align*}
Introducing the stopping time $\tau_M:=\{0\leq t\leq T:J_t^x-J_t^{x,\epsilon}\geq M\}$, with $M>0$,
It\^{o}'s formula yields
\begin{align}\label{eq:Jxepsilon-Jx}
\EE\big[|\mathcal{E}_{t\wedge\tau_M}^{x,\epsilon}|^2\big]&
= \mbox{$ 2\lambda\sigma^2(2\alpha-1)\EE\brac{\int_0^{t\wedge\tau_M} \mathcal{E}_{s}^{x,\epsilon} J_s^{x,\epsilon}\brac{(X_s^x)^{2(\alpha-1)}-\phi_s^\epsilon} +\left(\mathcal{E}_{s}^{x,\epsilon}\right)^2(X_s^x)^{2(\alpha-1)}ds} $} s\nonumber \\
&\quad+2\ \mbox{$ EE\brac{\int_0^{t\wedge\tau_M}{ \mathcal{E}_{s}^{x,\epsilon} J_s^{x,\epsilon}\brac{b'(X_s^x)-\xi_s^\epsilon} + \left(\mathcal{E}_{s}^{x,\epsilon}\right)^2 b'(X_s^x)}ds}$}  \nonumber \\
&\quad +\alpha^2\sigma^2 \mbox{$ \EE\brac{\int_0^{t\wedge\tau_M} \Big(\brac{(X_s^x)^{\alpha-1}-\psi_s^\epsilon}J_s^{x,\epsilon}+\mathcal{E}_{s}^{x,\epsilon}(X_s^x)^{\alpha-1}\Big)^2ds} $} .
\end{align}
Using \eqref{eq:psi phi inequalit2}, there exists a non-negative constant $C$ independent on $\epsilon$ and $M$ such that
\begin{align*}
\phi_t^\epsilon-(X_t^x)^{2(\alpha-1)}&\leq (X_t^{x+\epsilon})^{2(\alpha-1)}-(X_t^x)^{2(\alpha-1)}\leq  \epsilon \, C J_t^{x,\epsilon} \left(|X_t^{x+\epsilon}|^{2\alpha -3} + |X_t^{x}|^{2\alpha -3} \right)\\
\psi_t^\epsilon-(X_t^x)^{\alpha-1}&\leq (X_t^{x+\epsilon})^{\alpha-1}-(X_t^x)^{\alpha-1}\leq  \epsilon \,C J_t^{x,\epsilon} \left(|X_t^{x+\epsilon}|^{\alpha -2} +|X_t^{x}|^{\alpha -2}\right),
\end{align*}
and
\begin{align}\label{eq:T1}
\begin{aligned}
\EE\big[\big|(J_t^{x,\epsilon})^2((X_t^x)^{2(\alpha-1)}-\phi_t^\epsilon)\big|\big]
&\leq  \epsilon \, C \EE\big[|(J_t^{x,\epsilon})^3
\left(|X_t^{x+\epsilon}|^{2\alpha -3} + |X_t^{x}|^{2\alpha -3} \right)|\big], \\
\EE\big[\big|J_t^{x,\epsilon}J_t^{x}((X_t^x)^{2(\alpha-1)}-\phi_t^\epsilon)\big|\big]
 &\leq  \epsilon \, C \EE\big[|(J_t^{x,\epsilon})^2J_t^{x}
\left(|X_t^{x+\epsilon}|^{2\alpha -3} + |X_t^{x}|^{2\alpha -3} \right)|\big], \\
\EE\big[|(J_t^{x,\epsilon})^2 ((X_t^x)^{(\alpha-1)}-\psi_t^\epsilon)^2|\big]
&\leq   \epsilon \, C \EE[\left(J_t^{x,\epsilon}\right)^3\left(|X_t^{x+\epsilon}|^{2\alpha -3} + |X_t^{x}|^{2\alpha -3} \right)] .
\end{aligned}
\end{align}
Condition \eqref{eq:Lambda3} on $B'_2$ allows to bound up to $\EE[\left(J_t^{x,\epsilon}\right)^6]$, and on $B^\lambda_2$ allows to bound up to $\EE[|X_t^{x}|^{4(\alpha -1)}]$. Similarly, from the $(\overline{\gamma}_{(1)},\underline{\gamma}_{(1)})$-locally Lipschitz property of $b'$, there exists a constant $C\geq0$ independent on $\epsilon$ and $M$ such that
\begin{align}\label{eq:T3}
 \EE\big[\left|(J_t^{x,\epsilon})^2\left(b'(X_t^x)-\xi_t^\epsilon\right)\right|\big]
&= \EE\big|(J_t^{x,\epsilon})^2\mbox{$\int_0^1$} \left(b'(X_t^x)-b'\left(X_t^{x}(\lambda)+\theta\epsilon\;J_t^{x,\epsilon}(\lambda)\right)\right)d\theta\big|
\\
& \leq \mbox{$\int_0^1$} \big[\EE\left| (J_t^{x,\epsilon})^2 \left|b'(X_t^x)-b'(X_t^{x}+\epsilon\theta J_t^{x,\epsilon})\right| \right|\big] d\theta \nonumber\\
& \leq \epsilon \, C\EE\big[\big| \left(J_t^{x,\epsilon}\right)^2 (1+\left(X_t^{x+\epsilon}\right)^{\overline{\gamma}_{(1)}}+\left(X_t^{x}\right)^{-\underline{\gamma}_{(1)}}) \big|\big], \\ \nonumber
\EE\big[\left|J_t^{x,\epsilon}J_t^{x}\left(b'(X_t^x)-\xi_t^\epsilon\right)\right|\big]&\leq\epsilon \ C\EE\big[\big| J_t^{x,\epsilon}J_t^{x} (1+\left(X_t^{x+\epsilon}\right)^{\overline{\gamma}_{(1)}}+\left(X_t^{x}\right)^{-\underline{\gamma}_{(1)}}) \big|\big], \nonumber
\end{align}
and $\EE[|X_t^{x}|^{2\overline{\gamma}_{(1)}}]$ is bounded under \eqref{eq:Lambda3}.

In \eqref{eq:Jxepsilon-Jx}, summing separately the three terms multiplying
 $(\mathcal{E}_{s}^{x,\epsilon})^2$, using \ref{C2:H_poly'} and next the Condition  \ref{eq:Lambda3} on $B'_2$  we get
\begin{align*}
& \left(\mathcal{E}_{s}^{x,\epsilon}\right)^2\big[ 2\lambda\sigma^2(2\alpha-1)(X_t^x)^{2(\alpha-1)}
+ 2  b'(X_s^x)  +2\alpha^2\sigma^2 (X_s^x)^{2(\alpha-1)}\big] \\
& \leq \left(\mathcal{E}_{s}^{x,\epsilon}\right)^2\big[ 2\lambda\sigma^2(2\alpha-1)(X_t^x)^{2(\alpha-1)}
+ 2  B'_1 - 2 B'_2 (X_s^x)^{2(\alpha -1)}  +\alpha^2\sigma^2 (X_s^x)^{2(\alpha-1)}\big] \leq  2  B'_1 \left(\mathcal{E}_{s}^{x,\epsilon}\right)^2.
\end{align*}
Coming back to \eqref{eq:Jxepsilon-Jx} using inequalities \eqref{eq:T1} and \eqref{eq:T3}, and $\mathcal{E}^{x,\epsilon} \leq J_t^{x,\epsilon} + J_t^{x}$,
\begin{align*}
\mbox{$\EE\big[|\mathcal{E}_{t\wedge\tau_M}^x|^2\big] \leq C \epsilon + 2 B'_1 \int_0^{t} \EE\big[|\mathcal{E}_{s\wedge\tau_M}^x|^2\big] ds,$}
\end{align*}
and applying Gronwall's Lemma. From which we obtain $\EE(\mathcal{E}_{t\wedge\tau_M}^{x,\epsilon})^2 \leq C~ {\epsilon}$ for all $t\in[0,T]$. We end this proof by taking limits $M\rightarrow+\infty$ and $\epsilon\rightarrow0$.

\section{ 
Final step for the proof of Proposition \ref{prop:KolPDE}}\label{sec:app:proof_FK}

We define for some non-negative integers $n,m$, the function $\widetilde{b}_{n,m}$ as
\[\widetilde{b}_{n,m}(x) = n b'(x)+m\alpha\sigma^2(2\alpha-1)x^{2(\alpha-1)},\quad~\forall x\geq0,\]
with, for $j=0,1,2,3$, and using the fact that $2\alpha-1 \leq \overline{\gamma}_{(j+1)} + j$ (see \ref{C2:H_loclip'}),
\begin{align}\label{eq:btilde}
\begin{aligned}
|\widetilde{b}^{(j)}_{n,m}|(x)&\leq C_{n,m} (1+ x^{-(\underline{\gamma}_{(j+1)} \vee (j -2(\alpha-1))}+x^{\overline{\gamma}_{(j+1)}+1}),\\
|\widetilde{b}^{(j)}_{n,0}|(x)&\leq C_n (1+ x^{-\underline{\gamma}_{(j+1)}}+x^{\overline{\gamma}_{(j+1)}+1}),
\end{aligned}
\end{align}
with the help of Lemma \ref{lem:LocallyLipschitz},
From \eqref{eq:D2After} and \eqref{eq:D2After_bis}, we get
\begin{align}\label{eq:d2u/dx2}
\tfrac{\partial^2 u}{\partial x^2}(t,x)&= \mbox{$ \EE\big[f^{(2)}(X_{T-t}^x(2\alpha))\exp\big\{\int_0^{T-t}\widetilde{b}_{2,1}(X_s^x(2\alpha))ds\big\}\big]$} \\
&\quad+ \mbox{$ \int_0^{T-t}\EE\big[\exp\big\{\int_0^{s}\widetilde{b}_{2,1}(X_s^x(2\alpha))ds\big\}\widetilde{b}_{1,0}^{\prime}(X_s^x(2\alpha))\tfrac{\partial u}{\partial x}(t+s,X_s^x(2\alpha))\big]ds. $}\nonumber
\end{align}
We identify \eqref{eq:d2u/dx2} with the form \eqref{eq:DerivativeRewritten} with
\begin{align*}
&f_2=f^{(2)} ~\mbox{ bounded}, \\
&g_2 = \widetilde{b}_{2,1} ~ \mbox{ bounded from above (assuming  $2B'_2 \geq \alpha (2\alpha -1) \sigma^2$)}~\mbox{  with } \\
&\qquad \qquad  |g_2'|(x) \leq C (1+x^{- (\underline{\gamma}_{(2)}\vee 3 -2\alpha )} + x^{\overline{\gamma}_{(2)} +1}), \\
& h_2=\widetilde{b}_{1,0}'\tfrac{\partial u}{\partial x}~\mbox{  with, using \eqref{eq:upper_first}}, \quad  |h_2|(x)\leq (1+x^{(\overline{\gamma}_{(2)} + 1)} + x^{-\underline{\gamma}_{(2)}}).
\end{align*}
Using \eqref{eq:btilde}  on $\widetilde{b}$, and the control of $\tfrac{\partial^2 u}{\partial x^2}$ in \eqref{eq:upper_second}, we determine the powers involved in the upper bound of $|h'_2|$ (that will coincide with the moments to bound for the control of $\tfrac{\partial^3 u}{\partial x^3}$) by evaluating
\begin{align*}
&\widetilde{b}_{3,1}^{\prime}(x)\tfrac{\partial^2 u}{\partial x^2}(\cdot,x) +
\widetilde{b}_{1,0}^{(2)}(x)  \\
&\quad \preceq (1+ x^{-(\underline{\gamma}_{(2)} \vee (3 -2\alpha))}+x^{\overline{\gamma}_{(2)}+1})(1+x^{\overline{\gamma}_{(2)}+1}+x^{-\underline{\gamma}_{(2)}}) + x^{-\underline{\gamma}_{(3)}}+x^{\overline{\gamma}_{(3)}+1}\preceq  1+x^{\overline{\beta}}
+ x^{\underline{\beta}},
\end{align*}
(using the Hardy symbol $\preceq$ as asymptotic notation) and hence $|h_2'|(x)\leq (1+x^{\overline{\beta}} + x^{-\underline{\beta}}
)$,  with
\begin{align*}
\overline{\beta} :=  2(\overline{\gamma}_{(2)}+1)\vee (\overline{\gamma}_{(3)}+1), \mbox{ and }  \underline{\beta} :=2\underline{\gamma}_{(2)}\vee ( \underline{\gamma}_{(2)} +3 -2\alpha) \vee  \underline{\gamma}_{(3)}.
\end{align*}
Therefore, we apply Proposition \ref{prop:interchang} with $\rho_0=\overline{\beta}$,  $\rho_2=\rho_4=\overline{\gamma}_{(2)}+1$, that must satisfy Condition \ref{eq:Lambda1} for $\lambda = 2\alpha$:
\begin{align}\label{eq:D3Before}
\tfrac{\partial^3 u}{\partial x^3}(t,x) &= \mbox{$ \EE\big[\exp\big\{\int_0^{T-t}\widetilde{b}_{2,1}(X_s^x)ds\big\}\big(f^{(3)}(X_{T-t}^x)J_{T-t}^x+f^{(2)}(X_{T-t}^x)\int_0^{T-t}\widetilde{b}_{2,1}'(X_s^x)J_{s}^xds\big)\big]$}  \nonumber \\
&\quad+ \mbox{$ \int_0^{T-t}\EE\big[\exp\big\{\int_0^{s}\widetilde{b}_{2,1}(X_s^x)ds\big\}\widetilde{b}_{1,0}^{\prime}(X_s^x)\tfrac{\partial u}{\partial x}(t+s,X_s^x)\int_0^s\widetilde{b}_{2,1}^{\prime}(X_r^x)J_r^xdr\big]ds$}  \nonumber \\
&\quad+ \mbox{$\int_0^{T-t}\EE\big[\exp\big\{\int_0^{s}\widetilde{b}_{2,1}(X_s^x)ds\big\}\widetilde{b}_{1,0}^{\prime}(X_s^x)\tfrac{\partial^2 u}{\partial x^2}(t+s,X_s^x)J_s^x\big]ds$}  \nonumber \\
&\quad+\mbox{$ \int_0^{T-t}\EE\big[\exp\big\{\int_0^{s}\widetilde{b}_{2,1}(X_s^x)ds\big\}\widetilde{b}_{1,0}^{(2)}(X_s^x)\tfrac{\partial u}{\partial x}(t+s,X_s^x)J_s^x\big]ds$} ,
\end{align}
where we write $X^x$ and $J^x$ for $X^x(2\alpha)$ and $J^x(2\alpha)$.

\paragraph{Estimates on $\tfrac{\partial^3 u}{\partial x^3}$ and $\tfrac{\partial^4 u}{\partial x^4}$. }
We apply the same technique as for the second derivative, namely, we rewrite the second and fourth terms of the sum in \eqref{eq:D3Before}, using the Markov property and time homogeneity of the process $(X_s^x(2\alpha);0\leq s \leq T-t)$ for the second term in $f^{(2)}$ in \eqref{eq:D3Before}:
\begin{align*}\label{eq:MarkovProp2}
& \mbox{$ \EE\big[f^{(2)}(X_{T-t}^x)\exp\big\{\int_s^{T-t}\widetilde{b}_{2,1}(X_s^x)ds\big\}\big|\Ff_s\big]$} \\
& = \mbox{$ \EE\big[f^{(2)}(X_{T-t-s}^y)\exp\big\{\int_0^{T-t-s}\widetilde{b}_{2,1}(X_r^y)dr\big\}\big]\big|_{y=X_s^x} $} \\
&=\tfrac{\partial^2 u}{\partial x^2}(t+s,X_s^x)
  \mbox{$-\int_0^{T-t-s}\EE\big[\exp\big\{\int_0^{r}\widetilde{b}_{2,1}(X_u^y)du\big\}\widetilde{b}_{1,0}^{\prime}(X_r^y)\tfrac{\partial u}{\partial x}(t+s+r,X_r^y)\big]\big|_{y=X_s^x}dr$} \\
&=\tfrac{\partial^2 u}{\partial x^2}(t+s,X_s^x)
 \mbox{$ -\int_s^{T-t}\EE\big[\exp\big\{\int_s^{r}\widetilde{b}_{2,1}(X_u^x)du\big\}\widetilde{b}_{1,0}^{\prime}(X_r^x))\tfrac{\partial u}{\partial x}(t+r,X_r^x)\big|\Ff_s\big]dr$} .
\end{align*}
We also use an integration by part in the second line of  \eqref{eq:D3Before}:
\begin{align*}
&\mbox{$ \EE\big[\int_0^{T-t}\big(\exp\big\{\int_0^{s}\widetilde{b}_{2,1}(X_r^x )dr\big\}\widetilde{b}_{1,0}^{\prime}(X_s^x )\tfrac{\partial u}{\partial x}(t+s,X_s^x )\big)\int_0^s\widetilde{b}_{2,1}^{\prime}(X_r^x )J_r^x dr\;ds\big]$} \\
&= \mbox{$ \EE\big[\int_0^{T-t}\big(\int_s^{T-t}\exp\big\{\int_0^{r}\widetilde{b}_{2,1}(X_u^x )du\big\}\widetilde{b}_{1,0}^{\prime}(X_r^x )\tfrac{\partial u}{\partial x}(t+r,X_r^x )dr\big)\widetilde{b}_{2,1}^{\prime}(X_s^x )J_s^x ds\big],$}
\end{align*}
where again we write $X^x$ and $J^x$ for $X^x(2\alpha)$ and $J^x(2\alpha)$. Then, substituting in \eqref{eq:D3Before} we get
\begin{align*}
\tfrac{\partial^3 u}{\partial x^3}(t,x) =& \mbox{$ \EE\big[\exp\big\{\int_0^{T-t}\widetilde{b}_{2,1}(X_s^x(2\alpha))ds\big\}f^{(3)}(X_{T-t}^x(2\alpha))J_{T-t}^x(\alpha)\big]$} \\
& \quad  \mbox{$ +\int_0^{T-t}\EE\big[\exp\big\{\int_0^{s}\widetilde{b}_{2,1}(X_r^x(2\alpha))dr\big\}\big(\widetilde{b}_{3,1}^{\prime}(X_s^x(2\alpha))\tfrac{\partial^2 u}{\partial x^2}(t+s,X_s^x(2\alpha)) $} \\
& \qquad\qquad\qquad\qquad\qquad+\widetilde{b}_{1,0}^{(2)}(X_s^x(2\alpha))\tfrac{\partial u}{\partial x}(t+s,X_s^x(2\alpha))\big)J_s^x(2\alpha)\big]ds.
\end{align*}
We consider the change of measure $\mathbb{Q}^{3\alpha}$ through the density $\mathcal{Z}_t^{(2\alpha,3\alpha)}$ (assuming \eqref{hypo:constraint} with $\lambda=2\alpha$, for which we observe that
\[
\mbox{$\exp\big\{\int_0^{t}\widetilde{b}_{2,1}(X_s^x(2\alpha))ds\big\}J_{t}^x(2\alpha )\mathcal{Z}_{t}^{(2\alpha,3\alpha)}=\exp\big\{ \int_0^{t}\widetilde{b}_{3,3}(X_s^x(2\alpha))ds \big\}.$}\]
Therefore, using again that $\textit{Law}^{\mathbb{Q}^{3\alpha}}(X^x(2\alpha))=\textit{Law}^{\PP}(X^x(3\alpha))$,  we  obtain
\begin{align}\label{eq:d3udx2}
\tfrac{\partial^3 u}{\partial x^3}(t,x)
= &  \mbox{$ \EE\big[\exp\big\{\int_0^{T-t}\widetilde{b}_{3,3}(X_s^x(3\alpha))ds\big\}f^{(3)}(X_{T-t}^x(3\alpha))\big]$}  \nonumber\\
& + \mbox{$ \int_0^{T-t}\EE\big[\exp\big\{\int_0^{s}\widetilde{b}_{3,3}(X_r^x(3\alpha))dr\big\}\big(\widetilde{b}_{3,1}^{\prime}(X_s^x(3\alpha))\tfrac{\partial^2 u}{\partial x^2}(t+s,X_s^x(3\alpha)) $} \nonumber \\
& \qquad\qquad\qquad+\widetilde{b}_{1,0}^{(2)}(X_s^x(3\alpha))\tfrac{\partial u}{\partial x}(t+s,X_s^x(3\alpha))\big)\big]ds.
\end{align}
Notice that $\widetilde{b}_{3,3}$ is bounded from above assuming $B'_2 \geq \alpha \sigma^2(2\alpha -1)$. By means of the boundedness of $\tfrac{\partial u}{\partial x}$ and $f^{(i)}$, we stay with
\begin{align*}
\mbox{$ \big|\tfrac{\partial^3 u}{\partial x^3}\big|(t,x) \leq C\big(1+\int_0^{T-t}\EE\big[|\widetilde{b}_{3,1}^{\prime}(X_s^x(3\alpha))\tfrac{\partial^2 u}{\partial x^2}(t+s,X_s^x(3\alpha))+\widetilde{b}_{1,0}^{(2)}(X_s^x(3\alpha))|\big]ds\big).$}
\end{align*}
Now using Corollary \ref{cor:flux}, with  $\overline{\beta}\leq 1 + \tfrac{2B_2^{3 \alpha}}{\sigma^2}$, we get
\begin{align}\label{eq:upper_third}
\sup_{t\in[0,T]}\big|\tfrac{\partial^3 u}{\partial x^3}\big|(t,x)\leq C\;(1+x^{-\overline{\beta}}+ x^{\underline{\beta}}).
\end{align}

In order to apply Proposition \ref{prop:interchang} a last time, we  identify in \eqref{eq:d3udx2}  the form \eqref{eq:DerivativeRewritten} with
\begin{align*}
& f_3(x)=f^{(3)}(x),~\mbox{ bounded,}\\
& g_3(x)=\widetilde{b}_{3,3}(x),~\mbox{bounded from above when $B_2' \geq \alpha \sigma^2(2\alpha-1)$}, \mbox{  with }  \\
& \qquad\qquad  \qquad |g_3'|(x) \leq C (1+x^{\overline{\gamma}_{(2)} + 1} + x^{-(\underline{\gamma}_{(2)} \vee (3-2\alpha) )}), \\
& h_3(\cdot;x)=\widetilde{b}_{3,1}'(x)\tfrac{\partial^2 u}{\partial x^2}(\cdot,x)+\widetilde{b}_{1,0}^{(2)}(x)\tfrac{\partial u}{\partial x}(\cdot,x),
~\mbox{ with}\quad |h_3|(\cdot;x)\leq C(1+x^{\overline{\beta}} +x^{- \underline{\beta}}).
\end{align*}
Again, using \eqref{eq:upper_first}, \eqref{eq:upper_second} and \eqref{eq:upper_third}, we estimate  the powers involved in the expression on $|h'_3|$ (that will coincide with the moments to bound for the control of $\tfrac{\partial^4 u}{\partial x^4}$) by evaluating
\begin{align*}
&\widetilde{b}_{3,1}^{\prime}(x)\tfrac{\partial^3 u}{\partial x^3}(\cdot,x) +
\widetilde{b}_{1,0}^{(3)}(x) + \widetilde{b}_{3,1}^{(2)} (x)\tfrac{\partial^2 u}{\partial x^2}(\cdot,x)\\
&
\preceq (1+ x^{-(\underline{\gamma}_{(2)} \vee (3 -2\alpha))}+x^{\overline{\gamma}_{(2)}+1})  \ (1+x^{\overline{\beta}}+x^{-\underline{\beta}})\
+  x^{-\underline{\gamma}_{(4)}}+x^{\overline{\gamma}_{(4)}+1} \\
& \quad
+ (1+x^{\overline{\gamma}_{(2)}+1}+x^{-\underline{\gamma}_{(2)}})(1+ x^{-(\underline{\gamma}_{(3)} \vee (4 -2\alpha))}+x^{\overline{\gamma}_{(3)}+1})\\
& \preceq(1+x^{\overline{\overline{\beta}}}+x^{-\underline{\underline{\beta}}}),
\end{align*}
with
\begin{align*}
\overline{\overline{\beta}} & := \big\{\overline{\beta} + (\overline{\gamma}_{(2)}+1) \big\}  \vee \big\{ \overline{\gamma}_{(2)}+ \overline{\gamma}_{(3)}+2 \big\} \vee (\overline{\gamma}_{(4)}+1)\\
& \quad = 3(\overline{\gamma}_{(2)}+1)\vee ( \overline{\gamma}_{(2)} + \overline{\gamma}_{(3)}+2) \vee (\overline{\gamma}_{(4)}+1),
\\
\underline{\underline{\beta}}  & := \big\{\big(\underline{\gamma}_{(2)} \vee (3 -2\alpha)\big)  + \underline{\beta} \big\}\vee \underline{\gamma}_{(4)} \vee \big\{\underline{\gamma}_{(2)} + \big(\underline{\gamma}_{(3)} \vee (4 -2\alpha)\big)\big\}\\
& \quad = \big\{\big(\underline{\gamma}_{(2)} \vee (3 -2\alpha)\big)  + (2\underline{\gamma}_{(2)}\vee ( \underline{\gamma}_{(2)} +3 -2\alpha) \vee  \underline{\gamma}_{(3)}) \big\}\vee \big\{\underline{\gamma}_{(2)} + \big(\underline{\gamma}_{(3)} \vee (4 -2\alpha)\big)\big\}\vee \underline{\gamma}_{(4)}.
\end{align*}
Then, assuming  $\overline{\overline{\beta}} \leq \tfrac12+\tfrac{B_2^{3\alpha}}{\sigma^2}$, we apply Proposition \ref{prop:interchang},
  obtaining $u\in\mathcal{C}^{1,4}([0,T]\times\RR^+)$. Using the Markov property and the time homogeneity of the process $(X_s^x(3\alpha);0\leq s \leq T-t)$, we deduce the following form  (with $X_s^x$ understood as $X_s^x(3\alpha)$)
\begin{align}\label{eq:D4After1}
\begin{aligned}
\tfrac{\partial^4 u}{\partial x^4}(t,x)
= & \mbox{$ \EE\big[\exp\big\{\int_0^{T-t}\widetilde{b}_{3,3}(X_s^x)ds\big\}f^{(4)}(X_{T-t}^x)J_{T-t}^x\big] $}\\
& + \mbox{$\int_0^{T-t}\EE\big[\exp\big\{\int_0^{s}\widetilde{b}_{3,3}(X_s^x)ds\big\}\big(\widetilde{b}_{4,1}^{(2)}(X_s^x)\tfrac{\partial^2 u}{\partial x^2}(t+s,X_s^x)$} \\
& \qquad +\widetilde{b}_{1,0}^{(3)}(X_s^x)\tfrac{\partial u}{\partial x}(t+s,X_s^x)+\widetilde{b}_{6,4}'(X_s^x)\tfrac{\partial^3 u}{\partial x^3}(t+s,X_s^x)\big)J_s^x\big]ds.
\end{aligned}
\end{align}
Considering the change of measure $\mathbb{Q}^{4\alpha}$ with density $\mathcal{Z}_t^{(3\alpha,4\alpha)}$, we have
$$
\mbox{$ \exp\big\{\int_0^{T-t}\widetilde{b}_{3,3}(X_s^x)ds\big\}J_{T-t}^x(3\alpha )\mathcal{Z}_{T-t}^{(3\alpha,4\alpha)}=\exp\big\{ \int_0^{T-t}\widetilde{b}_{4,6}(X_s^x(3\alpha))ds\big\}\leq C$},$$
with $\widetilde{b}_{4,6}(x) = 4 b'(x) + 6\alpha \sigma^2(2\alpha -1) x^{2(\alpha-1)}$ bounded from above according to
$$\widetilde{b}_{4,6}(x) \leq  4 B_1^\prime - 4 B_2^\prime x^{2(\alpha-1)} + 6\alpha \sigma^2(2\alpha -1) x^{2(\alpha-1)}, $$
and the assumption that $B_2^\prime > \tfrac{6}{4}\alpha \sigma^2(2\alpha -1)$. Therefore we start to bound $|\tfrac{\partial^4 u}{\partial x^4}|$ with
\begin{align*}%\label{eq:D4After}
\big|\tfrac{\partial^4 u}{\partial x^4}\big| (t,x)
\leq C\big( 1 + \sup_{s\in[0,T]} \EE\big|\big\{\widetilde{b}_{4,1}^{(2)}\tfrac{\partial^2 u}{\partial x^2}(t+s)
+\widetilde{b}_{1,0}^{(3)}\tfrac{\partial u}{\partial x}(t+s)+\widetilde{b}_{6,4}'\tfrac{\partial^3 u}{\partial x^3}(t+s)\big\}(X_s^x(4\alpha)) \big|\big).
\end{align*}
Combining this with the previous polynomial bounds for the derivatives and the control of moments for the process $X^x(4\alpha)$ in Corollary \ref{cor:flux}, under \ref{C2:H5}, we get
\begin{align*}\mbox{$
\big|\tfrac{\partial^4 u}{\partial x^4}\big|(t,x)
\leq C\;(1+x^{\overline{\overline{\beta}}} + x^{-\underline{\underline{\beta}}}). $}
\end{align*}

We have obtained that $u\in\mathcal{C}^{1,4}([0,T]\times[0,+\infty))$ with partial derivatives satisfying \eqref{eq:PDE bounds1}. In view of the polynomial growth property of the maps $x\mapsto\frac{\partial u}{\partial x}(t,x)$, $\tfrac{\partial^2 u}{\partial x^2}(t,x)$, $b(x)$, $x^\alpha$ and the appropriate control of the $\overline{\overline{\beta}}$-th moment of the flow,  one can easily adapt the proof in Friedman \cite[Ch. 5, Th 6.1]{Friedman} to show that $u(t,x)$ solves the  Kolmogorov PDE \eqref{eq:KolPDE}.

We end this proof by reporting the conditions required  on $B_2, B'_2$, $\sigma$, $\alpha$, $\overline{\gamma}_{(i)}$, $\underline{\gamma}_{(i)}$  in order to get all the controls to be applied in  the previous steps, the combination of which forming  \ref{C2:H5}:

$\bullet$ At most we used the upper-bound on the moment $\sup_{t\in[0,T]}(\EE|(X_t(4\alpha))^{\overline{\overline{\beta}}}|+\EE|(X_t(3\alpha))^{2\overline{\overline{\beta}}}|)$, by applying Corollary \ref{cor:flux} with  the double constrain that
$\tfrac{\sigma^2}{2} \big( \overline{\overline{\beta}}  + 8\alpha -1\big)  \leq B_2$ and $\tfrac{\sigma^2}{2} \big( 2\overline{\overline{\beta}}  + 6\alpha -1\big)  \leq B_2$,
knowing that $ \overline{\overline{\beta}} := 3(\overline{\gamma}_{(2)}+1)\vee ( \overline{\gamma}_{(2)} + \overline{\gamma}_{(3)}+2) \vee (\overline{\gamma}_{(4)}+1)$.

$\bullet$ We also have to justify the use of Girsanov transform, by applying Lemma \ref{lemJ} under the sufficient condition that
\begin{align*}
\begin{array}{ll}
\mbox{ if } b(0)=0,&~~ \quad  \tfrac{\sigma^2}{2}(7\alpha -1)  \leq {B_2}, ~~\mbox{already satisfied} \\
\mbox{ if } b(0)>0,&~~\quad  \tfrac32< \alpha, \quad\mbox{ and }\quad  \tfrac{\sigma^2}{2} \big(6 \alpha + \tfrac{\alpha^2}{\sigma^2}\big) \leq  B_2.
\end{array}
\end{align*}

$\bullet$ We have bounded the terms involving  $J_t$ coming after the Girsanov transform, and at most  the term 
$$\exp\big\{\int_0^{T-t}\widetilde{b}_{3,3}(X_s^x)ds\big\} \ J_{T-t}^x(3\alpha ) \ \mathcal{Z}_{T-t}^{(3\alpha,4\alpha)},$$
by assuming $B_2^\prime > \sigma^2\alpha(3 \alpha -\tfrac{3}{2})$.

$\bullet$ Finally, we considered the necessary condition on $B_2'$ in order to apply Proposition \ref{prop:interchang} up to $\lambda=3\alpha$:
${B_2'}\geq \sigma^2\alpha( \tfrac{17}{2} \alpha -3 )$.

\section{Proof of Proposition \ref{prop:XMoments} and related Lemmas}\label{sec:appendix:proof_wellposedness}

\paragraph{Proof of Proposition \ref{prop:XMoments}. }
From the Lamperti-type transformation $\widetilde{X}=X_t^{-2(\alpha-1)}$, with $\alpha>1$,  the well-posedness of \eqref{eq:IntroSDE} essentially relies on the existence and uniqueness of a positive solution to the one-dimensional SDE
\begin{equation}\label{SDEAux1}
d\widetilde{X}_t=\widetilde{b}(\widetilde{X}_t)dt - 2\sigma(\alpha-1)\sqrt{\widetilde{X}_t}~d{W}_t,~~\widetilde{X}_0=x^{-2(\alpha-1)},
\end{equation}
where  the drift function $\widetilde{b}$,  defined as
\[\widetilde{b}(x)=(\alpha-1)\left(\sigma^2(2\alpha-1)-
2{{x}^{1 + \frac{1}{2(\alpha -1)}}b({x}^{- \frac{1}{2(\alpha-1)}})}\right),\]
is a $(\overline{\gamma}, \underline{\gamma})$-locally Lipschitz, for some strictly positive  $\overline{\gamma}$ and $\underline{\gamma}$. Then pathwise uniqueness holds for the solution of \eqref{SDEAux1}  (see e.g. Ikeda-Watanabe~\cite[Theorem 3.1]{IkedaWatanabe}).

\paragraph{Feller test for non explosion. } From the regularity of the coefficients of \eqref{SDEAux1}, a weak solution up to an explosion time exists (in the sense of Definition 5.1 in Karatzas-Shreve~\cite{KarShr-88}).  Considering
\[
S_{n}:=\inf\big\{0\leq t: \widetilde{X}_t \notin (\tfrac{1}{n}, n)\big\}, \quad S = \lim_{n\rightarrow+\infty} S_n,
\]
we use a Feller test  to show that $\PP(S=+\infty) = 1$, or equivalently that $v(0^+):=\lim_{x\rightarrow 0^+}v(x) = v(+\infty):=\lim_{x\rightarrow +\infty}v(x) = +\infty$, where $x\mapsto v(x)$ is defined from the scale function $x \mapsto p(x)$ as
\begin{align*}
p(x)&=\int_c^x\exp\{-\int_c^z \tfrac{\widetilde{b}(y)}{2\sigma^2(\alpha-1)^2y} dy\} \ dz ,\quad\quad v(x) = \int_c^x p'(y)\big(\int_c^y\tfrac{dz}{2\sigma^2(\alpha-1)^2~p'(z)z}\big)dy,
\end{align*}
with $c>0$ (e.g. \cite[Theorem 5.29]{KarShr-88}). From the  $2(\alpha-1)$-locally Lipschitz continuity of $b$,  we have, for all  $x>0$,
\[b(x^{-\frac1{2(\alpha-1)}})\geq b(0)- Cx^{-\frac1{2(\alpha-1)}}(1+x^{-1}).
\]
Multiplying both sides of the inequality by $-x^{\tfrac{1}{2(\alpha-1)}}$, it follows that
\begin{align*}
\frac{\widetilde{b}(x)}{(\alpha-1)x}  = \frac{\sigma^2(2\alpha-1)}{x}-2x^{\tfrac{1}{2(\alpha-1)}}b(x^{-\tfrac1{2(\alpha-1)}})
\leq \frac{(2\alpha-1)\sigma^2+2C}{x}+2C-2b(0)x^{\tfrac1{2(\alpha-1)}},
\end{align*}
and
\begin{align}\label{eq:bTilde growth}
\widetilde{b}(x)\leq (\alpha-1)\big(2C+(2\alpha-1)\sigma^2+2C x -2b(0)~x^{\frac{2\alpha-1}{2(\alpha-1)}}\big).
\end{align}
In the same way, using \ref{C2:H_poly}, we get the lower bounds
\[
\frac{(2\alpha-1)\sigma^2+2B_2}{x}-2B_1-2b(0) x^{\frac1{2(\alpha-1)}}\leq \frac{\widetilde{b}(x)}{(\alpha-1)x},
\]
\begin{align*}
\text{and} \quad\widetilde{b}(x) \geq  (\alpha-1)(2\alpha-1)\sigma^2+2B_2 -2B_1 (\alpha -1) x -2b(0) (\alpha -1) x^{\frac{2\alpha -1}{2(\alpha-1)}}.
\end{align*}
From the lower and upper bounds on $\frac{\tilde{b}(x)}{x(\alpha-1)}$, we derive the following estimates for $p(0^+)$ and $p(+\infty)$:
\begin{align}\label{eq:bounds p tilde}
\begin{aligned}
&p(0^+)= -\int_0^c\exp\{\int_z ^c\tfrac{\widetilde{b}(y)}{2\sigma^2(\alpha-1)^2y} dy\}\ dz \leq -\int_0^c\exp\big\{\beta(B_2)\log(\tfrac{c}{z})\big\}dz = -\infty,\\
&p(+\infty)\geq \int_c^\infty\exp\{-\beta(C)\log(\tfrac{z}{c})+\frac{2b(0)}{\sigma^2(2\alpha-1)}(z^{\tfrac{2\alpha-1}{2\alpha-2}}
-c^{\frac{2\alpha-1}{2\alpha-2}})-\tfrac{C}{\sigma^2(\alpha-1)}(z -c)\} \ dz,
\end{aligned}
\end{align}
where $a\mapsto\beta(a)=\tfrac{{2a}/{\sigma^2}+2\alpha-1}{2(\alpha-1)}> 1$ for $a\geq 0$. Therefore we have that $p(0^+)=-\infty$  implies $v(0^+)=+\infty$ (see e.g \cite{KarShr-88}, Problem 5.27).

To check that $v(+\infty)=+\infty$, we distinguish the  two cases: $b(0)>0$ and $b(0)=0$ in the upper-bound \eqref{eq:bounds p tilde}.  If $b(0)>0$, the dominant term in \eqref{eq:bounds p tilde} corresponds to the integral of $z\mapsto \exp\{\frac{2b(0)}{\sigma^2(2\alpha-1)}z ^{\frac{2\alpha-1}{2\alpha-2}}\},$ from which we obtain $p(+\infty)=+\infty$ implying $v(+\infty)=+\infty$. If $b(0)=0$, we must analyze the behavior of $v$ at infinity, for which we have
\begin{align*}
v(+\infty)&\simeq \int_c^\infty\int_c^yz^{-1}\exp\{-\int_z^y\frac{\widetilde{b}(\theta)}{2\sigma^2(\alpha-1)^2\theta} d\theta\}\ dz dy\\
&\quad \geq \int_c^\infty\int_c^yz^{-1} \exp\{- \beta(C)\log(\frac{y}{z}) - \frac{C}{\sigma^2(\alpha-1)}(y-z)\} \ dz dy \\
&\quad \geq \int_c^\infty\int_c^yz^{-1} \exp\{\widetilde{\beta} \log(\frac{z}{y}) - \frac{C}{\sigma^2(\alpha-1)}(y-z)\} \ dz dy \\
&\qquad = \int_c^\infty y^{-\widetilde{\beta}}\exp \{-\frac{C y}{\sigma^2(\alpha-1)}\}
\int_c^y z^{\widetilde{\beta}-1} \exp\{\frac{C z}{\sigma^2(\alpha-1)}\} \ dz dy,
\end{align*}
where we are allowed to choose $\widetilde{\beta}> \beta(C)+1$ such that $\widetilde{\beta}-2>0$.
Applying an integration by parts on the inner integral, we have
\begin{align*}
&\int_c^yz^{\widetilde{\beta}-1} \exp\{\frac{C z}{\sigma^2(\alpha-1)}\} \ dz dy \\
& = \frac{\sigma^2(\alpha-1)}{C} \big[y^{\widetilde{\beta}-1 } \exp\{\frac{C y }{\sigma^2(\alpha-1)}\} -c^{\widetilde{\beta}-1 } \exp\{\frac{C c }{\sigma^2(\alpha-1)}\} - \int_c^y (\beta-1) z^{\widetilde{\beta}-2}   \exp\{\frac{C z}{\sigma^2(\alpha-1)}\} \ dz\big]
\end{align*}
and consequently
\begin{align*}
&v(+\infty)  \\
& \mbox{$\geq \frac{\sigma^2(\alpha-1)}{C} \big[ \int_c^\infty\frac{dy}{y}\big[ 1 -    \frac{c^{\widetilde{\beta}-1 }}{y^{\widetilde{\beta}-1}} \exp\{-\frac{C (y-c)}{\sigma^2(\alpha-1)}\} \big] - \int_c^\infty \frac{dy}{y^2} \int_c^y (\beta-1) \frac{z^{\widetilde{\beta}-2}}{y^{\widetilde{\beta}-2}} \exp\{-\frac{C (y -z)}{\sigma^2(\alpha-1)}\} \ dz \big].$}
\end{align*}
The second integral being  finite and the first one diverging  to
$+\infty$, we can conclude that $v(+\infty)=+\infty$ and thus there exists a unique strictly positive strong solution to \eqref{SDEAux1} for all $t\in[0,T]$.  Using reversely the Lamperti transformation, this immediately implies  that  $X_t={\widetilde{X}_t}^{\frac1{2(1-\alpha)}}$ satisfies the SDE \eqref{eq:IntroSDE} on $(0,T]$ and the pathwise uniqueness of strictly positive solution is also granted.
\paragraph{Moment controls for the solution to \eqref{eq:IntroSDE}. }
Applying It\^{o}'s formula to the stopped process $({X}_{t\wedge\tau_M};0\leq t\leq T)$ with $\tau_M=\inf\{t\in[0,T]:{X}_t\geq M\},$ and, using \ref{C2:H_poly}, we get: for all $p\geq\tfrac12$
\begin{align*}
&\EE[X_{t\wedge \tau_M}^{2p}]= x^{2p}+2p\EE\brac{\int_0^{t\wedge \tau_M}X_s^{2p-1}\big\{b(X_s)
+(2p-1)\tfrac{\sigma^2}{2}X_s^{2\alpha-1}\big\}ds}\\
&\mbox{$\leq x^{2p}+2p\int_0^t\left\{B_1 \EE [X_{s\wedge \tau_M}^{2p}]
+b(0) \EE[X_{s\wedge \tau_M}^{2p-1}] +\tfrac12\left(\sigma^2(2p-1)-2B_2\right)\EE[ X_{s\wedge \tau_M}^{2p+2\alpha-2}]\right\}ds.$}
\end{align*}
Then, for all $p$ such that $1\leq 2p\leq \tfrac{2B_2}{\sigma^2}+1$, the last term is non-positive and so, for some constant $C$ depending on the parameters $p,B_i, b(0)$ and $T$, for all $0\leq t\leq T$, we get
\begin{align*}
\EE[X_{t\wedge \tau_M}^{2p}]&\leq x^{2p}+C\int_0^t\EE[X_{s\wedge \tau_M}^{2p}]ds.
\end{align*}
From Gronwall's inequality, we conclude on the $2p$\,th-moment control of $X$. We extend the result for an arbitrary exponent $0\leq 2p\leq 1 $, by simply applying H\"{o}lder's inequality.

\paragraph{Proof of Lemma \ref{lem:Negative}.  }
Let $\widetilde{X}$ be the previous Lamperti transform,  solution to \eqref{SDEAux1}. By means of the It\^o's lemma (for simplification, we omit the localization argument previously used in the proof of Proposition \ref{prop:XMoments}) and the upper bound \eqref{eq:bTilde growth}, we have for all $t\in[0,T]$, and for all power $q>0$
\begin{align*}
\EE[\widetilde{X}_{t}^{q}]
= & x^{-2(\alpha-1)q}+q\EE\big[\int_0^{t}\widetilde{X}^{q-1}_s(2(\alpha-1)^2\sigma^2(q-1)+\widetilde{b}(\widetilde{X}_s))ds\big]\\
&\leq x^{-2(\alpha-1)}+q(\alpha-1)\int_0^{t}\EE\big[\widetilde{X}^{q-1}_s\left(2(\alpha-1)\sigma^2~q+\varphi(\widetilde{X}_s) \right)\big] ds,
\end{align*}
where $x\mapsto \varphi(x):= 2C+\sigma^2+ 2C x - 2b(0) x^{\frac{2\alpha-1}{2\alpha-2}}$ achieves its maximum at the point $x^* = (\tfrac{2 C (\alpha-1)}{b(0)(2\alpha-1)})^{2\alpha-2}$. It follows that
\begin{align}\label{eq:estimationNegative}
\EE[\widetilde{X}_{t}^{q}]&\leq x^{-2(\alpha-1)q}+(\alpha-1)\varphi(x^*)\int_0^{t}q\EE[\widetilde{X}^{q-1}_s] ds+2(\alpha-1)^2\sigma^2\int_0^{t}q^2\EE[\widetilde{X}^{q-1}_s] ds.
\end{align}
By Gronwall's lemma, we deduce the finiteness of all positive moments of $\widetilde{X}$, and retrieving $X$ through the transformation $X_t={\widetilde{X}_t}^{\frac1{2(1-\alpha)}}$, we  conclude on the lemma.

\paragraph{Proof of Lemma  \ref{lem:Exponentials}. }
When $b(0)=0$, we consider CIR process $Y$, unique strictly positive  strong solution to the SDE
\begin{equation}\label{SDEAux}
dY_t=(\alpha-1)\big\{\big[\sigma^2(2\alpha-1)+2B_2\big]-2B_1Y_t\big\}dt-2\sigma(\alpha-1)\sqrt{ Y_t} dW_t,~~Y_0=x^{2(1-\alpha)}
\end{equation}
(the strict positivity is granted as two times the drift at point zero is greater than the square of the diffusion parameter, see e.g. \cite{BD15}). Applying the classical comparison principle for one-dimensional SDEs (see  e.g \cite[Chp.VI, Th 1.1]{IkedaWatanabe}), we have
\begin{align*}
\widetilde{X}_{t}\geq Y_{t},~~\mbox{or equivalently},~~ X_t\leq Y_t^{\frac1{2(1-\alpha)}}, \quad\PP-\mbox{a.s.}\forall t\in[0,T].
\end{align*}
The bound  \eqref{ExpMoments0} directly follows from the application of Lemma A.2 in \cite{BD15} (with $\nu=\frac{\sigma^2+2B_2}{2\sigma^2(\alpha-1)}$).

We consider now the case $b(0)>0$. By Itô's formula
\begin{equation*}
\mu\int_0^tX_s^{2\alpha-2}ds = \frac{2\mu}{\sigma^2}\left(\int_0^t\frac{b(X_s)}{X_s}ds+\sigma\int_0^tX_s^{\alpha-1}dW_s-\log\left(\frac{X_t}{x}\right)\right).
\end{equation*}
Using the growth condition of $b$ in \ref{C2:H_poly}, we get
\begin{align*}
\begin{aligned}
& \EE\big[\exp\{\mu\int_0^tX_s^{2\alpha-2}ds\}\big]
= \EE\big[\left(\frac{X_t}{x}\right)^{-\frac{2\mu}{\sigma^2}}\exp\{\frac{2\mu}{\sigma^2}\int_0^t\frac{b(X_s)}{X_s}ds+\frac{2\mu}{\sigma^2}\int_0^tX_s^{\alpha-1}dW_s\}\big]\\
&\leq C~ \EE\big[\left(\frac{X_t}{x}\right)^{-\frac{2\mu}{\sigma^2}}
\exp\{\frac{2\mu b(0)}{\sigma^2}\int_0^t\frac{1}{X_s}ds\}
\exp\{-\frac{2\mu B_2}{\sigma^2}\int_0^t X_s^{2(\alpha-1)}ds+\frac{2\mu}{\sigma^2}\int_0^tX_s^{\alpha-1}dW_s\}\big].
\end{aligned}
\end{align*}
Hence, applying H\"{o}lder's inequality for  $q,p$ such that $1=\frac{1}{q} +\frac{1}{p}+\frac\mu{B_2 \sigma^2}$,  we have
\begin{align*}
\EE\big[\exp\{\mu\int_0^tX_s^{2\alpha-2}ds\}\big]
&\leq C(1+x^{\frac{2\mu}{\sigma^2}})~\EE^{\frac{1}{p}}\big[\exp\{\frac{2\mu b(0)~p}{\sigma^2}\int_0^t\frac{1}{X_s}ds\}\big]\\
&\hspace{0.5in}\times\EE^{\frac{\mu}{B_2\sigma^2}}\big[\exp\{-2 B_2^2\int_0^t X_s^{2(\alpha-1)}ds+2B_2\int_0^tX_s^{\alpha-1}dW_s\}\big]\\
&\leq C(1+x^{\frac{2\mu}{\sigma^2}})~\EE^{\frac{1}{p}}\big[\exp\{\frac{2\mu b(0)~p}{\sigma^2}\int_0^t\frac{1}{X_s}ds\}\big]
\end{align*}
(using the  introduced notation $\EE^\beta$).
Expanding the last term using series and using Jensen's inequality, we have for all $t\in[0,T]$ and all $\mu< B_2\sigma^2$
\begin{align}\label{eq:expomoments}
\begin{aligned}
\EE\big[\exp\{\mu\int_0^tX_s^{2\alpha-2}ds\}\big]&\leq C(1+x^{\frac{2\mu}{\sigma^2}})\EE^{\frac{1}{p}}\Big[\sum_{k\geq 0}\frac{1}{k!}{\left(\theta\int_0^t\frac{1}{X_s}ds\right)^k}\Big] \\
&\leq C(1+x^{\frac{2\mu}{\sigma^2}})\Big(\sum_{k\geq 0}\frac{(\theta~T)^k}{k!}\sup_{t\in[0,T]}\EE[X_t^{-k}]\Big)^{\frac{1}{p}}
\end{aligned}
\end{align}
where $\theta = \frac{2\mu b(0) p}{\sigma^2}$ and the inequality in \eqref{eq:expomoments} is obtained from the Fubini-Tonelli Theorem.

It remains to prove the finiteness of the series in the inequality above.
Yet estimate on the negative moments of $X$ in Lemma \ref{lem:Negative} (or positive moments for $\widetilde{X}$) does not permit to directly conclude  as we need to explicit the way the order $k$ appears in the estimation. To overcome this difficulty we come back to estimate \eqref{eq:estimationNegative} and balance the dependence on the power as follows: for all $t\in[0,T]$ and all power $q>1$, by
Young's inequality, we have
$\widetilde{X}_t^{q-1}\leq \frac1q +\widetilde{X}_t^q$ and $q\widetilde{X}_t^{q-1}\leq q^{q-1} +\widetilde{X}_t^q$ from which we deduce that for any $a>0$ and $b>0$
\[
a q\EE[\widetilde{X}_t^{q-1}] +  b q^2\EE[\widetilde{X}_t^{q-1}]
\leq a + a q\EE[\widetilde{X}_t^q] +  b q\left( q^{q-1}+\EE[\widetilde{X}_t^q]\right)=a + b q^{q}+ (a+b)q\EE[\widetilde{X}_t^q].
\]
Applying these  to the right hand side of \eqref{eq:estimationNegative} gives
\[
\EE[\widetilde{X}_{t}^{q}]\leq x^{-2(\alpha-1)q}
+a_1 + 2(\alpha-1)^2\sigma^2 q^q
+  a_2 (\alpha-1) q\EE[\widetilde{X}_t^q],
\]
where $a_1 =(\alpha-1)\varphi(x^*)$,   $a_2=\sigma^2(2\alpha-2)+\varphi(x^*)$. Next, applying Gronwall's inequality, we get
\begin{align*}
\sup_{t\in[0,T]}\EE[\widetilde{X}_{t}^{q}]&\leq
\left(x^{-2(\alpha-1)q} + a_1 + 2(\alpha-1)^2\sigma^2 q^q\right)
\exp\{a_2 (\alpha-1) q T\},
\end{align*}
and this bound can be extended for arbitrary $q>0$ using Jensen's inequality. The estimation of the $(-k)$-th moment for $X$ is then  obtained from the transformation $X_t^{-k}={\widetilde{X}_t}^{\frac{k}{2(\alpha-1)}}$:
\begin{equation}\label{eq:NegativesM estimation}
\sup_{t\in[0,T]}\EE[{X}_{t}^{-k}]\leq \Big(x^{-k}+a_1+2(\alpha-1)^2\sigma^2\big(\frac{k}{2(\alpha-1)}\big)^{\frac{k}{2(\alpha-1)}}\Big)\exp\{a_2~\frac{k}{2} T\}.
\end{equation}
Coming back to \eqref{eq:expomoments}, we use the above estimate as follows
\begin{align}\label{eq:seriesEstim}
\begin{aligned}
&\sum_{k=0}^{+\infty}\frac{1}{k!}\left(\theta\,T \right)^k\sup_{t\in[0,T]}\EE[X_t^{-k}]\\
&\leq \sum_{k=0}^{+\infty} \frac{1}{k!}\big(\theta\,T e^{\frac{a_2}{2}T} \big)^k(x^{-k}+a_1)+2(\alpha-1)^2\sigma^2\sum_{k=0}^{+\infty}\frac{1}{k!}\big(\theta\,T e^{\frac{a_2}{2}T}\big)^k \big(\frac{k}{2\alpha-2}\big)^{\frac{k}{2(\alpha-1)}}\\
&\leq\exp\{\frac{\theta\,T e^{\frac{a_2}{2}T}}{x}\}+a_1\exp\{{\theta\,T e^{\frac{a_2}{2}T}}\}+2(\alpha-1)^2\sigma^2\sum_{k=0}^{+\infty}\beta(k),
\end{aligned}
\end{align}
by setting $\beta(k)=\tfrac{1}{k!}\left(\tfrac{\theta~T \exp\{\frac{a_2}{2}T\}}{(2\alpha-2)^{1/(2\alpha-2)}}\right)^k ~k^{\frac{k}{2(\alpha-1)}}$. We observe that
\begin{align*}
\lim_{k\rightarrow+\infty}\frac{\beta(k+1)}{\beta(k)}&=\frac{\theta\,T e^{\frac{a_2}{2}T}}{(2\alpha-2)^{\frac{1}{2(\alpha-1)}}}\lim_{k\rightarrow\infty}\left(\frac{k+1}{k}\right)^{\frac{k}{2(\alpha-1)}}~\frac1{(k+1)^{\frac{2\alpha-3}{2(\alpha-1)}}},
\end{align*}
which converges to zero when $\alpha>\frac32$, and thus
$\sum_{k=0}^{\infty}\beta(k)$ is finite.
This ends the proof by substituting \eqref{eq:seriesEstim} in \eqref{eq:expomoments}.

\section{Some complementary numerical experiments}\label{appendix:numerical}

We complement Section \ref{sec:Numeric}  by testing the obtained theoretical convergence  rate through a larger set of model cases  \eqref{eq:proto_CEV_with_const}, where the parameters of the model satisfy all the assumptions, and some other cases where not all the hypotheses are fulfilled.  In particular we illustrate the fact that hypothesis  \ref{C2:H5} do not correspond to a necessary condition.

We also report on the comparison between exp-ES and other proposed schemes of the literature (see Section \ref{sec:comparison}).

\paragraph{Numerical parameters.  }For all the presented numerical experiments,  we consider a unit terminal time $T=1$, the initial condition $x=1$ and the time step $\Delta t=1/2^p$, for $p=1,\ldots, 10$. In addition, the simulation of the ensemble average of the scheme $\EE f(\overline{X}_T)$ is performed by a Monte Carlo approximation, according to $n=10^5$ independent trajectories.

\paragraph{Test functions. }Along this section, we consider four different test functions, not all bounded,
\[f(x)=x,\;x^2,\;\frac1x,\;\exp(-x^2).\]

\paragraph{Model cases. }Denoting $\kappa$  as the left-handside of \eqref{eq:constraint_a1}  or \eqref{eq:constraint_a32}, we consider the following cases, determined by the data $(B_0, B_1,B_2,\sigma,\alpha)$
\begin{center}
\begin{tabular}{lll}
{\bf Case 1}\qquad
$(0,0,2,\tfrac{1}{10},\tfrac32)$
&  $dX_t = -2X_t^{2} dt+ \frac{X_t^{3/2}}{10} dW_t$
&\qquad $\kappa > 1.95$ \\
\blue{{\bf Case 2}\qquad %
$(0,0,3,1,\tfrac54)$ }
&\blue{   $dX_t = -3X_t^{{3}/{2}} dt+ X_t^{{5}/{4}} dW_t $}
&\qquad\blue{ $\kappa < -3$ }\\
\blue{{\bf Case 3}\qquad %
$(0,0,1,1,\tfrac32)$}
&\blue{   $dX_t = -X_t^{2} dt+\sigma X_t^{3/2} dW_t$}
&\qquad\blue{ $\kappa < -3$} \\
\blue{ {\bf Case 4}\qquad %
$(1,1,\tfrac{2}{5},\tfrac{1}{10},3)$ }
&\blue{  $dX_t =(1+X_t -\tfrac{2}{5} X_t^{5} )dt+\tfrac{X_t^{3}}{10} dW_t$}
&\qquad\blue{$\kappa < -4$ }\\
{\bf  Case 5}\qquad $(0,0,10,\tfrac12,\tfrac98)$
&  $dX_t = -10X_t^{5/4} dt+ \frac{X_t^{9/8}}{2} dW_t $
&\qquad $\kappa>8$ \\
\blue{{\bf  Case 6}\qquad  $(0,0,\frac{1}{100},\frac{1}{10},\tfrac54)$ } &\blue{ $dX_t = -\frac{X_t^{2}}{100} dt+\frac{X_t^{5/4}}{10} dW_t$}
&\qquad\blue{ $\kappa < 0$ }\\
{\bf  Case 7}\qquad $(0,0,\frac{2}{5},\frac{1}{10},3)$
&  $dX_t = -\frac{2}{5} X_t^{5} dt+\frac{X_t^{3}}{10} dW_t$
&\qquad $\kappa > 0.2$
\end{tabular}
\end{center}
with three of them, {\bf Cases 2, 3} and {\bf 6},  that are not satisfying \ref{C2:H5}.

The model in {\bf Case 3} satisfies assumptions of Theorem 2.1 in \cite{Kloeden} that states that  the approximated moments by the  Euler-Maruyama scheme and the strong $L^p$-error associated to moment-approximations diverges. Also in \cite{Kloeden}, the authors prove the divergence in the weak sense for  the $p$-th moments of the Euler-Maruyama scheme in that case.

\paragraph{Computation of the reference values. }
From Proposition \ref{prop:XMoments} and Lemma \ref{lem:Negative}, it is easy to check that for {\bf Case 1} to {\bf Case 7},  the parameters guarantee the finiteness of the expectation $\EE f(X_T)$ for each test function in the set $\{x,x^2,\frac1x,\exp(-x^2)\}$.

For both test functions $f(x)=x$ and $f(x)=x^2$,   it is possible to compute the reference values of $\EE[X_T]$ and $\EE[X_T^2]$, for {\bf Case 1} to {\bf Case 6}, through Lamperti transformation leading to CIR process for which a closed form for negative moments is available (see details in  \cite[Chapter 1]{kerlyns}). Thus, the first and the second moment reduces to
\begin{align*}
\EE[X_T]&= \tfrac{\left(4\sigma(\alpha-1)\right)^{\frac{1}{1-\alpha}}}{\Gamma(\frac1{2\alpha-2})}\int_0^1r^{\tfrac{1}{2(\alpha-1)}-1}(1-r)^{-\frac{B_2}{\sigma^2(\alpha-1)}}\exp\big\{-\tfrac{r}{2\sigma^2(\alpha-1)^2}\big\}dr,\\
\EE[X_T^2]&= \tfrac{\left(2\sigma^2(\alpha-1)^2\right)^{\frac{1}{1-\alpha}}}{\Gamma(\frac1{\alpha-1})}\int_0^1r^{\tfrac{1}{\alpha-1}-1}(1-r)^{-\frac{\sigma^2+2B_2}{2\sigma^2(\alpha-1)}}\exp\big\{-\tfrac{r}{2\sigma^2(\alpha-1)^2}\big\}dr.
\end{align*}
The choices of the power $\alpha$ in the Cases listed above are made in order to correspond to some explicit values of $\Gamma(\frac1{2\alpha-2})$ and  $\Gamma(\frac1{\alpha-1})$.

For the other cases, $f(x)=\frac1x$ and $f(x)=\exp(-x^2)$, we compute the reference values based on a Monte Carlo method combined with  the scheme exp-ES \eqref{eq:DecomposednumericScheme2}
\[\EE[f(X_T)]\approx\frac{1}{n_0}\sum_{i=1}^{n_0}f(\overline{X}_T(\omega_i,\Delta t_{\text{ref}}))\]
with $n_0=10^6$ or $10^7$, depending on the case,  and  $\Delta t_{\text{ref}} =2^{-14}$.

\medskip
Numerical results are shown in Table \ref{tab:C2experiment1-sup}. Except for {\bf Case 6} in red cells and {\bf Case 3} with test function $f(x)=x^2$,   we can observe the rate of convergence of order one in all the rows corresponding to the selection of bounded/unbounded test functions: the error is divided by 2 when going from left to right,  even if some saturation can be observed for the smallest error values $(p=8,9,10)$ when Monte Carlo error starts to be dominant.   This behavior is also confirmed in Figure \ref{im:C2experiment1}, illustrating the results obtained in Table \ref{tab:C2experiment1} in a log-log scale, with an additional line (in black) representing the reference convergence order 1.  This confirms that our proofs can certainly be extended for a larger class of test functions, and model parameters.

In particular, we highlight the parameter {\bf Case 3} $(0,0,1,1,3/2)$ that converges weakly with order one for $f(x) = x, \frac1x, \exp\{-x^2\}$, even if \ref{C2:H5} is not fulfilled, and even moreover   we know that the classical Euler-Maruyama scheme is  strongly diverging (as stated in \cite{Kloeden}) in this case.

\begin{figure}[H]%[ht!]
\centering
{\includegraphics[width=0.95\textwidth]{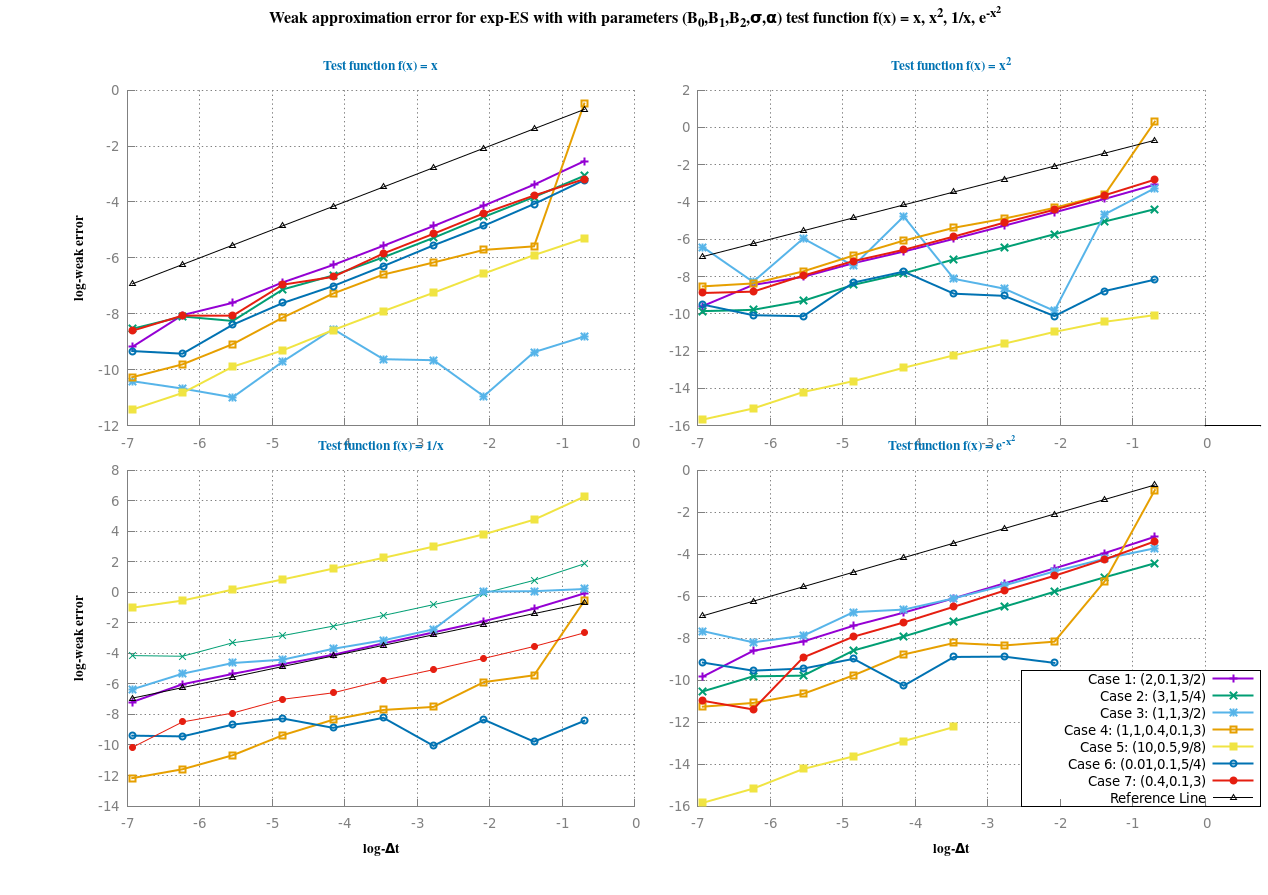}}
%\rule{\textwidth}{.1pt}
\caption{\label{im:C2experiment1-sup} Weak approximation error for the exponential-Euler scheme applied to  \eqref{eq:proto_CEV_with_const}, with  {\bf Case 1  to 7} (in log-log scale), the weak error is compared with the reference slope of order 1 (black).}
\end{figure}

\begin{table}[H]%[ht]
\centering
{\footnotesize{
\begin{tabular}{l| l l l l l l l l}
\toprule
\multicolumn{9}{c}{Weak Error with $\Delta t={2^{-p}}$, for $p=2,\ldots,9$}\\
\hline
\multicolumn{9}{c}{\tabhead{Case 1: $(B_0, B_1, B_2,\sigma,\alpha) = (0,0,2,\frac{1}{10},\frac32)$ and \ref{C2:H5} is valid.}}\\
\hline
\tabhead{{\small Test function}}
& \tabhead{$p=2$} &\tabhead{$p=3$}& \tabhead{$p=4$} & \tabhead{$p=5$} &\tabhead{$p=6$} & \tabhead{$p=7$} & \tabhead{$p=8$}&\tabhead{$p=9$}\\
\midrule
$f(x)=x$ & 3.397e-2& 1.606e-2& 7.756e-3& 3.823e-3& 1.923e-3& 1.033e-3& 4.965e-4&3.199e-4 \\
$f(x)=x^2$ & 2.147e-2& 1.043e-2&5.102e-3 &2.529e-3 &1.277e-3 &6.864e-4 &3.297e-4&2.131e-4\\
$f(x)=\frac1x$ & 3.451e-1& 1.521e-1&7.311e-2 &3.556e-2 &1.678e-2 & 9.122e-3& 4.793e-3& 2.428e-3\\
$f(x)=e^{-x^2}$ & 1.94e-2& 9.378e-3& 4.568e-3&2.258e-3 &1.135e-3 &6.06e-4 &2.874e-4 &1.829e-4\\ \hline
\multicolumn{9}{c}{\cellcolor{blue!20}\tabhead{Case 2: $(B_0, B1, B_2,\sigma,\alpha) =(0,0,3,1, \frac54)$ and \ref{C2:H5} is  not valid.}}\\
\hline
\tabhead{{}}
& \tabhead{$p=2$} &\tabhead{$p=3$}& \tabhead{$p=4$} & \tabhead{$p=5$} &\tabhead{$p=6$} & \tabhead{$p=7$} & \tabhead{$p=8$}&\tabhead{$p=9$}\\
\midrule
$f(x)=x$& 2.179e-2& 1.069e-2& 5.07e-3& 2.529e-3 & 1.32e-3& 8.021e-4& 2.598e-4&3.043e-4\\
$f(x)=x^2$ &6.412e-3 &3.243e-3 &1.582e-3 &8.397e-4 &3.965e-4 &2.148e-4 &9.101e-5 &5.065e-5 \\
$f(x)=\frac1x$ & 2.208& 9.48e-1& 4.503e-1&2.215e-1 &1.1e-1 & 5.93e-2&3.68e-2 & 1.526e-2 \\
$f(x)=e^{-x^2}$ & 6.113e-3&3.07e-3 &1.5e-3 &7.5e-4 &3.65e-4 &1.868e-4 &5.625e-5 &5.439e-5  \\ \hline
\multicolumn{9}{c}{\cellcolor{violet!20}\tabhead{Case 3: $(B_0,B_1,B_2,\sigma,\alpha) =(0,0,1,1, \frac32)$ and \ref{C2:H5} is  not valid.}}\\
\hline
\tabhead{{}} & \tabhead{$p=2$} &\tabhead{$p=3$}& \tabhead{$p=4$} & \tabhead{$p=5$} &\tabhead{$p=6$} & \tabhead{$p=7$} & \tabhead{$p=8$}&\tabhead{$p=9$}\\
\midrule
$f(x)=x$& 2.3e-2& 1.219e-2&5.864e-3 &2.893e-3 & 1.255e-3& 9.507e-4& 3.13e-4& 3.14e-4\\
\cellcolor{red!25} $f(x)=x^2$ \cellcolor{red!25}  &\cellcolor{red!25}  9.408e-3 &\cellcolor{red!25}   5.302e-5 & \cellcolor{red!25}  1.749e-4 & \cellcolor{red!25}  2.956e-4 &\cellcolor{red!25}  8.41e-3 & \cellcolor{red!25}  5.95e-4 &\cellcolor{red!25}  2.574e-3 & \cellcolor{red!25}  2.52e-4  \\
$f(x)=\frac1x$ &1.084 & 1.058& 8.919e-2&4.341e-2 & 2.483e-2& 1.224e-2& 9.783e-3& 4.861e-3  \\
$f(x)=e^{-x^2}$ & 1.485e-2&8.108e-3 &4.162e-3 &2.248e-3 &1.31e-3 &1.164e-3 &3.78e-4 &2.762e-4   \\ \hline
\multicolumn{9}{c}{ \cellcolor{blue!25}  \tabhead{Case 4: $(B_0,B_1,B_2,\sigma,\alpha) = (1,1,\frac25,\frac{1}{10},3)$ and \ref{C2:H5} is not valid.}}\\
\hline
\tabhead{{}}
& \tabhead{$p=2$} &\tabhead{$p=3$}& \tabhead{$p=4$} & \tabhead{$p=5$} &\tabhead{$p=6$} & \tabhead{$p=7$} & \tabhead{$p=8$}&\tabhead{$p=9$}\\
\midrule
$f(x)=x$ & 3.741e-3& 3.292e-3& 2.103e-3& 1.364e-3& 6.915e-4 & 2.936e-4& 1.131e-4& 5.497e-5\\
$f(x)=x^2$ & 2.687e-2&1.332e-2 &7.476e-3 &4.584e-3 &2.302e-3 &1.034e-3 &4.423e-4 &2.312e-4 \\
$f(x)=\frac1x$ &4.374e-3 & 2.804e-3& 5.505e-4& 4.508e-4& 2.383e-4&8.651e-5 & 2.313e-5& 9.249e-6 \\
$f(x)=e^{-x^2}$ &5.027e-3 &2.846e-4 &2.372e-4 &2.666e-4 & 1.545e-4& 5.72e-5 &2.363e-5 &1.529e-5  \\ \hline
\multicolumn{9}{c}{\tabhead{Case 5: $(B_0, B_1, B_2,\sigma,\alpha) = (0,0, 10,\frac{1}{2},\frac98)$ and \ref{C2:H5} is valid.}}\\
\hline
\tabhead{{}}
& \tabhead{$p=2$} &\tabhead{$p=3$}& \tabhead{$p=4$} & \tabhead{$p=5$} &\tabhead{$p=6$} & \tabhead{$p=7$} & \tabhead{$p=8$}&\tabhead{$p=9$}\\
\midrule
$f(x)=x$ & 2.735e-3& 1.413e-3&7.122e-4 & 3.69e-4 & 1.873e-4&9.082e-5 &5.053e-5 &1.979e-5 \\
$f(x)=x^2$ &2.955e-5 &1.718e-5 &9.153e-6 &4.892e-6 &2.511e-6 &1.226e-6 &6.833e-7 &2.82e-7 \\
$f(x)=\frac1x$ & 1.175e+2& 4.523e+1& 2.025e+1&9.639 & 4.765&2.35 & 1.202& 5.874e-1\\
$f(x)=e^{-x^2}$ &2.952e-5 &1.715e-5 &9.128e-6 &4.867e-6 &2.486e-6 &1.201e-6 &6.583e-7 &2.571e-7 \\
\hline
\multicolumn{9}{c}{\cellcolor{red!25}\tabhead{Case 6: $(B_0, B_1,B_2,\sigma,\alpha) =(0,0,\frac1{100},\frac1{10}, \frac54)$ and \ref{C2:H5} is  not valid.}}\\ \hline
\tabhead{{}}
& \tabhead{$p=2$} &\tabhead{$p=3$}& \tabhead{$p=4$} & \tabhead{$p=5$} &\tabhead{$p=6$} & \tabhead{$p=7$} & \tabhead{$p=8$}&\tabhead{$p=9$}\\
\midrule
$f(x)=x$ & 8.51e-5& 1.756e-5& 6.374e-5& 6.616e-5 &1.943e-4 &6.087e-5 & 1.69e-5&2.31e-5 \\
$f(x)=x^2$ &1.538e-4 &4.004e-5 &1.187e-4 &1.334e-4 &4.363e-4 &2.421e-4 &3.981e-5 &4.204e-5  \\
$f(x)=\frac1x$ & 5.64e-5& 2.389e-4& 4.294e-5&2.704e-4 &1.394e-4 &2.57e-4 &1.724e-4 & 7.968e-5 \\
$f(x)=e^{-x^2}$ & 1.623e-5& 1.033e-4&1.396e-4 &1.372e-4 &3.481e-5 &1.251e-4 &7.891e-5 &7.133e-5  \\ \hline
\multicolumn{9}{c}{\tabhead{Case 7: $(B_0,B_1,B_2,\sigma,\alpha) = (0,0,\frac25,\frac{1}{10},3)$ and \ref{C2:H5} is valid.}}\\
\hline
\tabhead{{}}
& \tabhead{$p=2$} &\tabhead{$p=3$}& \tabhead{$p=4$} & \tabhead{$p=5$} &\tabhead{$p=6$} & \tabhead{$p=7$} & \tabhead{$p=8$}&\tabhead{$p=9$}\\
\midrule
$f(x)=x$ & 1.698e-2& 7.814e-3& 3.881e-3& 1.844e-3&9.058e-4 & 4.943e-4&2.266e-4 &8.054e-5\\
$f(x)=x^2$ & 2.608e-2& 1.207e-2&6.018e-3 &2.864e-3 &1.405e-3 &7.688e-4 &3.515e-4 &1.489e-4\\
$f(x)=\frac1x$ & 2.902e-2&1.323e-2 & 6.337e-3&3.156e-3 & 1.387e-3& 9.119e-4& 3.679e-4&2.059e-4 \\
$f(x)=e^{-x^2}$ & 1.436e-2& 6.571e-3&3.232e-3 &1.505e-3 &7.1e-4 &3.609e-4 &1.341e-4 &1.122e-5 \\ 
\bottomrule
\end{tabular}
}}
\caption{\label{tab:C2experiment1-sup} Observed numerical weak error $|\EE[f(X_T)] - \frac{1}{n}\sum_{i=1}^{n}f(\overline{X}_T(\omega_i, 2^{-p}))|$ for {\bf Case 1} to {\bf Case 7}.}
\end{table}

\subsection{Some comparison with other schemes} \label{sec:comparison}
One of the significant advantages of the proposed exp-ES is that it easily addresses the control of the moments of the numerical approximation (see Lemma \ref{leq:Schememoments}).
We seek now to (numerically) observe the stability of the exponential-Euler scheme. To this aim, we compare the  exp-ES with the four others following schemes proposed in the literature:
\begin{itemize}
\item Symmetrized Euler scheme (SES) (see e.g. \cite{BD15} and the reference therein), defined  by
\[\overline{X}_{t_{n+1}}=|\overline{X}_{t_{n}}-B_2\overline{X}_{t_{n}}^{2\alpha-1}\Delta t + \sigma \overline{X}_{t_{n}}^\alpha (W_{t_{n+1}}-W_{t_n})|,\]
which is the closest form of the classical Euler scheme to be applied to SDE \eqref{eq:proto_CEV_with_const}.

\item Symmetrized Milstein scheme (SMS) (see e.g. \cite{B17} and the reference therein), defined by
\[\overline{X}_{t_{n+1}}=|\overline{X}_{t_{n}}-B_2\overline{X}_{t_{n}}^{2\alpha-1}\Delta t + \sigma \overline{X}_{t_{n}}^\alpha (W_{t_{n+1}}-W_{t_n})+\alpha\sigma^2 \overline{X}_{t_{n}}^{2\alpha-1} \left((W_{t_{n+1}}-W_{t_n})^2-\Delta t\right)|.\]
\item Tamed Euler scheme  (TES, see \cite{Kloeden2}), defined  by
\[\overline{X}_{t_{n+1}}=\overline{X}_{t_{n}}-\tfrac{B_2\overline{X}_{t_{n}}^{2\alpha-1}\Delta t}{1+B_2|\overline{X}_{t_{n}}^{2\alpha-1}|\Delta t} + \sigma \overline{X}_{t_{n}}^\alpha (W_{t_{n+1}}-W_{t_n}).\]
\item Stopped tamed Euler scheme (STES, see \cite{Jentzen-b} and the reference therein) defined  by
\[\overline{X}_{t_{n+1}}=\overline{X}_{t_{n}}+\tfrac{-B_2\overline{X}_{t_{n}}^{2\alpha-1}\Delta t+ \sigma \overline{X}_{t_{n}}^\alpha (W_{t_{n+1}}-W_{t_n})}{1+\big(B_2\overline{X}_{t_{n}}^{2\alpha-1}\Delta t+ \sigma \overline{X}_{t_{n}}^\alpha (W_{t_{n+1}}-W_{t_n})\big)^2} {\bf 1}_{\big\{|\overline{X}_{t_n}|<\exp\{\sqrt{|\ln(\Delta t)|}\}\big\}}.\]
\end{itemize}

Results are shown in Figure \ref{im:Comp} and Table \ref{tab:Comp}.
Table \ref{tab:Comp}  reports on the stability of the exp-ES, in comparison with the other  schemes. In particular, we experiment some instability with the tamed schemes when $\Delta t$ is not small enough (marked in Table \ref{tab:Comp} as {-} for the missing values). We also observe abnormally large level of errors for SES ({\bf Case 1}) and SMS ({\bf Cases 1, 3, 5}) when $\Delta t$ is not small enough as well.

In terms of convergence rate, the scheme exp-ES behaves very well, in the average of the other schemes, and even better in {\bf Cases 2, 3, 5}. On the contrary, {bf Case 3} (where the explicit Euler scheme is strongly diverging) is particularly unstable for the SMS and TES, STES. The same behavior with a smaller impact is observed in {bf Case 5}.

\begin{table}[H]
{\footnotesize{
\begin{tabular}{c| l l l l l l l l}
\toprule
&\multicolumn{8}{c}{Observed weak Error with $\Delta t={2^{-p}}$ }\\
\hline
\tabhead{{\small Cases $(B_2,\sigma,\alpha$)}}& \tabhead{$p=1$}
& \tabhead{$p=2$} &\tabhead{$p=3$}& \tabhead{$p=4$} & \tabhead{$p=5$} &\tabhead{$p=6$} & \tabhead{$p=7$} & \tabhead{$p=8$}\\
\midrule
{\small Case 1  $(2,\frac{1}{10},\frac32)$}&\multicolumn{8}{c}{ }\\
\cellcolor{blue!35}{\small\bf{exp-ES}}  & 7.866e-2 & 3.402e-2 & 1.6e-2 & 7.829e-3& 3.939e-3&1.918e-3 &8.92e-4 &4.774e-4 \\
{\bf{ SES}} & 1.557e-4 & 2.866e-4 & 2.573e-4&1.328e-4 &1.747e-4 & 4.402e-4& 4.66e-4&1.602e-4 \\
{\bf{SMS}} & 9.462e-2 &2.272e-2& 3.028e-2&  \cellcolor{red!25} 1.065e+17& 1.480e-2&5.901e-3 &3.103e-3 &9.497e-4  \\
{\bf{STES}} & 6.939e-2 & 5.43e-2 & 2.969e-2& 1.518e-2&7.906e-3 &3.383e-3 &2.111e-3 &6.74e-4 \\
{\bf{ TES }} & 3.419e-2 & 2.12e-2& 1.14e-2 &6.091e-3 &3.109e-3 &1.759e-3 &7.997e-4 &3.626e-4  \\
\hline
{\small Case 2  $(3,1,\frac54)$} & & &  &  & & & & \\
\cellcolor{blue!35}{\small\bf{exp-ES}}& 4.655e-2 & 2.237e-2& 1.05e-2&5.263e-3 &2.872e-3 &1.489e-3 &2.852e-4 &2.458e-4 \\
{\bf{ SES}} &3.815e-1 & 5.788e-2&4.519e-2 &2.219e-2 &1.048e-2 &4.736e-3 &2.434e-3 &1.128e-3 \\
{\bf{ SMS}} &2.813e-1 & 7.577e-2& 3.312e-2&1.59e-2 &7.82e-3 &3.798e-3 &1.86e-3 &9.431e-4 \\
{\bf{STES }} &\cellcolor{red!25} - & \cellcolor{red!25}-&\cellcolor{red!25}- &\cellcolor{red!25}- &6.596e-3 &3.411e-3 &5.594e-4 &1.895e-3 \\
{\bf{ TES }} & \cellcolor{red!25}- &\cellcolor{red!25}- &\cellcolor{red!25}- &\cellcolor{red!25}- & 4.292e-3 & 1.177e-3& 8.831e-4&2.887e-4  \\
\hline
{\small Case 3  $(1,1,\frac32)$}&\multicolumn{8}{c}{ }\\
\cellcolor{blue!35}{\small\bf{exp-ES}}& 1.195e-2 & 6.353e-3&2.94e-3 &1.523e-3 &8.059e-4 &5.286e-4 &3.573e-6 &2.319e-5 \\
{\bf{ SES }} & 6.306e-3 & 6.306e-3& 6.384e-3& 6.571e-3& 6.308e-3& 6.416e-3&6.008e-3&5.995e-3 \\
{\bf{ SMS }} & 5.396e-2& 1.567e-1 & \cellcolor{red!25}   1.841e+6 &1.411e-1 &1.243e-1 &1.178e-1 &1.134e-1 &1.116e-1 \\
{\bf{ STES }} & \cellcolor{red!25}- &\cellcolor{red!25} -&\cellcolor{red!25}- &\cellcolor{red!25}- &\cellcolor{red!25}- &5.979e-3 & 8.053e-3& 1.758e-3\\
{\bf{ TES }} & \cellcolor{red!25}- & \cellcolor{red!25}-& \cellcolor{red!25}-&\cellcolor{red!25}- &\cellcolor{red!25}-& \cellcolor{red!25}-&\cellcolor{red!25}- &  2.715e-3\\ \hline
{\small Case 5  $(10,\tfrac{1}{2},\tfrac98)$}&\multicolumn{8}{c}{ }\\
\cellcolor{blue!35}{\small\bf{exp-ES}}& 4.980e-3 &2.739e-3&1.416e-3 &7.226e-4 &3.707e-4 &1.837e-4 &8.583e-5 &4.565e-5 \\
{\bf{ SES}} & \cellcolor{red!25}24.287 & \cellcolor{red!25}17.878& 4.936e-3& 3.262e-3& 1.746e-3& 9.092e-4&4.598e-4 &2.655e-4 \\
{\bf{ SMS}} & \cellcolor{red!25}25.272 & \cellcolor{red!25}20.77& 4.968e-3& 3.556e-3& 2.156e-3& 1.45e-3& 1.336e-3&1.654e-3 \\
{\bf{STES }} & 5.49e-1&\cellcolor{red!25}- &\cellcolor{red!25}- &2.849e-3 &1.652e-3 & 8.828e-4& 4.905e-4&2.324e-4\\
{\bf{ TES }} & \cellcolor{red!25} - & \cellcolor{red!25} -& \cellcolor{red!25} - &2.252e-3 &1.21e-3 &6.29e-4 &3.17e-4 &1.938e-4  \\
\hline
{\small Case 6  $(\frac{1}{100},\frac{1}{10},\frac54)$}&\multicolumn{8}{c}{ }\\
\cellcolor{blue!35}{\small\bf{exp-ES}}& 4.728e-3 & 5.047e-3&4.656e-3 &4.998e-3 &5.406e-3 &4.959e-3 &4.587e-3 &4.917e-3 \\
{\bf{ SES}} & 9.382e-1 & 7.76e-2& 7.229e-2& 5.053e-2& 3.962e-2&3.5e-2 &3.231e-2 &3.109e-2 \\
{\bf{ SMS}} & 2.816e-1 & 7.682e-2& 3.409e-2& 1.674e-2& 8.629e-3& 4.738e-3& 2.724e-3&1.783e-3 \\
{\bf{ STES }} & 8.355e-4 & 8.966e-4& 6.832e-4&2.909e-4 &1.913e-3 &8.137e-4 &2.142e-4 &3.546e-4 \\
{\bf{ TES }} & 1.582e-4 & 1.222e-3&1.698e-3 &3.282e-4 &7.02e-5 &8.031e-4 &9.383e-4 &9.324e-4  \\
\hline
\bottomrule
\end{tabular}
}}
\caption{\label{tab:Comp} Comparison of the weak approximation error for test function $f(x)=x$. The comparison consider the following numerical schemes: exponential Euler, Symmetrized Euler and  Milstein  schemes, Tamed and Stopped  Tamed Euler schemes. }
\end{table}

\begin{figure}[H]
\centering
\subfigure[Case 1: $(2,\tfrac{1}{10},\tfrac32)$.]{\includegraphics[width=0.49\textwidth]{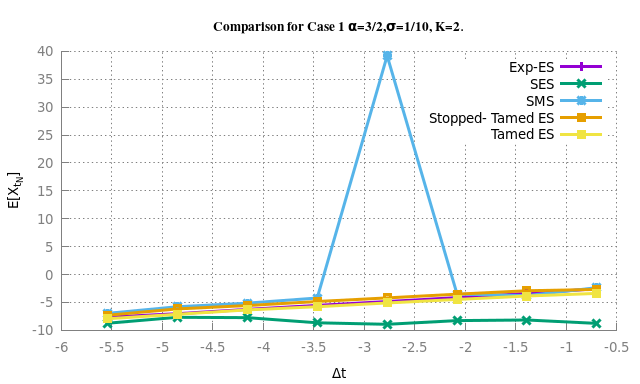}\label{im:comp1}}
\subfigure[Case 2: $(3,1,\frac54)$.]{\includegraphics[width=0.49\textwidth]{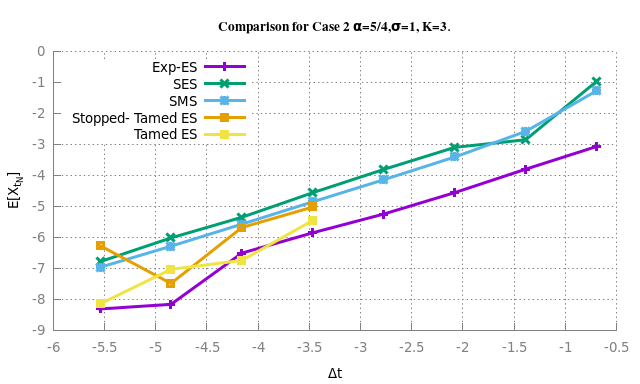}\label{im:comp3}}\\
\subfigure[Case 3 : $(1,1,\tfrac32)$.]{\includegraphics[width=0.49\textwidth]{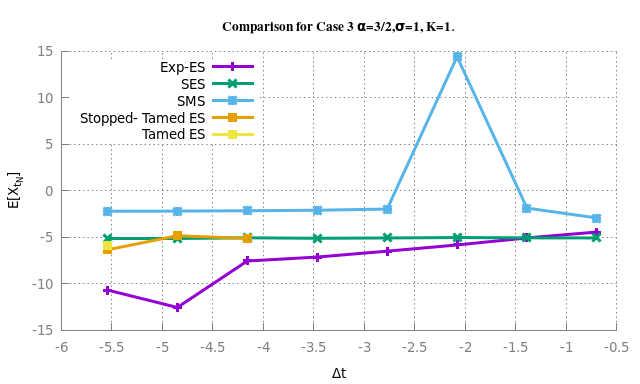}\label{im:comp5}}
\subfigure[Case 5: $(10,\tfrac{1}{2},\tfrac98)$.]{\includegraphics[width=0.49\textwidth]{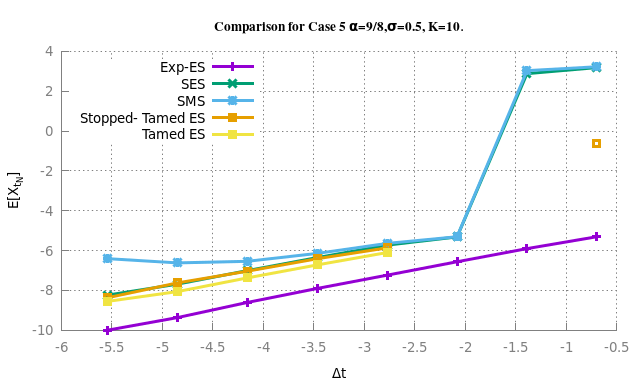}\label{im:comp2}}\\
\subfigure[Case 6: $(\tfrac{1}{100},\tfrac{1}{10},\tfrac54)$.]{\includegraphics[width=0.49\textwidth]{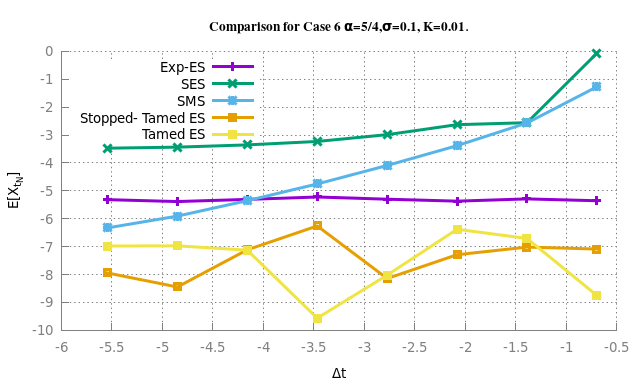}\label{im:comp4}}
\caption{\label{im:Comp} Weak approximation error for the exponential-Euler scheme in {\bf Cases 1, 2, 3, 5, 6} (in log-log scale). The weak error for the   exponential-Euler scheme (purple) is compared with the weak error for the SES (green) and the SMS (blue).}
\end{figure}


\small
\begin{thebibliography}{10}
\bibitem{Ait96}
Y.~Ait-Sahalia.
\newblock {Testing continuous-time models of the spot interest rate}.
\newblock {\em Rev. Financial Stud.}, 9:385--426, 1996.

\bibitem{Alfonsi}
A.~Alfonsi.
\newblock {On the discretization schemes for the CIR (and Bessel squared)
  processes}.
\newblock {\em Monte Carlo Methods and Applications}, 11:355--384, 2005.

\bibitem{Jentzen-a}
M.~Beccari, M.~Hutzenthaler, A.~Jentzen, R.~Kurniawan, F.~Lindner, and
  D.~Salimova.
\newblock {Strong and weak divergence of exponential and linear-implicit Euler
  approximations for SPDEs
  %stochastic partial differential equations
  with
  superlinearly growing nonlinearities}.
\newblock arxiv:1903.06066.

\bibitem{BD15}
M.~Bossy and A.~Diop.
\newblock {Weak convergence analysis of the symmetrized Euler scheme for one
  dimensional SDEs with diffusion coefficient $|x|^\alpha$,
  $\alpha\in[\frac{1}{2},1)$}.
\newblock {\em Inria Research report n$^o$ 5396, arxiv:1508.04573}, 2010.

\bibitem{B17}
M.~Bossy and H.~Olivero.
\newblock {Strong convergence of the symmetrized Milstein scheme for some
  CEV-like SDEs}.
\newblock {\em Bernoulli}, 24(3):1995--2042, 08 2018.

\bibitem{Chan}
K.~C. Chan, G.~A. Karolyi, F.~A. Longstaff, and A.~B. Sanders.
\newblock {An empirical investigation of alternative models of the short-term
  interest rate}.
\newblock {\em J. Finance}, 47:1209--1227, 1992.

\bibitem{Cha16}
J-F. Chassagneux, A.~Jacquier, and I.~Mihaylov.
\newblock {An explicit Euler scheme with strong rate of convergence for financial SDEs with non-Lipschitz coefficients}.
\newblock {\em Society for Industrial and Applied Mathematics}, 7:993--1021,
  2016.

\bibitem{CIR}
J.~Cox, J.~Ingersoll, and S.~Ross.
\newblock {A theory of the term structure of interest rates}.
\newblock {\em Econometrica}, 53:385--407, 1985.

\bibitem{Delbaen}
F.~Delbaen and H.~Shirakawa.
\newblock {A note on option pricing for the constant elasticity of variance
  Model}.
\newblock {\em Asia Pacific Financial Markets}, 9:85--99, 2002.

\bibitem{DNS11}
S.~Dereich, A.~Neuenkirch, and L.~Szpruch.
\newblock {An Euler-type method for the strong approximation of the
  Cox–Ingersoll–Ross process}.
\newblock {\em Proc. R. Soc.}, 468:1105--1115, 2012.

\bibitem{Friedman}
A.~Friedman.
\newblock {\em {Stochastic Differential Equations and Applications}}.
\newblock Volume 1. Academic Press, New York, 1975.

\bibitem{Gyongy}
I.~Gyongy.
\newblock {A note on Euler's approximations}.
\newblock {\em Potential Anal.}, 8:205--216, 1998.

\bibitem{Spr}
D.~Higham, X.~Mao, J.~Pan, and L.~Szpruch.
\newblock {Numerical simulation of a strongly nonlinear Ait-Sahalia-type
  interest rate model}.
\newblock {\em BIT Numer Math.}, 51:405--425, 2010.

\bibitem{Higham}
D.~Higham, X.~Mao, and A.~Stuart.
\newblock S{trong convergence of Euler-type methods for nonlinear stochastic
  differential equations}.
\newblock {\em SIAM Numeric Analysis}, 40:1041--1063, 2002.

\bibitem{HocOst-10}
M. Hochbruck and A. Ostermann.
\newblock Exponential integrators.
\newblock {\em Acta Numerica}, 19:209--286, 2010.

\bibitem{Jentzen-b}
M.~Hutzenthaler and A.~Jentzen.
\newblock {On a perturbation theory and on strong convergence rates for
  stochastic ordinary and partial differential equations with non-Globally
  monotone coefficients}.
\newblock {\em Ann. Probab.}, 48(1):53--93, 2020.

\bibitem{Kloeden}
M.~Hutzenthaler, A.~Jentzen, and P.~Kloeden.
\newblock {Strong and weak divergence in finite time of Euler's method for
  SDEs
%  stochastic differential equations
  with non-globally Lipschitz continuous
  coefficients}.
\newblock {\em Proc. R. Soc.}, 467:1563--1576, 2010.

\bibitem{Kloeden2}
M.~Hutzenthaler, A.~Jentzen, and P.~Kloeden.
\newblock {Strong convergence of an explicit numerical method for SDEs with
  non-globally Lipschitz continuous coefficients}.
\newblock {\em Ann. Probab.}, 22:1611--1641, 2012.

\bibitem{IkedaWatanabe}
N.~Ikeda and S.~Watanabe.
\newblock {Stochastic differential equations and diffusion processes}.
\newblock {\em North-Holland Publishing Company}, 1981.

\bibitem{KarShr-88}
I.~Karatzas and S.~Shreve.
\newblock {\em Brownian Motion and Stochastic Calculus}.
\newblock Springer-Verlag, Berlin, 1988.

\bibitem{Lamba}
H.~Lamba, J.~Mattingly, and A.~Stuart.
\newblock {An adaptive Euler-Maruyama scheme for SDEs convergence and
  stability}.
\newblock {\em IMA Numeric Analysis}, 27:479--506, 2007.

\bibitem{kerlyns}
K.~Mart\'inez.
\newblock {\em {Penalized Stochastic Optimal Control Problems for Singular
  McKean-Vlasov Dynamics and Turbulent Kinetic Energy modeling with Calibration   on Lagrangian Turbulent Flow Models}}.
\newblock PhD thesis, Doctorate in Mathematics of Valparaíso -- Consortium between Universidad de Valparaíso, Pontificia Universidad Católica de Valparaíso, Universidad Técnica Federico Santa María. 
August 2019. \bibinfo{title}{\href{https://www.researchgate.net/publication/339988242_Thesis_manuscript_Introduction_and_Chapter_II}{{Chapter II--DOI: 10.13140/RG.2.2.17285.09446}}}.

\bibitem{MilsteinTr}
G.~Milstein and M.~Tretyakov.
\newblock {Numerical integration of stochastic differential equations with
  nonglobally Lipschitz coefficients}.
\newblock {\em SIAM Numeric Analysis}, 43:1139--1154, 2005.

\bibitem{Pope-da-63}
D.~A. Pope.
\newblock An exponential method of numerical integration of ordinary
  differential equations.
\newblock {\em Communications of the ACM}, 6(8):491--493, 1963.

\bibitem{Protter-04}
P.~Protter.
\newblock {\em {Stochastic Integration and Differential Equations}}.
\newblock Second Edition. Springer-Verlag, Berlin, 2004.

\bibitem{Sabanis}
S.~Sabanis.
\newblock {Euler approximations with varying coefficients: The case of
  superlinearly growing diffusion coefficients}.
\newblock {\em Ann. Probab.}, 26:2083--2105, 2016.

\bibitem{Yan}
L.~Yan.
\newblock {The Euler scheme with irregular coefficients}.
\newblock {\em Ann. Probab.}, 30:1172--1194, 2002.

\end{thebibliography}
\end{document}